\documentclass[a4,reqno,12pt]{amsart}
\usepackage{graphicx}
\usepackage{color}
\usepackage{appendix}
\usepackage{amsmath}
\usepackage{amsfonts}
\usepackage{amssymb}
\numberwithin{equation}{section}

\providecommand{\U}[1]{\protect\rule{.1in}{.1in}}
\providecommand{\lb}{\left(}
\providecommand{\rb}{\right)}
\providecommand{\lbr}{\left\{}
\providecommand{\rbr}{\right\}}

\providecommand{\dd}{{\rm d}}
\providecommand{\be}{\begin{equation}}
\providecommand{\ee}{\end{equation}}

\providecommand{\1}{\ifmmode {1\hskip -3pt \rm{I}}}

\providecommand{\df}{\stackrel{\Delta}{=}}

\providecommand{\eqvs}{\stackrel{\sim}{=}}
\providecommand{\smof}[1]{{\mathrm o}\lb #1\rb }
\providecommand{\bigof}[1]{{\mathrm O}\lb #1\rb }
\providecommand{\Kb}{{\mathbf K}_\beta}

\providecommand{\pKb}{\partial {\mathbf K}_\beta}
\providecommand{\ugamma}{\underline{\gamma}}

\providecommand{\uempty}{\underline{\emptyset}}
\providecommand{\ucalC}{\underline{\calC}}
\providecommand{\uGamma}{\underline{\Gamma}}

\providecommand{\suptwo}[2]{\sup_{\substack{#1 \\ #2}}} % sup with 2 lines
\providecommand{\sumtwo}[2]{\sum_{\substack{#1 \\ #2}}} % sum with 2 lines
\providecommand{\abs}[1]{\left| #1\right|}

\providecommand{\calA}{\mathcal{A}}
\providecommand{\calB}{\mathcal{B}}
\providecommand{\calC}{\mathcal{C}}

\providecommand{\calH}{\mathcal{H}}

\providecommand{\calK}{\mathcal{K}}
\providecommand{\calL}{\mathcal{L}}

\providecommand{\calP}{\mathcal{P}}
\providecommand{\calQ}{\mathcal{Q}}

\providecommand{\calY}{\mathcal{Y}}

\providecommand{\frA}{\mathfrak{A}}

\providecommand{\frE}{\mathfrak{E}}

\providecommand{\frP}{\mathfrak{P}}

\providecommand{\frR}{\mathfrak{R}}
\providecommand{\frS}{\mathfrak{S}}

\providecommand{\bbE}{\mathbb{E}}

\providecommand{\bbG}{\mathbb{G}}

\providecommand{\bbI}{\mathbb{I}}

\providecommand{\bbN}{\mathbb{N}}

\providecommand{\bbP}{\mathbb{P}}

\providecommand{\bbR}{\mathbb{R}}
\providecommand{\bbS}{\mathbb{S}}

\providecommand{\bbZ}{\mathbb{Z}}

\providecommand{\sfc}{\mathsf{c}}

\providecommand{\sfe}{\mathsf{e}}

\providecommand{\sfg}{\mathsf{g}}
\providecommand{\sfh}{\mathsf{h}}

\providecommand{\sfl}{\mathsf{l}}
\providecommand{\sfm}{\mathsf{m}}
\providecommand{\sfn}{\mathsf{n}}

\providecommand{\sfr}{\mathsf{r}}
\providecommand{\sfs}{\mathsf{s}}
\providecommand{\sft}{\mathsf{t}}
\providecommand{\sfu}{\mathsf{u}}
\providecommand{\sfv}{\mathsf{v}}
\providecommand{\sfw}{\mathsf{w}}
\providecommand{\sfx}{\mathsf{x}}
\providecommand{\sfy}{\mathsf{y}}
\providecommand{\sfz}{\mathsf{z}}

\providecommand{\sfA}{\mathsf{A}}

\providecommand{\sfO}{\mathsf{O}}
\providecommand{\sfP}{\mathsf{P}}
\providecommand{\sfR}{\mathsf{R}}

\providecommand{\sfU}{\mathsf{U}}
\providecommand{\sfV}{\mathsf{V}}

\providecommand{\sfX}{\mathsf{X}}

\providecommand{\gm}[1]{\gamma^{[#1 ]}}
\providecommand{\Gm}[1]{\Gamma^{[#1 ]}}
\providecommand{\chip}{\chi^\prime}

\newtheorem{theorem}{Theorem}

\newtheorem{definition}[theorem]{Definition}

\newtheorem{lemma}[theorem]{Lemma}

\newtheorem{proposition}[theorem]{Proposition}
\newtheorem{remark}[theorem]{Remark}

\providecommand{\step}[1]{S{\small TEP}\,#1}
\providecommand{\opt}[1]{C{\small ASE}\,#1}

\providecommand{\lxu}{\ell_\sfx^{{ \sfu}}}
\providecommand{\lxv}{\ell_\sfx^{{ \sfv}}}

\begin{document}

\title[{Interaction versus Repulsion for polymers}]{ {Interaction versus {entropic}  repulsion for low temperature Ising polymers}}

\begin{abstract}
Contours associated to many interesting low-temperature statistical mechanics models
(2D Ising model, (2+1)D SOS interface model, etc) can be described as
self-interacting and self-avoiding walks on $\mathbb Z^2$. When the model is defined in
a finite box, the presence of the boundary induces an
interaction, that can turn out to be attractive, between the contour
and the boundary of the box. On the other hand, the
contour cannot cross the boundary, so it feels entropic repulsion from
it. In various situations
of interest \cite{CLMST1,CLMST2,CMTrep,SENYADIMA}, a crucial technical problem is to
prove that entropic repulsion prevails over the pinning
interaction: in particular, the contour-boundary interaction should not
modify significantly the contour partition function and the related
surface tension should be unchanged. Here we prove that this is indeed the case, at least at
sufficiently low temperature, in a quite general framework that
applies in particular to the models of interest mentioned above.
\end{abstract}
\author{Dmitry Ioffe}
\address{Faculty of IE\&M, Technion, Haifa 32000, Israel}
\email{ieioffe@ie.technion.ac.il}
\thanks{The research of D.I. was supported by Israeli Science Foundation grant 817/09 and
by the Meitner Humboldt Award. The hospitality of Bonn University during the academic year 2012-13
is gratefully acknowledged.}
\author{Senya Shlosman}
\address{Aix Marseille Universit\'e, Universit\'e de Toulon,
CNRS, CPT UMR 7332, 13288, Marseille, France;
Inst. of the Information Transmission Problems,
RAS, Moscow, Russia} \email{senya.shlosman@univ-amu.fr}

\author{Fabio Lucio Toninelli}
\address{Universit\'e de Lyon, CNRS and Institut Camille Jordan, Universit\'e Lyon 1,
    43 bd
 du 11 novembre 1918, 69622 Villeurbanne, France}
\email{toninelli@math.univ-lyon1.fr}

\date{\today}
\maketitle

\section{Introduction}
Two-dimensional statistical mechanics models are  often
conveniently rewritten in terms of contour ensembles: for instance,
for the Ising model contours are curves, separating
 $+$ spins from $-$ spins, while
for the $(2+1)$-dimensional SOS interface model, contours
correspond to level lines of the interface.
See Section
\ref{sec:appli} for various examples. At low temperature
$1/\beta$, the ensemble of non-intersecting contours $\gamma$ is defined by the weight
%of the type
\be
\label{eq:IP-weights}
w\left(  \gamma\right)  =\exp\Bigl[  -\beta\left\vert \gamma\right\vert
+\sum_{\calC:\calC\cap\Delta_\gamma\neq\varnothing}\Phi
\left(  \calC,\Delta_\gamma\cap \calC\right)
\Bigr]  ,
\ee
(the notation $\calC\cap\Delta_\gamma\neq\varnothing$ essentially means that the sum is taken over all sets $\calC\subset \bbZ^2$ that
intersect the contour $\gamma$,  see  Section \ref{sec:themain} for more details).  The first term $\beta|\gamma|$ tends to make the contour as short as
possible, while the ``decoration term'' containing $\Phi$ can be seen
as a self-interaction of the path. This self-interaction is small for
$\beta$ large (see \eqref{07}) but on the other hand it is non-local.
In  this ensemble
%the ensemble \eqref{06},
a long contour typically has a
Brownian behavior under diffusive rescaling.
When the contour is close to the boundary of the system, as discussed
in Section \ref{sec:appli}, the
potentials $\Phi$ are modified to some $\tilde \Phi$ (that still
satisfy \eqref{07} with the same value for $\chi$) and this results in an effective
interaction $\tilde \Phi-\Phi$ with the boundary, that may well turn out to
be attractive (there is no way to control
apriori its sign). On the other hand, since the
contour cannot cross the boundary of the system, it feels entropic
repulsion from it and
it is not obvious  whether pinning or repulsion prevails.
%Let us mention that
 This issue turns out to be one of the main
technical difficulties in recent studies of fluctuations of low-temperature
discrete interface models \cite{CLMST1,CLMST2,SENYADIMA,CMTrep}.
Its treatment in the book \cite{DKS} contains a mistake. Until
now, this difficulty has been bypassed via model-dependent tricks --
for example, via FKG inequalities in \cite{CLMST1,CLMST2,CMTrep},
but we feel that a more general solution is called for. See Appendix  A
 below for a simple patch for \cite{DKS}.

A well known and simpler problem \cite{GBbook}, that essentially corresponds to the
situation where
$\gamma$ is a directed walk and the potentials $\Phi$ act only at zero
distance, can be
formulated as follows: let $\gamma=(\gamma_n)_{n\ge0}$ be a centered random walk on
$\bbZ$ with $\gamma_0=0$, conditioned to be non-negative. Let us bias its law by the exponential of the
number of returns to zero, times some positive parameter $\epsilon$.
Then, it is known that there exists a critical $\epsilon_c>0$ such
that for $\epsilon>\epsilon_c$ the walk is positively recurrent while
for $\epsilon<\epsilon_c$ it is transient. In some particular cases,
one can sharply identify \cite{Yoshida} the critical point, which
turns out to depend crucially on the variance of the random walk step
($\epsilon_c$ tends to zero when the variance goes to zero). The
problem exposed above boils down essentially to deciding whether the
interaction $\tilde\Phi-\Phi$ corresponds to an $\epsilon$ that is
below or above the critical threshold.

Our main result here (Theorem \ref{NoPin}) is that, for $\beta$ large, the pinning
interaction $\tilde \Phi-\Phi$,
 with $\tilde \Phi$ and $\Phi$
satisfying \eqref{07} for $\chi>\frac{1}{2}$,
is not sufficient to pin the contour to
the boundary, and the contour behaves essentially as if only the
entropic repulsion were present (more precisely, the ratio of
partition functions of the models with and without pinning interaction
is uniformly bounded, but more information can be deduced on the
similarity between the contour laws themselves,
see discussion in Section \ref{sec:appli}).

We would like to emphasize that there is a subtle point here. It is
true that for $\beta $ large the pinning potential becomes exponentially small
(because of \eqref{07}). However, in this regime the variance of
the contour steps in the direction perpendicular to the wall is
exponentially small as well! (It is due to the $\beta|\gamma|$ term in \eqref{06}). As we mentioned, in the directed walk case it is
known that $\epsilon_c$ scales to zero with the variance, so there is
really a non-trivial competition to be considered in the $\beta$ large limit.

In fact, if the self-interaction $ \Phi$ is
just a bit stronger -- for example, it still satisfies \eqref{07}, but
with a smaller $\chi\le{\frac{1}{2}}$ -- then the pinning can
happen for some potential modification $\tilde \Phi-\Phi$, so our result is quite sharp.  To see
this, consider the situation where the endpoints of the path are
$(0,0)$ and $(L,0)$ and contours $\gamma$ are constrained in the upper
half-plane, i.e. the system boundary is the horizontal line
$y=0$. Assume also that $\Phi$ satisfies \eqref{07} with $\chi=1/2$
and that $\tilde\Phi(\calC)- \Phi(\calC)$ vanishes except when
$\calC=\{x\}$ with $x$ a lattice site touching both the contour and
the line $y=0$, in which case we put \[\tilde\Phi(\{x\})-
\Phi(\{x\})=M \exp(-\beta)\] (this is compatible with $\chi=1/2$,
since ${\rm diam}_\infty(\{x\})=1$ in \eqref{07}). It is known from
\cite[Chapter 4]{DKS} that
\[
 \sum_{\gamma}w(\gamma)\le \frac{C(\beta)}{\sqrt L}e^{- L(\beta+O(e^{-\beta}))}
\]
for $\beta$ large.
On the other hand, for the ensemble with modified potential $\tilde\Phi$ we can lower bound the sum $\sum_{\gamma}\tilde w(\gamma)$ by keeping
only the configuration $\gamma$ that joins the two endpoints with a straight segment of length $L$. Using the decay properties of $\Phi$ and standard estimates from \cite{DKS} we find then
\[
 \sum_{\gamma}\tilde w(\gamma)\ge e^{-L (\beta+O(e^{-\beta})+M e^{-\beta})}.
\]
If $M>0$ is  chosen sufficiently large, we see that the partition function
of the modified ensemble is exponentially larger than the original one, i.e. pinning prevails.

We will prove  Theorem \ref{NoPin} in a rather general context, i.e., we will
only assume some symmetry and decay properties of the potentials
$\Phi$ (that are verified for the various examples mentioned in
Section \ref{sec:appli}) but we will avoid using any of the special
features of these models (FKG inequalities, etc). In the directed walk
case, the problem can be easily solved via renewal theory, since the
set of return times to zero forms a renewal sequence. This is not the
case for the set of contour/boundary contacts, due to backtracks and
self-interactions of the contour. Another new difficulty is that the
pinning potential $\tilde\Phi-\Phi$, while weak, has infinite range,
so contour-wall interactions occur irrespective of their mutual
distance (of course, the strength decays fast with the distance).
The basic idea of the proof is to identify
a suitable effective random walk structure related to the contour.
{
Such approach was worked out in various disguises
in the framework of the Ornstein-Zernike theory \cite{CIV03,CIL,IV08}.
}
Once this is done, a crucial
role is played by an identity of Alili and Doney \cite{AD}.
This identity relates two quantities:

\noindent -- the probability for a one-dimensional random
walk to go from $x>0$ to $y>0$ in time $T$, conditionally on staying
positive in between,

\noindent --  the number of ladder heights of this random walk up to
time $T$,
%\noindent
see \eqref{eq:AD-expBound}.

{An  adjustment of the above approach
in the context of  effective random walk decomposition of sub-critical percolation
clusters (at a fixed value of $p<p_c$) was worked out in \cite{CIL}.
In the latter case, however, the interaction between different clusters could be bypassed, and as a
result, it was not necessary to investigate
its competition  with the entropic repulsion between
effective random walks.

One of the main thrusts
of this work is to derive methods which, in a situation
when interaction may be attractive, enable to control  degeneracy
of variance versus
degeneracy of pinning as $\beta$ becomes large.}
In principle our approach should apply for more complicated
geometries of open contours $\gamma$ in
\eqref{eq:IP-weights} and, accordingly, for more complicated
 energy  functions than just $\abs{\gamma}$.
For instance it should apply for low temperature two dimensional Blume-Capel
model in the regime when there are two stable ordered phases \cite{HK03}
What is
important is an intrinsic renewal structure of $\gamma$ which gives rise
to an effective random walk decomposition with exponentially decaying
distribution of steps.
The main  simplifying feature of
low temperature
Ising Polymers is that  there are only four basic steps (see
Figure~\ref{fig:Animals})  one needs to consider in order to sort
 out the pinning issue for a general class of interactions subject
 to Assumptions (P1)-(P3) below. As a result the covariance
 structure of the effective walk and, accordingly, the competition between
 pinning and entropic repulsion can be quantified in somewhat explicit
 terms.

\section{The Main Result}

\subsection{The contour ensemble}
\label{sec:themain}

The interface $\gamma$ is an open contour: it is a connected
collection $e_{1},\dots,e_{k}$ of bonds of the dual lattice $\mathbb{Z}_{\ast
}^{2}=\mathbb{Z}^{2}+(1/2,1/2)$, connecting two  points
$a\neq b\in \mathbb{Z}_{\ast
}^{2}$  (we write $\gamma:a\mapsto b$) such that:
\begin{enumerate}
\item $e_{i}\ne e_{j}$ for every $i \ne j$;
\item for every $i$, $e_{i}$ and $e_{i+1}$ have a common vertex in
$\mathbb{Z}^{2}_{*}$;
\item if four bonds $e_{i}, e_{i+1}$ and $e_{j}, e_{j+1}$, $i \ne j$ meet
at some $x\in\mathbb{Z}^{2}_{*}$, then $e_{i}, e_{i+1}$ are on the same side
of the line across $x$ with slope $+1$ (and the same holds for $e_{j},
e_{j+1}$)
\end{enumerate}
The third condition corresponds to the usual ``south-west splitting rule''
that is commonly adopted for Ising-type contours \cite{DKS}. Given a contour
$\gamma$, we let $\Delta_\gamma$ denote the set of sites in $\mathbb{Z}^{2}$
that are either at distance $1/2$ from $\gamma$ or at distance $1/\sqrt2$ from
it, in the south-west or north-east direction, \cite{DKS}. Also
we let $\left\vert \gamma\right\vert $ denote the number of bonds in $\gamma$.

Let $\calC\subset\mathbb{Z}^{2}$ be a finite subset. In what follows we will
identify $\calC$ with the union $\cup_{x\in \calC}S_{x}$ of closed unit squares
$S_{x}\subset\mathbb{R}^{2}$ centered at $x.$ If $\calC$ is connected, then we
denote by ${\rm diam}_\infty(\calC)$ its diameter  in the
$\|\cdot\|_\infty$-norm; if $\calC$ is not connected, then by
convention we set ${\rm diam}_\infty(\calC)=\infty.$
Note that, with
our conventions, if $\calC$ is a single point $x\in \bbZ^2$, then
${\rm diam}_\infty(\calC)=1$.

To every pair $\gamma,\calC$ with $\gamma$ an open contour and finite
$\calC\subset\mathbb{Z}^{2}$, we assign a function
(or \emph{potential}) $\Phi(\calC;\gamma)$ which we
assume satisfy:
\begin{enumerate}
\item [(P1)] Locality: $\Phi$ depends on $\gamma$ only through $\calC\cap \Delta_\gamma$:
  \begin{eqnarray}
    \label{eq:5}
\Phi(\calC;\gamma)=\Phi\left(  \calC,\Delta_\gamma\cap \calC\right)
  \end{eqnarray}
\item  [(P2)] Decay: there exist some $\chi>0$,
%$\beta_0>0$
such that for all $\beta$ sufficiently large,
\begin{equation}
\left\vert \Phi\left(  \calC,\Delta_\gamma\cap \calC\right)  \right\vert \leq
\exp\left\{  -\chi
%\left(
{
\beta
}
%-\beta_{0}\right)
({\rm diam}_\infty(\calC)+1)
\right\} . \label{07}%
\end{equation}
\item [(P3)] Symmetry: $\Phi$ possesses
translational symmetries of $\bbZ^2$,
i.e. that $\Phi\left(  \calC,\Delta_\gamma\cap \calC\right) $ is
unchanged
if both $\calC$ and $\gamma$ are translated by some vector $\sfu$.
In addition we assume that the surface tension
  $\tau_\beta(\sfx)$  which is defined
below possesses the full set of
discrete symmetries (rotations by a multiple of $\pi/2$ and
reflections with respect to axis and diagonal directions) of
$\bbZ^2$.

\end{enumerate}
 The polymer weight associated
to a contour $\gamma$ is defined as
\begin{equation}
w\left(  \gamma\right)  =\exp\Bigl[  -\beta\left\vert \gamma\right\vert
+\sum_{\calC:\calC\cap\Delta_\gamma\neq\varnothing}\Phi\left(  \calC,\Delta_\gamma\cap \calC\right)
\Bigr]  , \label{06}%
\end{equation}
where the sum goes over all finite connected subsets $\calC \subset\mathbb{Z}^{2}$.

\subsection{The modified potential landscape}

\label{sec:modi}
We use notation $\sf0_* = (1/2 , 1/2 )$ for the origin of the
dual lattice $\bbZ_*^2$.
For a unit vector $\sfn$ define the half-plane
{
$\calH_{+ ,\sfn} = \lbr
\sfx~:~ \lb \sfx - \sf0_*\rb \cdot\sfn  \geq  0\rbr$.
}
For $\sfu\in \bbZ_*^2 \cap \calH_{+ ,\sfn}$ we use $\dd_\sfn(\sfu )\in\bbN$ for
the distance in the
$\|\cdot\|_\infty$-norm from $\sfu$ to $\calH_{+,\sfn}^{\sfc}\cap\bbZ_*^2$.
 Define $\calB_{+ ,\sfn}  = \lbr \sfu \in Z_*^2\cap \calH_{+ ,\sfn} ~: ~ \dd_\sfn(\sfu ) = 1 \rbr$,
that is  $\calB_{+ ,\sfn}$ is a lattice approximation of
the boundary $\partial\calH_{+ , \sfn}$.

The modified
polymer weight $\tilde{w}\left( \gamma\right) $ is defined by the
formula $\left( \ref{06}\right) ,$ with potential $\Phi$ replaced by
some (not necessarily translation %or rotation
invariant) potential $\tilde{\Phi},$
such that
\begin{itemize}
\item $\tilde{\Phi}\left(  \calC,\calC\cap\Delta_\gamma\right)  =\Phi\left(  \calC,\calC\cap
\Delta_\gamma\right)  $ if $\calC$ is contained in $\calH_{+ ,\sfn}$;

\item $\tilde{\Phi} \left(  \calC,\calC\cap\Delta_\gamma\right)  $ satisfies $\left(
\ref{07}\right)  $ for all $\calC$. % , possibly with $\beta_{0}$ replaced by some
% other constant $\tilde\beta_{0}$.
\end{itemize}
{   Note that if} $\tilde\Phi>\Phi$, the modification of the potentials can
introduce an attractive interaction with
the line $\calB_{+ ,\sfn}$.
Nevertheless, our main result says that the model
with modified weights {and} with the restriction that $\gamma$ stays in
$\calH_{+ ,\sfn}$ (we write $\gamma\in\calP_{+ ,\sfn}$)
has the same surface
tension as the original one.

\begin{theorem}
\label{NoPin} Let $\tilde{\Phi}$, $\Phi$ be as above. Assume that
$\chi> {\frac{1}{2}
}
$ in \eqref{07}

and let $\beta$ be large enough.
For all
{
$\arg( \sfn )\in [-\frac{\pi}{4} , \frac{3\pi}{4}]$,
}
the following two surface tensions
coincide:%
\begin{eqnarray}
\tau(\beta,\sfn)=-\lim_{N\rightarrow\infty}\frac{1}{\beta d_N}\ln\Big(
\sum_{\gamma:0_* \mapsto \sfx_N}w(  \gamma)  \Big)  =-\lim_{N\rightarrow\infty}%
\frac{1}{\beta d_N}\ln\Big(  \sumtwo{\gamma: 0_* \mapsto \sfx_N}{ \gamma\in\calP_{+ ,\sfn}}\tilde{w}(
\gamma)  \Big)   \label{eq:2tau}%
\end{eqnarray}
where $\sfx_N$ is  {any}   sequence of points in $\calB_{+ ,\sfn}$ whose
Euclidean distance $d_N$ from the origin diverges.
{In words, the (possible extra attraction) $\tilde{\Phi}$ can not
produce the pinning of $\gamma$ to the wall $\partial\calH_{+ , \sfn}$.}

A stronger result holds: there exist constants $c_1(\beta), c_2 (\beta
)$ such that,
for $\beta$ large enough,
\begin{equation}
\label{eq:1}
c_1(\beta) \sumtwo{\gamma:\sfx\mapsto \sfy}{\gamma\in\calP_{+ ,\sfn}}
w\left(  \gamma\right)   \leq
\sumtwo{\gamma:\sfx\mapsto \sfy}{\gamma\in\calP_{+ ,\sfn}}\tilde{w}\left(
  \gamma\right) \leq
c_2(\beta)
 \sumtwo{\gamma:\sfx\mapsto \sfy}{\gamma\in\calP_{+ ,\sfn}}
w\left(  \gamma\right)
\end{equation}
uniformly for all $\sfx,\sfy\in \calB_{+ ,\sfn}$.
\end{theorem}

\subsection{Examples,  applications and perspectives}

\label{sec:appli}
Here we give some applications of our main result and mention some future generalizations. One of the main
points here is to emphasize that contour ensembles with ``modified
potential landscape'' as in
Section \ref{sec:modi} arise quite naturally in low-temperature statistical
mechanics, without any need to introduce the ``landscape
modification'' by hand.  Since this section serves mainly as a motivation, we will
skip technical details and concentrate on the main ideas.

Consider the two-dimensional Ising model at low temperature
$\beta>\beta_c$, in the upper half-plane $\calH_+$. Put $+$ boundary
conditions on the horizontal line $y=0$, except along the segment
joining  $A=(0,0)$ to $B=(L,0)$, where the boundary condition is
$-$. Then, there is a unique open contour $\gamma$, joining $A$ to
$B$ and contained in $\calH_+$, separating $+$ from $-$ spins. For
$\beta $ sufficiently large the weight of $\gamma$
is proportional to \cite{DKS}
\begin{eqnarray}
  \label{eq:6}
\tilde w(\gamma)=
\exp\Big(-\beta|\gamma|+\tilde \Psi(\gamma)\Big):=
\exp\Big(-\beta|\gamma|+\sum_{\calC: \calC\cap
  \Delta_\gamma\ne\varnothing} \tilde  \Phi(\calC,\Delta_\gamma\cap\calC)\Big) {\bf 1}_{\gamma\subset \calH_+}
\end{eqnarray}
where
\[
 \tilde \Phi(\calC,\Delta_\gamma\cap\calC)=
\Phi(\calC,\Delta_\gamma\cap\calC){\bf 1}_{\calC\subset\calH^+}
\]
and
the potentials $\Phi$ satisfy properties (P1)-(P3) of Section
\ref{sec:themain}, in particular with $\chi=2$ in \eqref{07}.
Actually, for the specific case of the nearest-neighbor Ising model
$\Phi(\calC,\Delta_\gamma\cap\calC)$ depends only on the first
argument.

The contour ensemble with weights $\tilde w(\gamma)$ differs
from the one with weights
\[
w(\gamma)=\exp\Big(-\beta|\gamma|+ \Psi(\gamma)\Big):=\exp\Big(-\beta|\gamma|+\sum_{\calC: \calC\cap
  \Delta_\gamma\ne\varnothing} \Phi(\calC,\Delta_\gamma\cap\calC)\Big) {\bf 1}_{\gamma\subset \calH_+}
\]
in that the potentials $\Phi$ with $\calC$ intersecting the lower
half-plane are missing: since the potentials have no definite sign,
this might result in an effective attractive pinning interaction with the boundary. If
this pinning effect prevailed, the partition function $\tilde
Z_L(\beta)$ associated to
the ensemble $\tilde w(\gamma)$ would be exponentially (in $L$) larger
than the partition function $Z_L(\beta)$ associated to $w(\gamma)$,
which itself is known to
behave like $\approx \exp(-\beta L \tau_\beta)$, with $\tau_\beta$ the
surface tension in the horizontal direction. Our Theorem
\ref{NoPin} shows that this does not happen (at least for $\beta$ large),
i.e. the surface tension is not changed by the presence of the system
boundary and actually the ratio of partition functions is bounded.

This implies that the laws $\tilde P$ and
$P$, associated to ensembles $\tilde w$ and $w$, are equivalent, in
the sense that an event $A$ that has small probability (for $L$ large)
w.r.t one of them has also small probability w.r.t. the other.
Indeed,
one has
\begin{eqnarray}
  \label{eq:7}
  \tilde P(A)=\frac{E(A; e^{\tilde \Psi(\gamma)-\Psi(\gamma)})}{E( e^{\tilde \Psi(\gamma)-\Psi(\gamma)})}.
\end{eqnarray}
The denominator is just the ratio of partition functions $\tilde
Z_L(\beta)/Z_L(\beta)$ and is bounded above and below by constants. As
for the numerator, via Cauchy-Schwartz it is upper bounded by
\[
\sqrt{P(A)}\sqrt{E(e^{2(\tilde \Psi(\gamma)-\Psi(\gamma))})}.
\]
The second expectation is bounded by a constant again thanks to
Theorem \ref{NoPin} (the factor $2$ just implies that the landscape
modification is a bit different in this case), so if $P(A)$ is
small also $\tilde P(A)$ is. The other bound is obtained similarly.

For the nearest-neighbor Ising model, some results of this kind may be
derived also from the exact solution \cite{MW}. A very different situation
occurs if one adds a boundary magnetic field which may beat entropic repulsion
or even attract far away contours \cite{PV99}. Note that the results of the
latter paper go well beyond exact solutions.

In our next example
(SOS model),
 no exact solution is available and our Theorem \ref{NoPin} seems
unavoidable, though in some cases FKG inequalities allow to bypass the
interaction-versus-repulsion problem \cite{CLMST1,CLMST2,CMTrep}.
The $(2+1)$-dimensional SOS model in a domain $\Lambda\subset \bbZ^2$ is
defined
through the collection of heights $\eta_x\in \bbZ, x\in \Lambda$ and
the Hamiltonian is the sum of the absolute value of the height
gradients between nearest-neighboring heights. If again
we take the
model in the upper half-plane, with boundary condition $\eta_x=0$ on
the horizontal line $y=0$ except along the segment from $A$ to $B$,
where heights are fixed to $\eta_x=1$, there exists a unique open contour
$\gamma$ joining $A$ to $B$, such that heights just below $\gamma $
are at least $1$ and just above $\gamma$ they are at most $0$.
Again, it is proven in \cite[Appendix 1]{CLMST1} that, for $\beta$
large, the distribution of $\gamma$ has weights of the form
\eqref{eq:6} and the results mentioned for the Ising model hold in
this case too.

The works \cite{CLMST1,CLMST2}
considered the SOS model in a $L\times L$ square box $\Lambda$, with
hard-wall constraint: $\eta_x\ge 0$ for every $x\in \Lambda$. Along the
route to prove results like dynamical metastability or laws of large
numbers and  cube-root
equilibrium fluctuations of the macroscopic level lines, one of the main technical
problems that was encountered there boiled down  to prove that the ratio of
partition functions $\tilde Z_L(\beta)/Z_L(\beta)$ introduced above is not
exponentially large, which is given directly by our present Theorem \ref{NoPin}. In
\cite{CLMST1,CLMST2}, instead, the problem had to be avoided via a
rather involved chain of monotonicity arguments that are not robust, since
they rely on the FKG inequalities satisfied by the SOS model.

Another problem encountered in \cite{CLMST1,CLMST2,CMTrep} was the
following: the various level lines of the SOS model at different heights
interact among themselves, in a way  very similar to how the contour
$\gamma$ of Section \ref{sec:modi} interacts with the line $\calB_{+,\sfn}$.
On large scales this mutual interaction should be negligible with respect to the
entropic repulsion (contours cannot cross) and the contours should not stick together.
Again, in  \cite{CLMST1,CLMST2,CMTrep} this problem was avoided via a
complicated monotonicity argument, while  the techniques developed here
could be generalized to prove directly the absence of pinning between
two or several interacting SOS contours, in analogy with Theorem
\ref{NoPin}. 
 We believe that the same type of ``no-pinning''
results will be instrumental in going beyond the results of
\cite{CLMST2} (where the contour fluctuations are proven to be of
order $L^{1/3}$) and to obtain the full scaling limit (of Airy diffusion
type) of ensemble of SOS level lines in presence of hard wall.
Scaling limits to Airy (or Ferrari-Spohn \cite{FS05})
diffusions
were recently
derived in the context of (directed) random walk bridges under rather
general tilted area constraints \cite{ISV14}.

Closely related is the model of facet formation \cite{SENYADIMA},
which is a combination of $(2+1)$-SOS interface with a high and low
density Bernoulli bulk fields (of particles) above and below it.
The system is modulated by the canonical constraint
$N^3 + aN^2$
on the total number of
particles, where $N$ is the linear size of the system.
As the parameter $a$ grows the system undergoes a sequence of first order
transitions in terms of number of macroscopic facets. Facets are
SOS-contours which interact exactly as it was described above, and
``no-pinning'' results become imperative for an analysis of the model
both on the level of thermodynamics and on the more refined level of
fluctuations. On the level of thermodynamics limiting facet
shapes look like a stack of optimal Wulff TV-shapes (flat edges connected
by portions of Wulff shapes - see \cite{SS96} for the
corresponding construction for the constrained 2D Ising model).
 On the level of fluctuations
one expects scaling limits to Airy diffusions
for portions of interfaces along flat edges.

\subsection{Organization of the paper.}
Below are brief guidelines for reading the paper.
\smallskip

\noindent
{\bf Section~\ref{sec:Reformulation}.}
In general, expansion of cluster weights $\Phi$
in \eqref{06} leads to summands
of both positive and negative sign. In Section~\ref{sec:Reformulation}
we rewrite weights in such a way that all terms in the low temperature
expansion become non-negative. This sets up the stage for a probabilistic
analysis of ensembles of decorated contours (with weights \eqref{eq:qweight} and
\eqref{eq:qweightn} without and, respectively, with interactions with
the wall).
In this reformulation our main result Theorem~\ref{NoPin} becomes
 Theorem~\ref{thm:nopinning}. For the rest of the paper we shall focus
 on proving the latter.

 The relation between induced free and pinned weights of
open contours is formulated in the two-sided bound \eqref{eq:path-A-weights}.
Accordingly, the relation between  free and pinned partition functions  of
 ensembles of open contours appears in the crucial
 (albeit crude) two-sided bound \eqref{eq:q+-crude}.
In the sequel we shall work on the level of resolution suggested by
 latter inequalities.
 \smallskip

 \noindent{\bf Section~\ref{sec:Irreducible}.}
 Irreducible decomposition
  \eqref{eq:irreducible}
 of decorated contours is developed
 in Section~\ref{sec:Irreducible}.
 Since weights of decorations become
 exponentially small as $\beta\to\infty$, this decomposition and
 its properties (most importantly the mass-gap estimate
 \eqref{eq:exp-tails})
 are  inherited  from irreducible decomposition of ensembles of ``naked''
 open contours with weights ${\rm e}^{-\beta\abs{\gamma}}$.
 The output of the  irreducible decomposition is formulated in Theorem~\ref{thm:OZ}.
 Renewal structures we analyze are generated by probability distributions
  \eqref{eq:OZ-weights} on the alphabet  of irreducible animals.
  Mass-gap estimate \eqref{eq:exp-tails} and the upper bound in
  \eqref{eq:q+-crude}   enable a reformulation
  of the upper bound in Theorem~\ref{thm:nopinning} as \eqref{eq:Target1}.
  \smallskip

  \noindent{\bf Section~\ref{sec:ERW}.}
 Irreducible decomposition of decorated contours gives rise
 to an effective random walk, which is
 introduced in Section~\ref{sec:ERW}. Contours $\gamma\subset \calH_{+, \sfn}$
 correspond to effective walk which stays above the wall. On the other hand,
 the constraint of effective random walk to stay positive is less restrictive
 than $\gamma\subset \calH_{+, \sfn}$, and in order to control the probability
 of the latter one needs to show that effective walks are sufficiently repelled
 from the wall. The main facts we need to prove about effective random walks are
 collected in Theorem~\ref{thm:TwoBounds}. In the end of Section~\ref{sec:ERW}
 we explain how (A) and (B) of Theorem~\ref{thm:TwoBounds}
 imply our target upper bound \eqref{eq:Target1} and hence the upper bound of
 Theorem~\ref{thm:nopinning}.
 \smallskip

 \noindent
 {\bf Section~\ref{sec:Thm-TB}}  is devoted to the proof of Theorem~\ref{thm:TwoBounds}.
 The arguments are based on Proposition~\ref{lem:Plus-bounds} and
 Proposition~\ref{prop:P-plus-bounds}, whose proof is relegated to
 Section~\ref{sec:provaprop10}.\newline
 The lower bound of Theorem~\ref{thm:nopinning} is
 established in Subsection~\ref{sec:lb} together with
 (B) of Theorem~\ref{thm:TwoBounds}. These are statements about entropic
 repulsion of
 effective walks from the wall $\calB_{+, \sfn}$.
 In order to prove Part (A) we encode
 the interaction with the wall as the recursion relation \eqref{eq:Recursion}
 for the quantity $\rho_\delta$ which is defined in \eqref{eq:rho-delta} and
 which appears in (A) of Theorem~\ref{thm:TwoBounds}. This recursion is rewritten
 as $\rho_\delta \leq a_\delta +b_\delta \rho_\delta$ in \eqref{eq:Recursion1},
 and (A) of Theorem~\ref{thm:TwoBounds} follows from
 Proposition~\ref{lem:Plus-bounds}.
 \newline
 Proving  Proposition~\ref{lem:Plus-bounds} and
 Proposition~\ref{prop:P-plus-bounds} are the only remaining tasks
 after completion of
  Section~\ref{sec:Thm-TB}.
 \smallskip

 \noindent
 {\bf Section~\ref{sec:A-D-RW}.}
Proofs of Proposition~\ref{lem:Plus-bounds} and of
Proposition~\ref{prop:P-plus-bounds}
are  heavily based on
 fluctuation and Alili-Doney type estimates on the effective random walk
 which are derived in Section~\ref{sec:A-D-RW}. Sharp asymptotics for the
 effective random walk are formulated in \eqref{eq:LB-P} of
 Proposition~\ref{prop:LB-P}.
 The quantity $b_{\sfx}$ (or later $b_\epsilon$) is subject to asymptotic
 relations of Proposition~\ref{prop:aep-bep}. Note that \eqref{eq:LB-P} are
 quite different from  usual Gaussian asymptotics.  Rather they appear as a mixture of
  Gaussian and Poissonian asymptotics. The corresponding decomposition of
  the effective random walk
  \eqref{eq:Rk-repr}
  is described in Subsection~\ref{sub:Rk}, which
  ends with the proof of Proposition~\ref{prop:LB-P}.
  This is one of two
  places where we make use of a particularly simple structure of
  open contours in models of Ising Polymers - at low temperatures the Poissonian
  part ($\sum\xi_i \sfU_i$ in \eqref{eq:Rk-repr}) is just a random staircase
  with two possible steps: right and up.

  As in \cite{CIL},   asymptotics of effective random walks constrained to
  stay above the wall
  (Subsection~\ref{sub:AD}) are based on Alili-Doney type identities
  \eqref{eq:AD-expBound}. However, one needs to deal with
  in general
  non-lattice directions of the wall and, most importantly, with
  degeneracies and non-Gaussian (on short scales) behaviour of the effective
  walks. These issues are addressed in Lemma~\ref{lem:Np-k},
  Lemma~\ref{lem:Nplus} and Lemma~\ref{lem:Nminus}. The latter Lemmas feature
  upper bounds on the expected number of ladder heights, and here we make the
  second use of the simplified structure of open contours in Ising Polymers:
  in order to derive these estimates we consider only three basic steps
   \eqref{eq:BasicProb} of the effective walks.
   \smallskip

 \noindent
 {\bf Section~\ref{sec:provaprop10}.} In this concluding section we prove
 Proposition~\ref{lem:Plus-bounds} and
 Proposition~\ref{prop:P-plus-bounds}.
 \smallskip

\noindent
\paragraph{\bf  Notations for constants and norms.}
It will be crucial in the whole work to be precise on which  estimates
are uniform with respect to $\sfn \in \mathbb S^1,\beta$ large and which are not.
Therefore, every time some constant $c(a,b,\dots)$ appears in an estimate, it
will be understood that it is \emph{not} uniform w.r.t parameters
$a,b,\dots$, while it is uniform w.r.t. everything else. On the other hand,  numerical values
of constants $c_1 (a,b,\dots) , c_2 (a,b,\dots), \dots$  may change between
different subsections.\newline
A particular role will be played by a mass gap constant $\nu_{\sfg} >0$ (see
\eqref{eq:exp-tails}) which is {\em independent} of $\beta$ and will be fixed throughout
the paper. In Sections 5-8 we also fix a positive constant
$\delta \in (0, \frac{\nu_{\sfg}}{4}]$.
\newline
$\abs{\cdot}_1$ is the $L^1$-norm of either $\bbR^2$ or, in most cases,  $\bbZ_*^2$. The notation
$\abs{\cdot}$ is reserved for number of edges in contours.

\section{Reformulation of the Main Result}
\label{sec:Reformulation}
\subsection{Hidden variables and independent increments representation}

The potentials $\Phi\left(  \cdot\right) ,\tilde\Phi\left(  \cdot\right)$ take values of both signs. For our
purposes it is more convenient if they take only positive values,  so
we manipulate them to obtain this
%goal.
property. The advantage is that, this
way,
the weights $q([\gamma,\ucalC])$ and $q^{+,\sfn}([\gamma,\ucalC])$ in
\eqref{eq:qweight}, \eqref{eq:qweightn} below are positive and can be
considered as a (non-normalized) probability law.
{ This construction goes back to  \cite{DS}.}
Given a contour $\gamma$, we define the set of (not necessarily
distinct) bonds
\[
\nabla_\gamma=\cup_{b=(x,x+e)\in \gamma}\{b,b+e,b-e\}.
\]
For
$b\in\nabla_\gamma$ the
multiplicity of $b$ is the number of bonds $b^{\prime}%
\equiv\left(  x,x+e\right)  $ in $\gamma,$ for which either
$b=b^{\prime},$ or else $b=b^{\prime}\pm e.$

Let $b\in\mathbb{Z}_{\ast}^{2}$ be some fixed bond. Define the value $c\left(
\beta\right)  $ by%
\[
c\left(  \beta\right)  =\sum_{\calC\subset\mathbb{Z}^{2}:\calC\cap b\neq\varnothing
}\exp\left\{  -\chi
%\left(
{
\beta
}
%-\beta_{0}\right)
({\rm diam}_\infty(\calC)+1)
\right\}  ,
\]
and for every connected $\calC$ and $\gamma$ put%
\[
\Phi^{\prime}\left(  \calC,\gamma\right)  =\Phi\left(  \calC,\Delta_\gamma\cap
\calC\right)  +|\calC\cap\nabla_\gamma|\exp\left\{  -
\chi
%\left(
\beta
%-\beta_{0}\right)
\left(
%\vert \delta \calC
{\rm diam}_\infty(\calC)+1
%\right\vert
\rb
 \right\}  ,
\]
where $|\calC\cap\nabla_\gamma|$ is the number of bonds in $\nabla_\gamma  $ that the set  $\calC$ intersects with, each bond
counted with its multiplicity.
Note that $\Phi'$ depends on $\gamma$ through both $\calC\cap\Delta_\gamma$ and
$\calC\cap \nabla_\gamma$, and also that
$\calC\cap\Delta_\gamma\ne\varnothing$
implies  $\calC\cap\nabla_\gamma\ne\varnothing$ (while the converse does
not
necessarily hold).

{   Clearly, by this we achieve that
$$\Phi^{\prime}\left(  \calC,\gamma
\right)  \ge0 \text{ if} \nabla_\gamma\cap \calC\neq\varnothing ,$$
 while at the same time the}
function $\Phi^{\prime}$ satisfies the same decay estimate $\left(  \ref{07}\right)
$ (with the constant $\beta_{0}$ slightly changed). We warn the reader that, for lightness of
notation, from now on $\beta_0$ will be simply removed from all formulas.
Note also that by definition the function $\Phi^{\prime}\left(  \calC,\gamma
\right)  $ inherits the translation invariance property.

It is easy to check (using the fact that $\nabla_\gamma$ contains
three bonds for each bond of $\gamma$) that the weight (\ref{06}) can be rewritten as
\begin{equation}
w\left(  \gamma\right)  =\exp\Bigl[ -\left(  \beta+3c\left(  \beta\right)
\right)  \left\vert \gamma\right\vert +\sum_{\calC:\calC\cap\nabla_\gamma
\neq\varnothing}\Phi^{\prime}\left(  \calC,\gamma\right)  \Bigr]  \label{eq:2}%
\end{equation}
and analogously for $\tilde{w}(\gamma)$.

\subsection{Representation of interfaces in terms of animals.}
{\bf Interfaces without pinning.}
Let us consider first interfaces without any wall or pinning potential.
Interfaces are modeled
by the following
ensemble of random
animals $\Gamma = [\gamma , \ucalC ]$, with
$\gamma$ an open contour on $\bbZ_*^2$
and $\ucalC = \lbr \calC_i\rbr$ a collection of connected
subsets of $\bbZ^2$, called `clusters'.
To an animal $\Gamma$ we associate
{   the
weight}
\be
\label{eq:qweight}
q ([\gamma , \ucalC ] ) = {\rm e}^{-\beta\abs{\gamma}}\prod_i \Psi (\calC_i ;\gamma )
%\1_{\gamma\subseteq A}
\
\text{ and we define}\
q (\gamma ) = \sum_{\ucalC} q ([\gamma , \ucalC ] )
\ee
where
\[\Psi(\calC;\gamma)= \left[\exp(\Phi'(\calC,\gamma))-1\right]\1_{\lbr\calC\cap \nabla_\gamma\neq 0\rbr}.
\]
We immediately recognize from \eqref{eq:2} that, modulo redefining
$\beta+3c(\beta)$ to be $\beta$, we have $w(\gamma)=q(\gamma)$.
The ``potential'' $\Psi$ is non-negative, translation-invariant,
 and is local in the
sense that it depends on $\gamma$ only through $\calC\cap \nabla_\gamma$.
We define the two-point function
\be
\label{eq:two-point}
G_\beta (\sfx ) = \sum_{\gamma : 0_*\mapsto\sfx }q (\gamma ).
\ee
It is well known that for the low temperature ($\beta$ large)  models which
we consider here,
the surface tension in \eqref{eq:2tau} exists. One can extend $\tau_\beta$
to a (strictly convex) function on $\mathbb R^2$, by letting
$\tau_\beta(\sfx )=|\sfx|\tau_\beta ( \sfn_\sfx )$,
where $
\sfn_\sfx = \sfx/\abs{\sfx}
$.
Recall that  we assume that
the surface tension  possesses the discrete reflection/rotation symmetries of $\bbZ^2$.
 Other properties of the surface
tension are given in Theorem \ref{thm:OZ}.
Also, it is known that
for all $\beta$ large
there exists a positive  locally analytic function $C (\beta , \cdot )$ on $\bbS^1$, such that
\be
\label{eq:surface-tension}
G_\beta (\sfx) = \frac{
C (\beta , \sfn_\sfx ) (1 +{\mathrm o} (1)  )
}{\sqrt {|\sfx|}}{\rm e}^{- \tau_\beta (\sfx)},
\ee
uniformly in $
\abs{\sfx} \to \infty
$. 
We will need later also the ``restricted two-point function''
\be
\label{eq:two-point-restr}
G_\beta \lb \sfx~|~\calP_{+ , \sfn
}\rb:=
\sumtwo{\gamma : 0_*\mapsto\sfx }{\gamma\in \calP_{+,\sfn}}q (\gamma )
\ee
(in general, we will write $G_\beta \lb \sfx~|~A\rb$ for the two-point
function with paths restricted to some set $A$).

\paragraph{\bf Interfaces with pinning.}
In analogy with \eqref{eq:qweight},
given animal $\Gamma$ we define
weights
\be
\label{eq:qweightn}
q^{+,\sfn} ([\gamma , \ucalC ] ) = {\rm e}^{-\beta\abs{\gamma}}\prod_i \Psi^{+,\sfn} (\calC_i ;\gamma )
\1_{\gamma\in \calP_{+ ,\sfn}}\ {\rm and}\
q^{+,\sfn} (\gamma ) = \sum_{\ucalC} q^{+,\sfn} ([\gamma , \ucalC ] )
\ee
with $\calP_{+ ,\sfn}$ the set of paths which stay inside $\calH_{+ ,
  \sfn}$
and
\[
\Psi_\beta^{+,\sfn} (\calC ;\gamma
)=\left[\exp(\tilde\Phi'(\gamma,\calC))-1
\right]\1_{\lbr\calC\cap \nabla_\gamma\neq 0\rbr}
\ge0.
\]
Again, we recognize that $q^{+,\sfn}(\gamma)$ is
just
$\tilde w(\gamma)$ (modulo redefining $\beta+3c(\beta)\mapsto \beta$).
The decay assumption \eqref{07} (and the analogous one for
$\tilde\Phi$) implies the following: for every $\gamma \in \calP_{+ , \sfn}$,
\be
\label{eq:path-A-weights}
 q  (\gamma ) {\rm exp}\Bigl[ -\sum_{\sfu\in\gamma}
{\rm e}^{-\chi
%(
\beta
%-\beta_0)
\lb
%1\vee
\dd_{\sfn} (\sfu ) +1 \rb }\Bigr] \le q^{+,\sfn} (\gamma )
\leq q  (\gamma ) {\rm exp}\Bigl[ \sum_{\sfu\in\gamma}
{\rm e}^{-\chi
\beta
%-\beta_0)
\lb
%1\vee
\dd_{\sfn} (\sfu ) +1 \rb }\Bigr]
\ee
with $\sfu\in \bbZ_*^2$ the endpoints of bonds of $\gamma$.

In analogy with  \eqref{eq:two-point}, we define the two-point function
``with pinning''
\be
\label{eq:two-pointn}
G^{+,\sfn}_\beta (\sfx ) = \sum_{\gamma : 0_*\mapsto\sfx }q^{+,\sfn} (\gamma )
\ee
(remark that only contours $\gamma\in\calP_{+,\sfn}$ contribute to the
sum).
{   Again, we will write $G^{+,\sfn}_\beta \lb \sfx~|~A\rb$ for the two-point
function with paths restricted to some set $A$.}
Then, the main claim \eqref{eq:1} in Theorem \ref{NoPin} can be reformulated as follows:
\begin{theorem}
 \label{thm:nopinning}
Assume that \eqref{eq:path-A-weights} holds with $\chi >1$.
Then there exists $\bar\beta = \bar\beta(\chi )$
such that the following holds: For any $\beta >\bar\beta$ there exist two
constants $c_1 (\beta )$ and $c_2 (\beta )$
such that
\be
\label{eq:nopinning}
c_1 (\beta ) G_\beta \lb \sfx~|~\calP_{+ , \sfn
}\rb \,
\le \,
G_\beta^{+ ,\sfn}  (\sfx )\,  \le\,  c_2 (\beta )\;
G_\beta \lb \sfx~|~\calP_{+ , \sfn
}\rb.
\ee
uniformly in $\sfx\in \calB_{+ ,\sfn} $  and $\arg(\sfn )\in
{
[-\frac{\pi}{4}, \frac{3\pi}{4}]
}
$.
\end{theorem}
We will see that the most difficult case is when % $\sfx\in\calH_{+ , \sfn} $ lies close to
$\sfn$ is a lattice direction.
We shall prove Theorem~\ref{thm:nopinning} uniformly in
 ${\rm arg}(\sfn )\in [\frac{\pi}{2} , \frac{3\pi}{4}]$:
the other cases will follow by lattice symmetries.

The result actually holds also if the endpoints of the contour are
not on the line $\mathcal B_{+,\sfn}$ (and the proof is easier).
\smallskip

\noindent
\paragraph{\bf Convention for lattice notation.} For historic reasons it was natural
to define contours as sets of edges on the dual lattice $\bbZ_*^2$. However, as far as
formulas are concerned, it is  more convenient to work with the direct lattice
$\bbZ^2$. From now on we shall identify $\bbZ_*^2$ with $\bbZ^2$ via the map
$\sfu\mapsto \sfu - \sf 0_*$. Under this map the positive half-plane should be redefined
as
\be
\label{eq:Hplus-new}
\calH_{\sfn}^+  = \lbr \sfx : \sfx\cdot \sfn \geq 0 \rbr .
\ee
Similarly, under the above convention $\calP_{+, \sfn}$ is the set of paths
$\gamma = (\gamma_0, \dots , \gamma_m )\subset \bbZ^2$ which satisfy $\gamma_i\cdot\sfn \geq 0$.

\section{Irreducible decomposition of interfaces}
\label{sec:Irreducible}
In this Section we describe a decomposition of decorated contours in
terms of strings of irreducible animals \eqref{eq:irreducible-A}.
At low temperatures this decomposition is inherited from the corresponding
irreducible decomposition of naked open contours with weights
${\rm e}^{-\beta\abs{\gamma}}$. The latter is based on
the mass-gap estimate
\eqref{eq:breakpoints-SAW}. In view of \eqref{eq:lengthsaw}
 and of \eqref{eq:ratio-SAW}, the mass-gap
property persists for decorated contours as soon as $\beta$ is sufficiently
large. This is \eqref{eq:exp-tails}, and the decay exponent (mass-gap) $\nu_{\sfg}$
which appears therein is fixed throughout the paper.
Properties of the irreducible decomposition are listed in
Theorem~\ref{thm:OZ}. \eqref{eq:OZ-weights} defines a class of probability
distributions on the alphabet of irreducible animals, which sets up the
stage for the renewal analysis in the sequel.

Ratios of partition functions
of pinned and free ensembles are controlled by
\eqref{eq:q+-crude}. By the mass gap estimate \eqref{eq:exp-tails}
the pinned two-point  function $G_\beta^{+ ,\sfn}  (\sfx )$
is bounded above by the expression in \eqref{eq:Gbeta-decomp}, and
consequently a proof of upper bound in Theorem~\ref{thm:nopinning} is
reduced to a verification of \eqref{eq:Target1}.

\subsection{ Crude comparison with
ensembles of SW paths.
}
%Recall that open contours are lattice paths subject to the SW-splitting rule.
Paths $\gamma = \lb\gamma_0 , \dots , \gamma_n\rb$ are open contours with edges
$e_l = (\gamma_{l-1}, \gamma_l )$ which obey rules as specified in the beginning of
Section~\ref{sec:themain}.
With each such path we may associate the ``cluster-free'' weight ${\rm e}^{-\beta\abs{\gamma}} $.
For a subset $\calP$ of paths, the restricted  two point functions for the SW-ensemble (SW recalling the south-west splitting rule) are defined via:
\[
 G_\beta^{{\rm SW }} (\sfx \big|~\calP ) =
\sumtwo{\gamma : 0\mapsto\sfx }{\gamma\in\calP} {\rm e}^{-\beta\abs{\gamma}} .
\]
The
two-point   functions  we are working with
is a
perturbation of the latter.
Although Theorem~\ref{thm:nopinning} eventually relies on a more delicate
analysis, heavy duty
estimates on exponential scales lead to a convenient geometric setup. Let us formulate basic
geometric properties of
free SW-paths:
\paragraph{\bf Forward cone $\calY$.} For the rest of this section  fix
$\kappa = \arctan (1/2)$ and define a
positive  cone
\be
\label{eq:Ycone}
 \calY = \lbr \sfx~:~ -   \kappa \leq {\rm arg}(\sfx ) \leq \frac{\pi}{2} +  \kappa\rbr .
\ee
The cone  $\calY$ is strictly  contained in the half-plane $\lbr\sfx = (x,y)~:~ x+ y\geq 0\rbr$
and it  contains the positive quadrant   $\calQ_+ \df \lbr \sfx = (x, y) : x,y\geq 0\rbr$
in its
interior (see Figure~\ref{fig:Decomp} below).
\paragraph{\bf Definition (break points of paths).}
A path $\gamma = \lb\gamma_0, \gamma_1, \dots , \gamma_n\rb$ is
said to have a  break point at $\sfu = \gamma_\ell\in\gamma ;\ 0 <\ell <n,$ if
\[
%\gamma\subset  (\sfu-\calY_2 )\cup(\sfu+\calY_2 )
\lbr\gamma_0 , \dots, \gamma_{\ell-1}\rbr \subset \gamma_\ell - \calY
%\gamma_m\cdot\sfe_1 \leq \gamma_\ell\cdot\sfe_1\ \forall~m <\ell
\quad{\rm and}\quad
  \lbr\gamma_{\ell +1}, \cdots , \gamma_n\rbr \subset\gamma_{\ell} +\calY .
%\gamma_m\cdot\sfe_1 > \gamma_\ell\cdot\sfe_1\  \forall~m >\ell .
\]
{If a path has no break points, it is called irreducible.}
\begin{lemma}
 \label{lem:saw}
There exist $\beta_0 <\infty$,  $\nu_0 >0$ and $r_0 <\infty $ such that
\be
\label{eq:lengthsaw}
G_\beta^{\rm SW} \lb\, \sfx\,  \big| \abs{\gamma}   \geq  r |{\sfx}|_1\rb
\leq c_1\, {\rm e}^{-2\nu_0\beta r\abs{\sfx}_1}
G_\beta^{\rm SW} \lb \sfx\rb
\ee
uniformly in $\beta\geq \beta_0$, $\sfx$ and $r\geq r_0$.
Furthermore,
let $\calP_{n}$ be the set of paths with at least $n$  break points. Then, there exist
$\delta_0, \nu_0 >0$ such that
\be
\label{eq:breakpoints-SAW}
G_\beta^{\rm SW} \lb\, \sfx\,  \big|
\calP_{\delta_0\abs{\sfx}_1}^{\sfc}\rb \leq c_1\,
{\rm e}^{-2\beta\nu_0 \abs{\sfx}_1} G_\beta^{\rm SW} \lb \sfx\rb.
\ee
uniformly in $\beta\geq \beta_0$ and  $\sfx\in\calQ_+$.
\end{lemma}

The inequality \eqref{eq:lengthsaw} is straightforward. The inequality
\eqref{eq:breakpoints-SAW} follows by an easy modification of renormalization
arguments leading to Theorem~3.1 in \cite{IV08}.
\newline
By \eqref{eq:path-A-weights}
\be
\label{eq:ratio-SAW}
 \left|\log\frac{q  (\gamma )}{q^{\rm SW} (\gamma )}\right|
 \leq c_2\, {\rm e}^{-\chi\beta}\abs{\gamma} .
\ee
Therefore, \eqref{eq:lengthsaw} and \eqref{eq:breakpoints-SAW} imply:
\be
\label{eq:breakpoints}
G_\beta  \lb\, \sfx\,  \big| \abs{\gamma}   \geq  r\abs{\sfx}_1\rb
\leq c_3\, {\rm e}^{-\nu_0\beta r\abs{\sfx}_1}
G_\beta  \lb \sfx\rb\quad
{\rm and}\quad
G_\beta \lb \sfx\, \big|   \calP_{\delta_0\abs{\sfx}_1}^{\sfc}\rb \leq c_3\,
{\rm e}^{-\beta\nu_0 \abs{\sfx}_1} G_\beta  \lb \sfx\rb
\ee
uniformly in $\beta$ sufficiently large, $\sfx\in\calQ_+$
and  $r\geq r_0$.
\paragraph{\bf Crude bounds on $G^{+ \sfn}_\beta (\sfx )$.}
With \eqref{eq:breakpoints}  at hand, it is an easy consequence of the   Ornstein-Zernike   theory developed in
\cite{IV08} (see the local limit formula (3.10) there)
that for any
large $\beta$ fixed
and
for any $\delta >0$,  one has $G_\beta ({\sfx} )
{
\leq
}
 G_\beta (\sfx |
\calP_{+ , \sfn} ){\rm e}^{{4} \delta \beta {|\sfx|_1}}$, uniformly
{
$\sfn$ and in {$\sfx\in \calB_{+ , \sfn}$ with
}
$\abs{\sfx}_1$ large.
}

{
Indeed, it is enough to consider $\arg (\sfn ) \in [\frac{\pi}{2} , \frac{3\pi}{ 4}]$. Let
$\sfv \in \calH_{+ , \sfn}$ be a unit vector such that $\sfv \not\in -\calY\cup\calY$. Define
$\sfO_\delta = \lfloor \delta\abs{\sfx}_1\rfloor\sfv$ and
$\sfx_\delta = \sfx+ \lfloor \delta\abs{x}_1 \rfloor\sfv$.
For $\abs{\sfx}_1$ large,  points $\sfO_\delta$ and $\sfx_\delta$ sit deep inside $\calH_{+, \sfn}$.
Define $D (\sfO_\delta , \sfx_\delta ) = \lb \sfO_\delta +\calY\rb\cap \lb \sfx_\delta -\calY\rb$, and
consider the restriction of
$G_\beta (\sfx |
\calP_{+ , \sfn} )$ to paths $\gamma$ which are concatenations
$\gamma = \gamma_1\circ\gamma_2\circ\gamma_3$, where $\gamma_1 :0\to\sfO_\delta$,
$\gamma_2 : \sfO_\delta\to \sfx_\delta$, $\gamma_3 : \sfx_\delta\to\sfx$, and, in addition,
$\gamma_i\cap D (\sfO_\delta , \sfx_\delta ) =\emptyset$ for
 $i=1, 3$, whereas $\gamma_2 \subset D (\sfO_\delta , \sfx_\delta )$. The contribution of
 $\gamma_1$ and $\gamma_3$ is bounded below by ${\rm e}^{-3\beta\abs{\sfx}_1}$. On the other hand, the main
 contribution from $\gamma_2$ come from those paths which obey the Brownian scaling and hence
 stay inside $\calH_{+ , \sfn }$.
} 

Therefore, quantities which are exponentially negligible with respect
to $G_\beta (\sfx )$ are exponentially negligible with respect to
$G_\beta (\sfx |
\calP_{+ , \sfn} )$ as well.

By \eqref{eq:path-A-weights}  we have for every $\calP$
 \be
\label{eq:q+-crude}
\begin{split}
&{
\sumtwo{\gamma :0 \mapsto\sfx}{\gamma\in \calP_{+, \sfn}\cap
\calP}
q (\gamma ){\rm exp}\lbr -  \sum_{\sfy\in\gamma}{\rm e}^{-\chi\beta (
%1\vee
%{\rm d}_\infty(\sfy ,
%\calH_{+ ,\sfn}^{\sfc})
\dd_\sfn (\sfy )
+1)}\rbr
}\\
&\qquad \leq \,
G^{+ , \sfn}_\beta (\sfx \big|~\calP )
\leq \sumtwo{\gamma :0 \mapsto\sfx}{\gamma\in \calP_{+, \sfn}\cap
\calP}
q (\gamma ){\rm exp}\lbr \sum_{\sfy\in\gamma}{\rm e}^{-\chi\beta (
%1\vee
%{\rm d}_\infty(\sfy ,
%\calH_{+ ,\sfn}^{\sfc})
\dd_\sfn (\sfy )
+1)}\rbr .
\end{split}
\ee

By \eqref{eq:breakpoints} 
we may restrict attention to
paths
$\gamma\in \calP_{\delta_0\abs{\sfx}}$.

A look at \eqref{eq:q+-crude} reveals that the lower bound in
\eqref{eq:nopinning} of Theorem~\ref{thm:nopinning} is the easier one.
Indeed, one should merely argue that for typical interfaces
$\gamma \subset \calP_{+, \sfn}$
 with
(full space) $q (\gamma )$-weights
 the quantity
$\sum_{\sfy\in\gamma}{\rm e}^{-\chi\beta (
%1\vee
%{\rm d}_\infty(\sfy ,
%\calH_{+ ,\sfn}^{\sfc})
\dd_\sfn (\sfy )
+1)}$
 is uniformly bounded from above. The latter property will be a
 consequence of the fact that
such typical interfaces
are sufficiently repelled from $\calB_{+, \sfn}$.
On the contrary, to prove the upper bound one should explore in
depth the competition
between pinning and
repulsion. Namely, the gain $\sum_{\sfy\in\gamma}{\rm e}^{-\chi\beta (
%1\vee
%{\rm d}_\infty(\sfy ,
%\calH_{+ ,\sfn}^{\sfc})
\dd_\sfn (\sfy )
+1)}$ should be measured
%on all possible scales
against the entropic price of
bringing interfaces $\gamma$  close to the wall.

\paragraph{\bf Irreducible decomposition of paths.}
 Paths $\gamma\in \calP_{\delta_0\abs{\sfx}}$  admit a natural irreducible decomposition
\be
\label{eq:irreducible}
\gamma = \gm{\sfl}\circ \gm{1}\circ \dots\circ \gm{n}\circ\gm{\sfr} .
\ee
Above $\gm{\sfl}$ is a left-irreducible path: $\gm{\sfl} \in \sfP_{\sfl}$ and  $\gm{\sfr}$
is a right-irreducible path: $\gm{\sfr} \in \sfP_{\sfr}$. The paths $\gm{1}, \dots ,\gm{n} \in
\sfP = \sfP_{\sfl}\cap\sfP_{\sfr}$ are irreducible.

The alphabets  $\sfP_{\sfl}$ and $\sfP_{\sfr}$ could be described as follows:
$\gamma = \lb \gamma_0 , \dots , \gamma_m\rb \in \sfP_{\sfl}$ if
%\[
% \text{
$\gamma$ does not contain break points and $\gamma\subseteq \gamma_m -
\calY$. %} .
%\]
Similarly,
$\gamma = \lb \gamma_0 , \dots , \gamma_m\rb \in \sfP_{\sfr}$ if
$\gamma$ does not contain break points and $\gamma\subseteq \gamma_0 + \calY$.
In the sequel we shall
 use $\ugamma$ for strings of letters from $\sfP$.
The notation $\sfP (\sfx ,\sfy)$ is reserved for irreducible paths with
end points at $\sfx$ and $\sfy$.
Note that
% $\sfP (\sfx ,\sfy)$ can be  non-empty only if $\sfx\prec\sfy$, and in the latter case
any path $\gamma\in \sfP (\sfx ,\sfy)$ automatically lies inside the diamond shape
(see Figure~\ref{fig:Decomp}).

{In the sequel we shall
 use $\ugamma$ for strings of letters from $\sfP$.
 The strings of $\ell$ letters will be denoted by $\frP_\ell$, $\ell\leq \infty$.} In this way,
\eqref{eq:irreducible} reads as $\gamma =
\gm{\sfl}\circ\ugamma\circ\gm{\sfr}$,
$\ugamma\in \frP_n$.
In general, with each path  $\gamma = (\gamma_0 , \dots , \gamma_k ) $ we associate:
 $\abs{\gamma} = k$ (the length of $\gamma$) and $\sfX (\gamma ) =
\gamma_k - \gamma_0$ (the displacement).
For $\ugamma \in\frP_n$ we define:
\[
 \abs{\ugamma} = \sum_1^n \abs{\gm{i}} %,\quad
%\tau (\ugamma ) = \sum_1^n \tau (\gamma_i )
\quad {\rm and}
\quad
\sfX (\ugamma ) = \sum_1^n \sfX (\gm{i}  ) .
\]
\paragraph{\bf Irreducible animals.}
Let us say that an animal $\Gamma = [\gamma , \ucalC]$
has a break point at
$\sfu\in\gamma$ if $\sfu$ is a break point of $\gamma$, and if
\[
 \cup_i \calC_i\subset \lb \sfu - \calY\rb\cup\lb \sfu+\calY\rb .
\]

The collections $\sfA_{\sfl}$, $\sfA_{\sfr}$ and $\sfA = \sfA_{\sfl}\cap\sfA_{\sfr}$
of respectively left irreducible, right irreducible and irreducible animals are defined as
in the case of paths. For instance, $\Gamma = [\gamma , \ucalC]\in\sfA (\sfx ,\sfy )$ if
$\Gamma$ does not contain
break points and
% \label{eq:Gamma-D}
$\Gamma\subset D (\sfx , \sfy ) ,$
where $\sfx , \sfy$ are  the end points of $\gamma = (\gamma_0 ,\dots, \gamma_n )$;
  $\sfx = \gamma_0$, $\sfy = \gamma_n$, and the diamond shape
$D (\sfx , \sfy )$ was defined in \eqref{eq:D-shape} (see Figure~\ref{fig:Decomp}). Note that
$[\gamma , \uempty ]\in \sfA (\sfx , \sfy )$ iff $\gamma\in\sfP (\sfx ,\sfy )$. More generally,
$[\gamma , \ucalC] \in \sfA$ implies that $\gamma$ is a word  from
$\frP_\ell$ for some $\ell\geq 1$.
\be
\label{eq:D-shape}
 D (\sfx , \sfy ) \df \lb \sfx +\calY\rb\cap \lb \sfy -\calY\rb .
\ee
\begin{figure}[t]
\begin{center}
\includegraphics[width=0.9\textwidth]{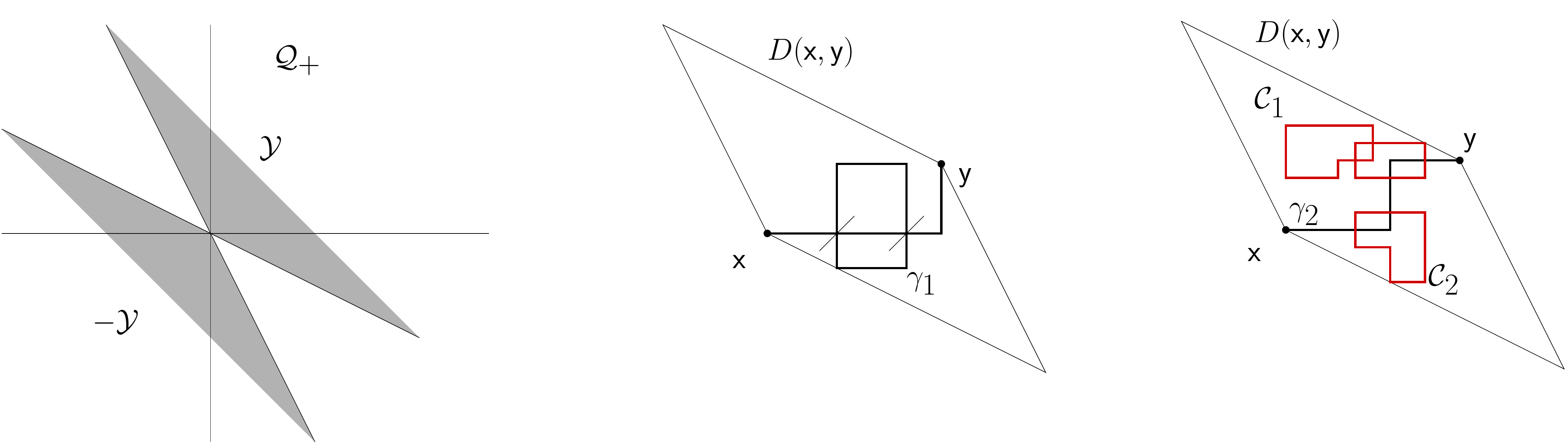}
\end{center}
\caption{Positive quadrant $\calQ_+$ and  cones $\calY$, $-\calY$. Irreducible
path $\gamma_1 \subset D(\sfx , \sfy )$ and irreducible animal $[\gamma_2 , \calC_1, \calC_2]$.
}
\label{fig:Decomp}
\end{figure}
In its turn the notation $\frA_\ell$ stands for words of
$\ell$ irreducible animals, and, in the latter case,  we shall write $\uGamma\in \frA_\ell$.
By construction, such $\uGamma$ is represented as $\uGamma = [\ugamma , \ucalC]$, where
$\ugamma$ is a concatenation of letters  from $\sfP$.
{The notation $\frA_\ell (\sfx, \sfy )$ stands for those elements of $\frA_\ell$ which
have the left end-point at $\sfx$ and the right end point at $\sfy$. Finally,
$\frA (\sfx , \sfy )\df  \cup_\ell\frA_\ell (\sfx, \sfy )$.}
In the case of animals the quantities  $\abs{\uGamma}$, $\sfX (\uGamma )$ are defined through
the corresponding path components. That is:
$\abs{[\gamma , \ucalC]}\df \abs{\gamma}$ and $\sfX\lb [\gamma , \ucalC ]\rb \df \sfX (\gamma )$.

As  discussed above we may restrict attention to those animals which contain
at least $\delta^\prime\abs{x}$ break points (for some $\delta^\prime >0$). This leads
to the irreducible decomposition
\be
\label{eq:irreducible-A}
[\gamma , \ucalC ] = \Gm{\sfl}\circ\Gm{1}\circ\dots\circ\Gm{n}\circ\Gm{\sfr} ,
\ee
with $\Gm{\sfl}\in\sfA_\sfl$, $\Gm{\sfr}\in\sfA_\sfr$ and $\Gm{1}, \dots , \Gm{n}\in \sfA$.
\paragraph{\bf Input from Ornstein-Zernike (OZ) theory.}
The relevant input from the OZ theory
{
(see for instance Subsections~3.3  and 3.4  of \cite{IV08})}
could be summarized as follows:
\begin{theorem}
 \label{thm:OZ}
 For all $\beta$ large enough
the surface tension $\tau_\beta$ in \eqref{eq:surface-tension} is well defined and it
is a support
function
of a convex set $\Kb$ with non-empty interior and locally analytic boundary $\pKb$, which
has a uniformly positive curvature. In particular $\tau_\beta$ is differentiable at any
$\sfx\neq 0$ and $\sfh = \sfh_\sfx\df \nabla\tau_\beta (\sfx )\in\pKb$.
The Wulff shape $\Kb$ inherits the full set of $\bbZ^2$-lattice symmetries from the
surface tension $\tau_\beta(\sfx)$.
In particular $\sfh_\sfx\in \calQ_+$ whenever $\sfx\in \calQ_+$.
In geometric
terms $\sfh_\sfx$ can be characterized
in the following way: $\sfx$ is direction of the outward normal
 to
$\pKb$ at $\sfh_\sfx$.
In view of smoothness and
strict convexity of $\pKb$ this is an unambiguous characterization.
\newline
For any $\sfx\in\calQ_+\setminus 0$ the collection of weights
\be
\label{eq:OZ-weights}
\bbP_\beta^{\sfh_\sfx }\lb \Gamma \rb \df {\rm e}^{\sfh_\sfx\cdot \sfX (\Gamma )}q (\Gamma )
\ee
is a probability distribution on the set $\sfA$ of irreducible animals.
The expectation $\sfv^* (\beta ,\sfx )$
of $\sfX (\Gamma )$ under
$\bbP_\beta^{\sfh_\sfx }$ is collinear to $\sfx$: there exists
$\alpha = \alpha (\beta , \sfx ) >0 $ such that $\sfv^* (\beta ,\sfx ) = \alpha\sfx$.
Note that, since $\tau_\beta $ is homogeneous of order one, $\sfv^*
(\beta ,\sfx )$
depends only on the direction of $\sfx$.

Furthermore, there exists a (mass-gap) constant $\nu_{\sfg}  >0$, such that
\be
\label{eq:exp-tails}
\sum_{\Gamma\in \sfA_\sfl} \bbP_\beta^{\sfh_\sfx} (\Gamma )\1_{\lbr \abs{\Gamma}\geq k\rbr} +
\sum_{\Gamma\in \sfA_\sfr} \bbP_\beta^{\sfh_\sfx} (\Gamma )\1_{\lbr \abs{\Gamma}\geq k\rbr}
\leq c\, {\rm e}^{-\nu_{\sfg}\beta k} ,
\ee
uniformly in $\beta$ large,  $\sfx\in\calQ_+$ and $k>1$.
\end{theorem}
{
\begin{remark}
\label{rem:tilts}
In the sequel we shall sometimes employ an alternative notation $\sfh_\epsilon = \sfh_\sfx$ for
$\sfx\in\calQ_+$ satisfying $\frac{\sfx}{\abs{\sfx}_1} = (1-\epsilon  , \epsilon )$.
\end{remark}
}
\paragraph{\bf The target {upper} bound.}
Let us fix (without loss of generality) $\sfn$ with ${\rm arg}(\sfn )
\in [\frac{\pi}{2} , \frac{3\pi}{4}]$ and
%let
 $\sfx \in \calB_{+, \sfn}$.
%\calQ_+\cap \calH_{+ , \sfn}$.
To facilitate notation set
\be
\label{eq:Al-plus}
\frA_\ell^{+, \sfn} = \frA_\ell\cap\lbr \uGamma = [\ugamma , \ucalC] :
\ugamma\subset\calH_{+ , \sfn}\rbr ,
\ee
and, for any $\sfu , \sfv \in\calH_{+ , \sfn}$,  $\frA_\ell^{+ ,\sfn}(\sfu, \sfv ) =
\frA_\ell^{+ , \sfn} \cap\lbr \uGamma = [\ugamma , \ucalC] : \ugamma :\sfu\mapsto \sfv\rbr$.

Recall that $\sfh=\sfh_\sfx = \nabla\tau_\beta (\sfx )$ so that
$\sfh\cdot\sfx = \tau_\beta (\sfx )$.
In view of \eqref{eq:q+-crude} and \eqref{eq:exp-tails},
{the two-point function $G_\beta^{+ , \sfn} (\sfx )$  is bounded above by}
\be
\label{eq:Gbeta-decomp}
%\begin{split}
%&
%\leq
{
c_1 (\beta )
}
\!\!\!
\sumtwo{\sfu  , \sfx  - \sfv\in \calY} {\sfu,\sfv\in \calH_{+,\sfn}}
%{\sum_{\Gm{\sfl}: 0\mapsto \sfu}}^{\!\!\!\!*}\
% {\sum_{\Gm{\sfr}: \sfv\mapsto \sfx}}^{\!\!\!\!*}\
{\rm e}^{-\nu_{\sfg}\beta\lb
%\abs{ \Gm{\sfl}} + \abs{ \Gm{\sfr}}
\abs{\sfu}_1 +\abs{\sfx - \sfv}_1
\rb}
%&
\sum_{\ell}\
\sum_{\uGamma\in\frA_\ell^{+ ,\sfn}(\sfu, \sfv )} \bbP_{\beta}^{\sfh}\lb \uGamma \rb
{\rm exp}\lbr \sum_{\sfy\in \ugamma} {\rm e}^{-\chi\beta (
\dd_\sfn (\sfy )
+1)}\rbr .
\ee
We used that $\abs{\sfX (\uGamma )}_1\leq
\abs{\uGamma}$, where $|\sfx|_1$ is the $L^1$ norm of $\sfx$.
Therefore, in order to derive the upper bound in \eqref{eq:nopinning}, 
it suffices to check that
there exists $0 < {4}\delta <\nu_{\sfg}$ such that
\be
\label{eq:Target1}
\begin{split}
&\suptwo{\sfu , \sfx  - \sfv\in \calY} {\sfu,\sfv\in \calH_{+,\sfn}}%\sup_{\sfu, \sfx  - \sfv\in \calY}
{\rm e}^{-{3}\delta \beta\lb
%\abs{ \Gm{\sfl}} + \abs{ \Gm{\sfr}}
\abs{\sfu} _1+\abs{\sfx - \sfv}_1
\rb}
%&
\sum_{\ell}\
\sum_{\uGamma\in\frA_\ell^{+ ,\sfn}(\sfu, \sfv )} \bbP_{\beta}^{\sfh}\lb \uGamma \rb
{\rm exp}\Bigl[ \sum_{\sfy\in \ugamma} {\rm e}^{-\chi\beta (
%{\rm d}(\sfy ,
%\calH_{+ ,\sfn}^{\sfc})
\dd_\sfn (\sfy )
%1\vee \sfn\cdot \sfy
+1)}\Bigr]\\
&\qquad\qquad   \leq
c_2(\beta)
\, G_\beta \lb \sfx \big| \calP_{+ , \sfn} \rb  {\rm e}^{\tau_\beta (\sfx )} \end{split}
\ee
for $\beta$ sufficiently large,
uniformly in $\sfn\in [\frac{\pi}{2} , \frac{3\pi}{4}]$ and $\sfx\in
\calB_{+ ,\sfn}$.

\section{Effective random walk}
\label{sec:ERW}
Steps of the effective random walk are displacements $\sfX (\Gamma )$
along irreducible animals $\Gamma$ which are sampled from the probability
distribution \eqref{eq:OZ-weights}. In this way the constraint
$\underline{\gamma}\subset \calH_{+, \sfn}$ is less stringent than the constraint that
corresponding effective walk stays above the wall. In terms of effective
random walks upper bounds on partition functions
with pinning are given by quantities $\bbG^{\sfh}_{\beta , +}$ defined in
\eqref{eq:G-plus}. The corresponding effective random walk quantities for
models without pinning are probabilities $\bbP_{\beta , +}^{\sfh}$ which are
defined in \eqref{eq:rho-delta}. In the end of the Section we formulate
Theorem~\ref{thm:TwoBounds} and explain how it implies our target upper bound
 \eqref{eq:Target1} and, consequently, the upper bound in Theorem~\ref{thm:nopinning}.

\paragraph{\bf Random walk representation and high temperature expansion.}
Let us reformulate the required bound in the effective random walk context:
For a word  $\uGamma \in \frA_\ell$ with the left end point at $\sfu$ set
\be
\label{eq:Rwalk}
 \sfR_\ell = \sfR_\ell (\uGamma ) = \sfu +  \sum_{i=1}^\ell \sfX (\gm{i} ) \df
\sfu +\sum_{i=1}^\ell \sfX_i ,
\ee
and, accordingly, define $Z_i = \sfX_i\cdot\sfn$ and
\be
\label{eq:Swalk}
% T_\ell = \sfR_\ell\cdot \sfe_1\ = T_0 +\sum_1^\ell \tau_i\  {\rm and}\
S_\ell = \sfR_\ell\cdot \sfn\ = \sfu\cdot\sfn + \sum_{i=1}^n \sfX_i\cdot\sfn =
S_0 +\sum_1^\ell Z_i.
\ee
{
For the random walk starting at $\sfu$ the probability $\bbP_{\beta }^\sfh
\lb \sfR_\ell = \sfv \rb
 = \bbP_{\beta}^\sfh \lb \frA_\ell (\sfu , \sfv )\rb$.
}

Define events (sets of words)
\be
\label{eq:Rlplus-n}
\frR_\ell^{+ , \sfn} =
 \frR_\ell^+ = \lbr \uGamma\in \frA_\ell:\, S_1, \dots , S_{\ell -1} \geq 0\rbr .
\ee
With a slight abuse of notation we shall think of $\frR_\ell^+$ both as a subset of
$\frA_\ell$ and as a subset of $\frA_{m}$ for any $\ell\leq  m\leq \infty $.
The notation $\frR_\ell^+ (\sfu , \sfv )$ stands for $\ell$-strings of irreducible animals
from $\frR_\ell^+$ with the left end point at $\sfu$ and the right end point at $\sfv$;
{
$\frR_\ell^+ (\sfu , \sfv ) = \frR_\ell^+\cap \frA\lb \sfu , \sfv \rb$.
}
 Note
that $\frA_{\ell}^+ (\sfu , \sfv ) \subseteq \frR_{\ell}^+ (\sfu , \sfv )$.

Given a string $\uGamma$ define
(recall the definition \eqref{eq:D-shape} of diamond shapes):
$
 D_\ell =  D (\sfR_{\ell - 1} , \sfR_{\ell })
%\lb \sfR_{\ell - 1}+\calY\rb \cap \lb \sfR_{\ell }-\calY\rb .
$.
By construction,  $\Gm{\ell}\subset D_\ell$. Next define
$d_\ell \df d (\sfR_{\ell - 1} , \sfR_{\ell} )\ge0$ via

\be
\label{eq:d-ell}
d_\ell
=d (\sfR_{\ell - 1} , \sfR_{\ell} )
=\min_{\sfy\in D_\ell \cap\calH_{+,\sfn}\cap \bbZ_*^2}\lb \dd_\sfn (\sfy ) +1\rb -2  \geq 0.
\ee
For strings $\uGamma\in \frA_n^+ (\sfu , \sfv )$ the contour part $\gm{\ell}$ of the $\ell$-th irreducible animal
$\Gm{\ell}$ satisfies:
\be
\label{eq:w-weights}
 \sum_{\sfy\in \gm{\ell}} {\rm e}^{-\chi\beta (\dd_\sfn (\sfy ) +1 )} \leq \abs{\gm{\ell}}
{\rm e}^{-\chi\beta (2 +d_\ell )} \df  \phi_\beta (\gm{\ell} )  .
\ee
Note that the weight $ \phi_\beta (\gm{\ell} )$ just defined can be quite large. But this will
be compensated by the fact that the probability $\bbP_{\beta}^{\sfh}\lb \Gm{\ell} \rb$ of
the corresponding animal $\Gm{\ell}$ is very small.

By \eqref{eq:w-weights},
\be
\label{eq:ves}
\sum_{\uGamma\in\frA_\ell^{+ ,\sfn}(\sfu, \sfv )} \bbP_{\beta}^{\sfh}\lb \uGamma \rb
{\rm exp}\lbr \sum_{\sfy\in \ugamma} {\rm e}^{-\chi\beta (
\dd_\sfn (\sfy )
+1)}\rbr \leq
%+\sum_{\ell}\
\sum_{\uGamma\in\frR_\ell^{+ ,\sfn}(\sfu, \sfv )} \bbP_\beta^{\sfh}\lb \uGamma \rb
\prod_{j=1}^{\ell} e^{ \phi_\beta (\gm{j} )}
\ee
for any $\ell  = 1,2,3\dots $.
Define
\be
\label{eq:G-plus}
{
\bbG^{\sfh }_{\beta , +} (\sfu,\sfv)
}
=
\sum_{\ell}\
\sum_{\uGamma\in\frR_\ell^{+ ,\sfn}(\sfu, \sfv )} \bbP_\beta^{\sfh}\lb \uGamma \rb
\prod_{i=1}^{\ell} {\rm e}^{\phi_\beta (\gm{i} )}.
\ee
We conclude that
\be
\label{eq:Target2}
\text{ left hand side of \eqref{eq:Target1}}\le
\suptwo{\sfu, \sfx  - \sfv\in \calY} {\sfu,\sfv\in \calH_{+,\sfn}}
{\rm e}^{-{3} \delta \beta \lb \abs{\sfu }_1 +\abs{\sfx -\sfv }_1\rb}
 \bbG^{\sfh }_{\beta , +}  (\sfu,\sfv) .
\ee
Observe that
\be
\label{eq:abs-vs-d}
|\sfu|_1+|\sfv -\sfx|_1\ge
\dd_\sfn (\sfu )+\dd_\sfn (\sfv) {-2} ,
\ee
if $\sfu,\sfx-\sfv\in\calY$ and $\sfu,\sfv\in \calH_{+,\sfn}$.

With \eqref{eq:Target2} and \eqref{eq:abs-vs-d} in mind define:
\be
\label{eq:rho-delta}
\bbP_{\beta , +}^{\sfh} (\sfu , \sfv ) = \sum_\ell
%\sum_{\uGamma\in \frR_\ell^{+ , \sfn} (\sfu, \sfv )}
\bbP_\beta^{\sfh}\lb
%\uGamma
\frR_\ell^{+ , \sfn} (\sfu, \sfv )
\rb  \ {\rm and}\
 \rho_\delta (\sfu , \sfv ) =
\frac{
{\rm e}^{-\delta\beta\dd_\sfn (\sfu )}
\bbG^{\sfh }_{\beta , +} (\sfu ,\sfv )
{\rm e}^{-\delta\beta\dd_\sfn (\sfv )}
}
{
{\rm e}^{\delta\beta \dd_\sfn (\sfu )
%\blue{
% +\abs{\sfu}_1
%}
}
\bbP_{\beta ,+}^{\sfh}\lb
\sfu , \sfv\rb
{\rm e}^{\delta\beta\dd_\sfn (\sfv )}.
}.
\ee
Then, 
\eqref{eq:Target1}  and hence the
{
upper bound
}
 of Theorem~\ref{thm:nopinning} are
consequence of:
\begin{theorem}
 \label{thm:TwoBounds}
There exist $0 < \delta < \nu_{\sfg}/4$ and  $\beta_0$ sufficiently large such that
{
the following holds:
For any $\beta >\beta_0$ there exists a constant $c = c (\beta )$, such that
}
\begin{enumerate}
\item [(A)] uniformly
in  ${\rm arg}\lb \sfn \rb \in \left[\frac{\pi}{2} , \frac{3\pi}{4}\right]$
and in all $\sfu , \sfv \in \calH_{+, \sfn}$, $\sfu,\sfx-\sfv\in\calY$,
\be
\label{eq:TwoBounds1}
\rho_\delta (\sfu , \sfv )\leq
c(\beta).
\ee
\item [(B)] uniformly in $\sfx\in \calB_{+ ,\sfn}$,
\be
\label{eq:TwoBounds2}
\suptwo{\sfu, \sfx  - \sfv\in \calY} {\sfu,\sfv\in \calH_{+,\sfn}}
%\sup_{\sfu , \sfx - \sfv\in\calY}
{\rm e}^{-\delta\beta\lb\abs{\sfu}_1 +\abs{\sfx - \sfv}_1\rb} 
\bbP_{\beta , +}^{\sfh_\sfx} (\sfu , \sfv )
%\\
%&
\leq c (\beta )
\,
G_\beta \lb\sfx \big| \calP_{+ ,\sfn }\rb e^{\tau_\beta(\sfx)} .
%\end{split}
\ee
\end{enumerate}
\end{theorem}
Claim (B)
is an expression of entropic repulsion, and it
has the same flavour as the
lower bound of
Theorem~\ref{thm:nopinning}. We prove both in Subsection~\ref{sec:lb}.

{In order to see how (A) and (B) imply
our target upper bound \eqref{eq:Target1}  notice that}
\be
\label{eq:Indeed}
\begin{split}
 & {\rm e}^{-{3} \delta \beta \lb \abs{\sfu }_1 +\abs{\sfx -\sfv }_1\rb}
 \bbG^{\sfh }_{\beta , +} (\sfu,\sfv)
\leq \rho_\delta (\sfu , \sfv ) {{\rm e}^{4\delta\beta}}
{\rm e}^{- \delta \beta \lb \abs{\sfu }_1 +\abs{\sfx -\sfv }_1\rb}
\bbP_{\beta , +}^{\sfh_\sfx} (\sfu , \sfv ) \\
&\qquad
\stackrel{\eqref{eq:TwoBounds1},\eqref{eq:TwoBounds2}}{ \leq}
 c(\beta)^2 {{\rm e}^{4\delta\beta}}\,
G_\beta \lb\sfx \big| \calP_{+ ,\sfn }\rb e^{\tau_\beta(\sfx)}
{
\df c_2 (\beta ) G_\beta \lb\sfx \big| \calP_{+ ,\sfn }\rb e^{\tau_\beta(\sfx)}
}
,
\end{split}
\ee
where the first inequality follows from \eqref{eq:abs-vs-d} and from the very definition
of $\rho_\delta$ in \eqref{eq:rho-delta},
whereas the second inequality is precisely
\eqref{eq:TwoBounds1}
and \eqref{eq:TwoBounds2}.
%
%\eqref{eq:Target1}, \eqref{eq:ves}, \eqref{eq:G-plus}
 The target bound \eqref{eq:Target1} follows from \eqref{eq:Indeed} because of  \eqref{eq:Target2}.
\qed

\section{Proof of Theorem~\ref{thm:TwoBounds} and the
lower bound of Theorem~\ref{thm:nopinning}}
\label{sec:Thm-TB}

{In this Section we prove  Theorem~\ref{thm:TwoBounds} and the lower bound
 in Theorem~\ref{thm:nopinning}.
 Claim (A) of Theorem~\ref{thm:TwoBounds} is the most difficult part and
the arguments  hinge upon  crucial estimates of  Proposition~\ref{lem:Plus-bounds}.
%and on Proposition~\ref{prop:P-plus-bounds}.
The proof of the latter is  relegated
to Section~\ref{sec:provaprop10}. Claim (B) of Theorem~\ref{thm:TwoBounds} and
of the lower bound in Theorem~\ref{thm:nopinning} are somewhat simpler statements.
The proof is based on Proposition~\ref{prop:P-plus-bounds}
(which is in its turn proved in Section~\ref{sec:provaprop10}) and, in the very end -
see \eqref{eq:K-adelta}, on Proposition~\ref{lem:Plus-bounds}.

The impact of the pinning potential is encoded in the recursion \eqref{eq:Gbound-1},
which, by taking maxima,  leads to the uniform recursion \eqref{eq:Recursion1}.
Proposition~\ref{lem:Plus-bounds} ensures that \eqref{eq:Recursion1} implies Claim (A)
of Theorem~\ref{thm:TwoBounds}.
In Section~\ref{sec:provaprop10} decay properties \eqref{eq:Psi-ub} of
potential $\Psi_\beta^{\sfh , \delta}$
 play an important role in the proof of Proposition~\ref{lem:Plus-bounds}.

Claim (B) of Theorem~\ref{thm:TwoBounds} and the lower bound of
Theorem~\ref{thm:nopinning} are statements about entropic repulsion
 of the effective random walk away from $\calB_{+, \sfn}$.
 Our approach is based on \cite{CIL}.
Key facts  along these lines
are formulated in Proposition~\ref{prop:LowerBounds}. Claim (B) of Theorem~\ref{thm:TwoBounds}
(in the form of \eqref{eq:ClaimB}) and lower bound of Theorem~\ref{thm:nopinning}
(in the form of \eqref{eq:lb-nopinning}) are easy consequences.
The proof of
Proposition~\ref{prop:P-plus-bounds}, which gives an upper bound on the left hand side
of \eqref{eq:TwoBounds2} in terms of $\bbP_{\beta , +}^{\sfh_\sfx} (0, \sfx )$,
is relegated to Subsection~\ref{sub:proofPx}.
}
\subsection{Claim (A)}
\label{sec:claim-a}
Let us start by making one remark:
\begin{remark}
\label{rem:Out-of-A}
By the first of
\eqref{eq:breakpoints} and by the crude upper bound \eqref{eq:q+-crude}
\[
\log \rho_{\delta} (\sfu , \sfv ) \leq
  r_0 \abs{\sfv -  \sfu}_1{\rm e}^{-{2}\chi\beta} - 2 \delta \beta
\lb \dd_\sfn (\sfu ) + \dd_\sfn (\sfv )\rb ,
\]
which means that
 $\rho_\delta (\sfu , \sfv ) \leq 1$ 
unless $\sfu$ and $\sfv$ stay appropriately close to $\calB_{+ ,\sfn }$ in the sense that
the pair  $\lb \sfu, \sfv\rb  \in \calA$, where
\be
\label{eq:crude-rho}
 \calA \df \lbr (\sfu , \sfv): \,{\sfu, \sfx  - \sfv\in \calY}, {\sfu,\sfv\in \calH_{+,\sfn}},
 \dd_\sfn (\sfu ) + \dd_\sfn (\sfv ) \leq  \frac{ r_0 {\rm e}^{-{2} \chi\beta}}{
 {2}
 \delta  \beta}
\abs{\sfv -  \sfu }_1 \rbr .
\ee
In particular, since $\chi >\frac{1}{2}$, we may assume that there exists $\nu >1$ such that
\be
\label{eq:nu-uv-cond}
|\sfv -\sfu |_1 \geq \lb \dd_\sfn (\sfu ) + \dd_\sfn (\sfv )\rb {\rm e}^{\nu\beta }\quad
\text{and}\quad |\arg (\sfv - \sfu ) - \arg (\sfn^\perp )| \leq {\rm e}^{-\nu\beta} ,
\ee
where $\sfn^\perp \df  (n_2 ,-n_1 )$.
There is no loss of generality (otherwise we would just consider the reversed
walk) to assume that $\sfn \cdot (\sfv - \sfu ) \geq 0$. This ensures that the drift $\sfh =
\nabla\tau_\beta (\sfv -\sfu )$ has non-negative entries.
\end{remark}
{  \qed} 

The proof   of the claim A    comprises several steps.

\noindent
\step{1} (Recursion)
Manipulating expansions of
\be
\label{eq:low-temp-pf}
%{\rm e}^{\sum \phi_\beta (\gm{i} )}
\prod_{i=1}^{\ell} {\rm e}^{\phi_\beta (\gm{i} )}
= \prod_{i=1}^{\ell}\lb 1 +\lb
{\rm e}^{\phi_\beta (\gm{i} )} - 1\rb\rb ,
\ee
 we infer that
\be
\label{eq:Gbound-1}
\begin{split}
 &\bbG^{\sfh }_{\beta , +}  (\sfu ,\sfv )  =
\bbP_{\beta , +}^{\sfh} (\sfu , \sfv )  +
\sum_{\sfw ,\sfz}\bbP_{\beta , +}^{\sfh} (\sfu , \sfw )
\Phi_{\beta , +}^\sfh (\sfw , \sfz)
\bbP_{\beta , +}^{\sfh} (\sfz , \sfv ) \\
&+ \quad
\sum_{\sfw_1 , \sfz_1}\sum_{\sfw_2 , \sfz_2}
%\sum_{\sfw ,\sfz}
\bbP_{\beta , +}^{\sfh} (\sfu , \sfw_1 )
\Phi_{\beta }^\sfh (\sfw_1 , \sfz_1)
\bbG^{\sfh }_{\beta , +}  (\sfz_1 ,\sfw_2)
\Phi_{\beta }^\sfh (\sfw_2 , \sfz_2)
\bbP_{\beta , +}^{\sfh} (\sfz_2 , \sfv ) ,
\end{split}
\ee
where we have defined:
\be
\label{eq:Quantities}
\Phi_{\beta }^\sfh (\sfw , \sfz ) =
%\sum_m
\bbE_\beta^{\sfh}\lb
\1_{
\sfA (\sfw ,\sfz )}\lb
{\rm e}^{\phi_\beta (\gamma )} - 1\rb \rb .
\ee
Equation \eqref{eq:Gbound-1} gives rise to the following recursion:
Set
\be
\label{eq:RecQuant-P}
\begin{split}
\underline{\bbP}_{\beta , +}^{\sfh , \delta}
(\sfs , \sft ) &=
{\rm e}^{-\delta\beta\dd_\sfn (\sfs )}
{\bbP}_{\beta , +}^{\sfh}
(\sfs , \sft )
{\rm e}^{-\delta\beta\dd_\sfn (\sft )}, \\
\bar{\bbP}_{\beta , +}^{\sfh , \delta}
%(\sfu , \sfv )
(\sfu , \sfv ) & =
{\rm e}^{\delta\beta \dd_\sfn (\sfu )
}
{\bbP}_{\beta , +}^{\sfh}
(\sfu , \sfv )
 {{\rm e}^   {\delta\beta\dd_\sfn (\sfv)}}
\end{split}
\ee
and
\be
\label{eq:RecQuant-Psi}
\Psi_{\beta }^{\sfh,\delta} (\sfw , \sfz ) =
%\sum_m
{\rm e}^{3\delta\beta\dd_\sfn (\sfw)}
\Phi_{\beta }^{\sfh} (\sfw , \sfz )
{\rm e}^{3\delta\beta\dd_\sfn (\sfz)}.
\ee
With this notation \eqref{eq:Gbound-1} and \eqref{eq:rho-delta} imply:
\be
\label{eq:Recursion}
\begin{split}
&\rho_{\delta} (\sfu , \sfv ) \leq {\rm e}^{-2\delta\beta \lb
\dd_\sfn (\sfu) + \dd_\sfn (\sfv)
\rb
}
 +
\sum_{\sfw ,\sfz}
\frac{
\underline{\bbP}_{\beta , +}^{\sfh , \delta} (\sfu , \sfw )
\Psi_{\beta }^{\sfh , \delta} (\sfw , \sfz)
\underline{\bbP}_{\beta , +}^{\sfh , \delta} (\sfz , \sfv)
}
{
\bar\bbP_{\beta , +}^{\sfh , \delta} (\sfu , \sfv )
}\\
&
+
\sum_{\sfw_1 , \sfz_1}\sum_{\sfw_2 , \sfz_2}
\frac{
\underline{\bbP}_{\beta , +}^{\sfh , \delta} (\sfu , \sfw_1 )
\Psi_{\beta}^{\sfh , \delta} (\sfw_1 , \sfz_1)
\underline{\bbP}_{\beta , +}^{\sfh , \delta} (\sfz_1 , \sfw_2 )
\Psi_{\beta }^{\sfh , \delta} (\sfw_2 , \sfz_2)
\underline{\bbP}_{\beta , +}^{\sfh , \delta} (\sfz_2 , \sfv )
}
{
\bar\bbP_{\beta , +}^{\sfh , \delta} (\sfu , \sfv )
}
\rho_\delta (\sfz_1 , \sfw_2 ) .
\end{split}
\ee

\begin{remark}
None of the ratios in  \eqref{eq:Recursion}
depends on the drift $\sfh$. It will be convenient to take $\sfh = \nabla
\tau_\beta (\sfv -\sfu )$.
\end{remark}

\noindent
\step{2} (Bounds on $\Psi_\beta^{\sfh,\delta}$)
Recall the definition of the weights $\phi_\beta$ in \eqref{eq:w-weights}. Then,
\be
\label{eq:Psi-delta}
%\begin{split}
 \Psi_\beta^{\sfh,\delta} (\sfw , \sfz )
{
\,
%\leq
= \,
}
{\rm e}^{
3\delta\beta \dd_\sfn (\sfw )} \bbE_{\beta}^{\sfh} \1_{\sfA (\sfw , \sfz )} \lb {\rm exp}\lb \abs{\gamma}{\rm e}^{-\beta\chi
\lb 2 + d (\sfw , \sfz )\rb} \rb -1\rb
{\rm e}^{ 3
\delta\beta \dd_\sfn (\sfz )} ,
%\end{split}
\ee
where the function $d (\sfw , \sfz )\geq 0$ was already defined in
 \eqref{eq:d-ell}:
\be
\label{eqdwz}
d (\sfw , \sfz ) =\min_{\sfy \in D (\sfw, \sfz )\cap\calH_{+,\sfn }\cap \bbZ_*^2} \lb \dd_\sfn (\sfy )+1\rb -2 .
\ee
\begin{lemma}
\label{lem:Psi-Bound}
There exist $\nu_2, \nu_3 >0$
{and $c_1, R <\infty$ such that for any
$\chip <\chi$ one can choose $\delta_0 >0$, such that,
}
uniformly in $\delta\leq \delta_0$, admissible pairs $\lbr\lb\sfw, \sfz\rb :
\sfz\in\sfw +\calY\rbr $
and $\beta$ large, the following holds:
\be
\label{eq:Psi-ub}
\Psi_{\beta , \delta}^\sfh (\sfw , \sfz ) \leq c_1\,
{\rm e}^{-2\chip \beta}
 \calK_\beta (\sfw , \sfz ),
\ee
where the kernel $\calK_\beta$ is given by
\be
\label{eq:K-kernel}
\calK_\beta (\sfw , \sfz  ) =
%\1_{\lbr \sfz\in\sfw +\calY\rbr}
\exp\left[{-\nu_3\beta \lb \dd_\sfn (\sfw )- R\rb_{+}} {-\nu_2\beta \abs{\sfz - \sfw}_1\1_{\abs{\sfz - \sfw}_1 >1}} {-\nu_3\beta \lb \dd_\sfn (\sfz )- R\rb_+ }\right].
\ee
\end{lemma}
{
\begin{remark}
 \label{rem:Rstrip}
Lemma~ states that $\Psi_{\beta , \delta}^\sfh (\sfw , \sfz )$
is at most of order $\exp(-2\chi'\beta)$ and that the kernel
$\calK_\beta (\sfw , \sfz )$
 decays exponentially
both
in $|\sfz-\sfw|_1$
and in the distances $\dd_\sfn (\sfw ) , \dd_\sfn (\sfz )$ from the wall.
In particular, $\sum_{\sfw , \sfz} \calK_\beta (\sfw , \sfz )$ is essentially a one dimensional
sum over lattice points inside a strip of width $R$ along $\partial \calH_{+ , \sfn }$.
\end{remark}
}
\paragraph{\bf Proof of Lemma~\ref{lem:Psi-Bound}}
By construction of  diamond shapes there exists $\alpha \in \bbN$, such that
\begin{equation}
\label{eq:alpha-ratio}
d (\sfw , \sfz ) \geq\frac{1}{2}\lb  \lb \dd_\sfn (\sfz ) - \alpha |\sfw -\sfz|_1\rb_+
+ \lb \dd_\sfn (\sfw ) - \alpha |\sfw -\sfz|_1\rb_+ \rb .
\end{equation}
Above $a_+ \df a\vee 0$. To facilitate  notation set
\be
\label{eq:fwz}
 f_{\sfw\sfz} (k) = k {\rm e}^{-2\chi\beta -\frac{\chi\beta}{2}\lb  \lb \dd_\sfn (\sfz )
- \alpha |\sfw -\sfz|_1\rb_+
+ \lb \dd_\sfn (\sfw ) - \alpha |\sfw -\sfz|_1\rb_+ \rb} \ee
Since $\abs{\sfX (\gamma )}\leq \abs{\gamma}$,
 and in view of \eqref{eq:alpha-ratio},
\be
\label{eq:step4-bound-Psi}
\begin{split}
\Psi_{\beta , \delta}^\sfh (\sfw , \sfz ) &\leq
{\rm e}^{3\delta\beta\lb \dd_\sfn (\sfw ) + \dd_\sfn (\sfz )\rb}
\sum_{k\geq \abs{\sfz -\sfw}_1}
\lb
{\rm e}^{f_{\sfw\sfz} (k)} - 1\rb
%{\rm exp}\lbr k {\rm e}^{-2\chi\beta - \chi\beta\lb \dd_\sfn (\sfz ) -\alpha k\rb_+}\rbr - 1\rb
\bbP_\beta^\sfh \lb \abs{\gamma} = k\rb \\
&\ \ \leq
{\rm e}^{3\delta\beta\lb \dd_\sfn (\sfw ) + \dd_\sfn (\sfz )\rb}
\sum_{k\geq \abs{\sfz -\sfw}_1}
f_{\sfw\sfz} (k) {\rm e}^{f_{\sfw\sfz} (k)} \bbP_\beta^\sfh \lb \abs{\gamma} = k\rb \\
&\ \ \leq
{\rm e}^{3\delta\beta\lb \dd_\sfn (\sfw ) + \dd_\sfn (\sfz )\rb}
\sum_{k\geq \abs{\sfz -\sfw}_1}
f_{\sfw\sfz} (k) {\rm e}^{f_{\sfw\sfz} (k) - 2\nu_{\sfg}\beta k\1_{k>1 }} \\
&\ \ \leq 2{\rm e}^{3\delta\beta\lb \dd_\sfn (\sfw ) + \dd_\sfn (\sfz )\rb} \sum_{k\geq \abs{\sfz -\sfw}_1}  f_{\sfw\sfz} (k) {\rm e}^{-\nu_{\sfg}\beta k\1_{k>1 }} .
\end{split}
\ee
In the last two  inequalities we relied on \eqref{eq:exp-tails} and on \eqref{eq:fwz}.

{Let us take a closer look at the definition  \eqref{eq:fwz} on $f_{\sfw\sfz} (k)$.
}
Recall that $\alpha, \nu_{\sfg} $ and $\chi$ are positive constants which
do not depend of $\beta$.
{Fix $\chip <\chi$ and $R >2\alpha$. Then, one can choose $\nu_2, \nu_3 >0$
and  $\delta_0 >0$ so small,  so that
for any $\delta\leq \delta_0$ and for any  non-negative numbers $a, b\geq 0$,
}
\be
\label{eq:new-param}
\begin{split}
&\nu_{\sfg} k\1_{k >1} +2\chi +\frac{\chi}{2}\lb \lb a-\alpha k\rb_+ + \lb b-\alpha k\rb_+\rb
-3\delta a {-3\delta b}\\
&\quad
\geq  \nu_2 k\1_{k >1} + 2\chip + \nu_3 \lb \lb a- R\rb_+ + \lb b- R\rb_+\rb .
\end{split}
\ee
Hence \eqref{eq:Psi-ub}.\qed
\smallskip

\noindent
\step{3} (Substitution and analysis of the Recursion \eqref{eq:Recursion})
As we noted in Remark~\ref{rem:Out-of-A} $\rho_\delta (\sfu , \sfv )\leq {1}$ for
$\lb\sfu , \sfv\rb\not\in\calA$.
Define
\[
 \rho_\delta = \max \{\sup_{(\sfu , \sfv)\in\calA} \rho_\delta (\sfu , \sfv ), {1} \} .
\]
Then \eqref{eq:Recursion} implies
\be
\label{eq:Recursion1}
\rho_\delta \leq 1 + a_\delta + b_\delta \rho_\delta ,
\ee
where
\be
\label{eq:a-delta}
a_\delta = \sup_{\lb\sfu , \sfv \rb\in\calA}
\sum_{\sfw ,\sfz}
\frac{
\underline{\bbP}_{\beta , +}^{\sfh , \delta} (\sfu , \sfw )
\Psi_{\beta }^{\sfh , \delta} (\sfw , \sfz)
\underline{\bbP}_{\beta , +}^{\sfh , \delta} (\sfz , \sfv)
}
{
\bar\bbP_{\beta , +}^{\sfh , \delta} (\sfu , \sfv )
} ,
\ee
and, accordingly,
\be
\label{eq:b-delta}
b_\delta = \sup_{\lb\sfu , \sfv \rb\in\calA}
\sum_{\sfw_1 , \sfz_1}\sum_{\sfw_2 , \sfz_2}
\frac{
\underline{\bbP}_{\beta , +}^{\sfh , \delta} (\sfu , \sfw_1 )
\Psi_{\beta}^{\sfh , \delta} (\sfw_1 , \sfz_1)
\underline{\bbP}_{\beta , +}^{\sfh , \delta} (\sfz_1 , \sfw_2 )
\Psi_{\beta }^{\sfh , \delta} (\sfw_2 , \sfz_2)
\underline{\bbP}_{\beta , +}^{\sfh , \delta} (\sfz_2 , \sfv )
}
{
\bar\bbP_{\beta , +}^{\sfh , \delta} (\sfu , \sfv )
}.
\ee
Our target bound \eqref{eq:TwoBounds1} follows then from {\eqref{eq:Recursion1} and:}
{
%In view of \eqref{eq:Recursion1} we need to check that
\begin{proposition}
\label{lem:Plus-bounds}
Fix $\delta >0$. Then
\be
\label{eq:Recursion-target}
a_\delta < \infty\ {\rm and}\ b_\delta <1 ,
\ee
uniformly in  $\sfn\in \left[\frac{\pi}{2} , \frac{3\pi}{4}\right]$,
 in all $\lb \sfu , \sfv \rb$ satisfying \eqref{eq:nu-uv-cond}
and
in all $\beta$ sufficiently large.
\end{proposition}
We prove Proposition~\ref{lem:Plus-bounds}  in Section \ref{sec:provaprop10}.
}

\subsection{Claim (B) of Theorem~\ref{thm:TwoBounds} and the lower bound in
  Theorem~\ref{thm:nopinning}}
\label{sec:lb}
First of all we may consider $\bbP_{\beta , +}^{\sfh_\sfx} (0, \sfx )$ instead of
the left hand side in \eqref{eq:TwoBounds2}. The proof of the following Proposition
is relegated to Subsection~\ref{sub:proofPx}:
\begin{proposition}
\label{prop:P-plus-bounds}
There exists $c_1 = c_1 (\beta )$ such that:
\be
\label{eq:ImplProp10}
\max\lbr
{\rm e}^{-\nu_{\sfg}\beta\abs{\sfx}_{{1}}},
\suptwo{\sfu, \sfx  - \sfv\in \calY} {\sfu,\sfv\in \calH_{+,\sfn}}
{\rm e}^{-\delta \beta\lb \abs{\sfu}_1 +\abs{\sfx -\sfv}_1\rb}
 \bbP_{\beta , +}^{\sfh_\sfx} (\sfu , \sfv )\rbr
  (\sfu , \sfv )
\leq c_1
\bbP_{\beta , +}^{\sfh_\sfx}  (0, \sfx ).
\ee
\end{proposition}
{
Thus, lower bounds
 for both $G_\beta \lb\sfx \big| \calP_{+ ,\sfn }\rb $ and
 $G_{\beta }^{+\sfn} (\sfx )$ may be derived in terms of
 $\bbP_{\beta , +}^{\sfh_\sfx}  (0, \sfx )$.

Given a string $\uGamma = \left[ \ugamma , \ucalC\right]$ of irreducible animals,
let us define (see \eqref{eq:w-weights})
\be
\label{eq:F-func}
F (\uGamma ) = F (\ugamma ) = \sum_\ell \abs{\gamma^{[\ell ]}} {\rm e}^{-\beta\chi (d_\ell +2 )}
= \sum_\ell \phi_\beta  (\gamma^{[\ell ]} ) .
\ee}
Both \eqref{eq:TwoBounds2}
and the lower bound in \eqref{eq:nopinning} are consequences of the
following proposition:
\begin{proposition}
 \label{prop:LowerBounds}
 For any $\beta\geq \beta_0$ there exist two constants $p_\beta >0$ and $K_\beta <\infty$
 such that the following two bounds hold  uniformly in ${\rm arg} (\sfn )\in
 \left[\frac{\pi}{2} ,\frac{3\pi}{4}\right]$ and $\sfx\in \calB_{+ , \sfn }$:
 \be
 \label{eq:ClaimB-cond}
\bbP_\beta^{\sfh_\sfx}
\lb \frA^{+, \sfn} (0, \sfx )~\Big|~
\frR^{+, \sfn} (0, \sfx )
 \rb  = \frac{\bbP_\beta^{\sfh_\sfx}\lb \frA^{+, \sfn} (0, \sfx )\rb}
 {\bbP_{\beta, +}^{\sfh_\sfx}\lb 0, \sfx \rb}
 \geq p_\beta ,
 \ee
 and,
 \be
 \label{eq:lb-exps}
 \bbE_\beta^{\sfh_\sfx}
\lb F (\ugamma )~\Big|~
\frR^{+, \sfn} (0, \sfx )
 \rb \leq K_\beta .
 \ee
\end{proposition}

{
Before proving Proposition~\ref{prop:LowerBounds} let us demonstrate
how it implies lower bounds in question:

Consider first  \eqref{eq:TwoBounds2}.
{ By
\eqref{eq:ImplProp10}
it would be enough}
to check  that there exists
a constant $c_2 = c_2 (\beta )$ such that
\be
\label{eq:ClaimB}
\bbP_{\beta , +}^{\sfh_\sfx} (0 , \sfx ) \leq
c_2
\,
G_\beta \lb\sfx \big| \calP_{+ ,\sfn }\rb e^{\tau_\beta(\sfx)} .
\ee
However,
\[
 G_\beta \lb\sfx \big| \calP_{+ ,\sfn }\rb e^{\tau_\beta(\sfx)} \geq \bbP_\beta^{\sfh_\sfx}
\lb \frA^{+,\sfn} (0, \sfx ) \rb .
\]
{
Indeed, the right hand side above is just a restricted sum over  animals with empty boundary pieces
in the irreducible decomposition \eqref{eq:irreducible-A}. By \eqref{eq:ClaimB-cond}
\[
 \bbP_\beta^{\sfh_\sfx}
\lb \frA^{+,\sfn} (0, \sfx ) \rb \geq
p_\beta \bbP_\beta^{\sfh_\sfx}
\lb \frR^{+,\sfn} (0, \sfx ) \rb
=p_\beta \bbP_{\beta , +}^{\sfh_\sfx} (0 , \sfx ) ,
\]
and
\eqref{eq:ClaimB} follows.

Turning to the lower bound in \eqref{eq:nopinning} note that by \eqref{eq:q+-crude}
\[
{\rm e}^{\tau_\beta (\sfx )} G_{\beta}^{+, \sfn} (\sfx ) \geq
\bbE_\beta^{\sfh_\sfx}\lbr {\rm e}^{- F (\ugamma )}\1_{\frA^{+,\sfn} (0, \sfx )}\rbr .
\]
If both \eqref{eq:ClaimB-cond} and \eqref{eq:lb-exps} hold,
then by  Markov inequality,
\[
 \bbP_\beta^{\sfh_\sfx}\lb F (\ugamma )\leq \frac{2K_\beta}{p_\beta} ; \frA^{+,\sfn} (0, \sfx )
 ~\Big| \frR^{+,\sfn} (0, \sfx ) \rb \geq \frac{p_\beta}{2} .
\]
This means that
\be
\label{eq:lb-nopinning}
{\rm e}^{\tau_\beta (\sfx )} G_{\beta}^{+, \sfn} (\sfx ) \geq
\frac{p_\beta}{2} {\rm e}^{- \frac{2K_\beta}{p_\beta }}
\bbP_\beta^{\sfh_\sfx}\lb \frR^{+,\sfn} (0, \sfx ) \rb
= \frac{p_\beta}{2} {\rm e}^{- \frac{2K_\beta}{p_\beta }}
\bbP_{\beta , +}^{\sfh_\sfx}\lb 0, \sfx  \rb
,
\ee
On the other hand by  \eqref{eq:exp-tails}
\[
 G_\beta \lb\sfx \big| \calP_{+ ,\sfn }\rb e^{\tau_\beta(\sfx)} \leq
 \sum_{\sfu , \sfv}
 {\rm e}^{-\nu_{\sfg}\beta\lb \abs{\sfu}_{{1}} +\abs{\sfx -\sfv}_{{1}}\rb}
 \bbP_{\beta , +}^{\sfh_\sfx} (\sfu , \sfv ) + \smof{{\rm e}^{-\nu_{\sfg}\beta\abs{\sfx}_{{1}}}} .
\]
In view of Proposition~\ref{prop:P-plus-bounds} (and \eqref{eq:ClaimB}) we conclude that
$G_\beta \lb\sfx \big| \calP_{+ ,\sfn }\rb e^{\tau_\beta(\sfx)}\eqvs
\bbP_{\beta , +}^{\sfh_\sfx} (0, \sfx ) $, and the lower bound \eqref{eq:nopinning}
indeed follows from \eqref{eq:lb-nopinning}. \qed
}
}

\begin{proof}[Proof of Proposition~\ref{prop:LowerBounds}]
The bound \eqref{eq:ClaimB-cond} has a transparent meaning:
it reflects entropic repulsion of the random walk $\lbr \sfR_\ell\rbr$ from $\calH_{+, \sfn}^c$
under  the conditional measures $\bbP_\beta^{\sfh_\sfx}
\lb \cdot ~\Big|~
\frR^{+, \sfn} (0, \sfx )
 \rb$. Recall
\eqref{eq:irreducible-A}
that {both events $\frA^{+, \sfn} (0, \sfx )$ and  $\frR^{+, \sfn} (0, \sfx )$}
are encoded in terms of words of irreducible animals
\be
\label{eq:Strings}
\uGamma = \Gm{1}\circ \dots \circ \Gm{m} ;\ \sfX (\uGamma )= \sfx ;\  m=1, 2, \dots
\ee
The event $\frR^{+, \sfn} (0, \sfx )$ contains all such words in \eqref{eq:Strings} for which
all the vertices $\sfR_\ell$ of the effective random walk \eqref{eq:Rwalk} belong to
$\calH_{+ ,\sfn}$. The event $\frA^{+, \sfn} (0, \sfx ) \subset \frR^{+, \sfn} (0, \sfx )$
is more restrictive: it requires that for any $\ell = 1, \dots , m$,
\be
\label{eq:gamma-cond}
\sfR_{\ell-1} + \gm{\ell}\subset \calH_{+ ,\sfn} .
\ee
Note that \eqref{eq:gamma-cond} is automatically satisfied whenever
(see \eqref{eq:D-shape})  $D_\ell \subset \calH_{+,\sfn }$.
For $k=1,2,\dots $
{
consider the following event:
}
%let us
%define two events $\frE^{+ ,\sfn}_{k , i} (0, \sfx )$; $i=1,2$. The first one is
\be
\label{frE-event1}
\frE^{+ ,\sfn}_{k} (0, \sfx ) = \bigcup_{m}\lb \frR^{+, \sfn}_m (0, \sfx )\cap
\lbr D_\ell\subset \calH_{+ ,\sfn}\ {\rm for}\ \ell = k, \dots , m-k\rbr\rb  .
\ee
By definition $\lbr D_\ell\subset \calH_{+ ,\sfn}\ {\rm for}\ \ell = k, \dots , m-k\rbr$
is a sure event whenever $k > m-k$.

A straightforward (and substantially simplified) modification
of the proof of Lemma~5.1 in \cite{CIL} implies that for all
$\beta $ sufficiently large
 there exists $k = k (\beta )$ such that
\be
\label{eq:k-choice}
\inf_{\sfx\in \calB_{+,\sfn }}\frac{
\bbP_\beta^{\sfh_\sfx}\lb \frE^{+ ,\sfn}_{k }(0, \sfx )\rb} {\bbP_{\beta, +}^{\sfh_\sfx }\lb 0, \sfx \rb}  > 0 .
\ee
Let us fix such $k$.
Consider the identity
{
\be
\label{eq:C-k-1sum}
 \bbP_\beta^\sfh \lb \frE^{+ ,\sfn}_{k }(0, \sfx )\rb
  =
 \sum_{\sfu , \sfv} \bbP_{\beta,+}^\sfh\lb \sfR_k = \sfu \rb
 \bbP_{\beta }^\sfh \lb \frA^{+ ,\sfn}( \sfu , \sfv )\rb
 \bbP_{\beta,+}^\sfh\lb \sfR_k = \sfx - \sfv \rb .
\ee
}
{
Since, $\bbP_{\beta }^\sfh \lb \frA^{+ ,\sfn}( \sfu , \sfv )\rb
\leq \bbP_{\beta , +}^\sfh ( \sfu , \sfv )$,
{
using
exponential tail estimates \eqref{eq:exp-tails} to control
$\bbP_\beta^\sfh\lb \sfR_k = \sfu \rb$ and $\bbP_\beta^\sfh\lb \sfR_k = \sfx - \sfv \rb$
 and
 \eqref{eq:ImplProp10},
}
 one infers that
 there exists $N_k (\beta )$ such that all the terms in
 \eqref{eq:C-k-1sum} which violate
 \be
 \label{eq:Rk-range}
 \abs{\sfu}_1,\, \abs{\sfx-\sfv}_1 < N_k
 \ee
 might be ignored.
 Precisely,  there exists $c_3 = c_3  (\beta ) > 0$, such that
 \be
 \label{eq:k-sum-bound}
 \sumtwo{\sfu , \sfv}{\abs{\sfu}_1, \abs{\sfx-\sfv}_1\leq N_k}
 \bbP_{\beta,+}^\sfh\lb \sfR_k = \sfu \rb
 \bbP_{\beta }^\sfh \lb \frA^{+ ,\sfn}( \sfu , \sfv )\rb
 \bbP_{\beta,+}^\sfh\lb \sfR_k = \sfx - \sfv \rb
 \geq c_3  \bbP_{\beta, +}^\sfh\lb 0, \sfx \rb ,
 \ee
 uniformly in $\sfx\in\calB_{+, \sfn}$ large.
}  On the other hand,
{
\be
\label{eq:rk-bound+}
\bbP_\beta^\sfh\lb \frA^{+, \sfn} (0, \sfu )\rb \ \text{and}\
\bbP_\beta^\sfh\lb \frA^{+, \sfn} (\sfv, \sfx )\rb
> 0,
\ee
for any $\sfu\in \calH_{+, \sfn}\cap\calY$ and $\sfv\in\calH_{+, \sfn}\cap\lb \sfx - \calY\rb$.
Hence, by \eqref{eq:k-sum-bound} there exists $c_4 (\beta ) >0$ such that
\be
 \label{eq:k-sum-bound-1}
 \sumtwo{\sfu , \sfv}{\abs{\sfu}_1, \abs{\sfx-\sfv}_1\leq N_k}
 %\bbP_\beta^\sfh\lb \sfR_k = \sfu \rb
 \bbP_\beta^\sfh\lb \frA^{+, \sfn} (0, \sfu )\rb
 \bbP_{\beta }^\sfh \lb \frA^{+ ,\sfn}( \sfu , \sfv )\rb
 %\bbP_\beta^\sfh\lb \sfR_k = \sfx - \sfv \rb
 \bbP_\beta^\sfh\lb \frA^{+, \sfn} (\sfv, \sfx )\rb
 \geq c_4  \bbP_{\beta, +}^\sfh\lb 0, \sfx \rb ,
 \ee
 uniformly in $\sfx\in\calB_{+, \sfn}$ large.
  Each term in $\bbP_{\beta }^\sfh \lb \frA^{+ ,\sfn}( 0 , \sfx )\rb $
 is overcounted at most $c_5 N_k^4$ times on the left hand side of \eqref{eq:k-sum-bound-1}.
 The inequality  \eqref{eq:ClaimB-cond} follows.
}
{
Let us turn to \eqref{eq:lb-exps}. Rewrite
\be
\label{eq:F-form}
F (\ugamma ) = \sum_\ell\sum_{\sfw , \sfz}  \phi_\beta (\gamma^{[\ell]} )
\1_{\lbr \sfR_{\ell -1} = \sfw , \sfR_\ell =\sfz\rbr }.
\ee
Hence,
\[
\bbE_\beta^{\sfh_\sfx }\lb F (\ugamma )\1_{\frR^{+ ,\sfn} (0, \sfx )}\rb =
\sum_{\sfw , \sfz}\bbP_{\beta , +}^{\sfh} (0, \sfw )\lb \bbE_\beta^\sfh
\1_{\sfA (\sfw , \sfz )}\phi_\beta (\gamma )\rb
\bbP_{\beta , +}^{\sfh} (\sfz, \sfx ).
\]
Since $\phi_\beta \leq {\rm e}^{\phi_\beta} - 1 $, a comparison with \eqref{eq:Psi-delta}
and with the right hand side of \eqref{eq:a-delta}  reveals that
\be
\label{eq:K-adelta}
\bbE_\beta^{\sfh_\sfx}
\lb F (\ugamma )~\Big|~
\frR^{+, \sfn} (0, \sfx )
 \rb \leq a_\delta
 {{\rm e}^{2\delta \beta} \df K_\beta ,
} \ee
}
{so our claim will follow once we prove Proposition~\ref{lem:Plus-bounds}
in Section \ref{sec:provaprop10}.}
\end{proof}

\section{Fluctuation  and Alili-Doney estimates}
\label{sec:A-D-RW}
Recall that in order to complete the proof of Theorem~\ref{thm:nopinning} it
remains  to verify the claims of  Proposition~\ref{lem:Plus-bounds} and
Proposition~\ref{prop:P-plus-bounds}.

At this stage we need to take a closer look at the local properties of the effective  walk
$\sfR_k$ defined in \eqref{eq:Rwalk}. In the sequel we shall restrict attention to
${\rm arg} (\sfx ) \in [0,2\pi /5]$. We shall represent
$\sfx = \abs{\sfx}_1\lb 1-\epsilon , \epsilon\rb$ and, accordingly, write
$\sfh_{\sfx}=\sfh_\epsilon$. If  ${\rm arg} (\sfx ) \in [0,2\pi /5]$
than the effective random walk has three
basic steps \eqref{eq:BasicProb}.
The rest of the steps satisfy \eqref{eq:non-basic}. This assertion is explained in
 Subsection~\ref{sub:steps}. Furthermore, sharp asymptotic description of $a_\epsilon$
 and $b_\epsilon$ are formulated in Proposition~\ref{prop:aep-bep}.

 Subsection~\ref{sub:Rk} is devoted to the proof of uniform local asymptotics
 of Proposition~\ref{prop:LB-P}. Note that
 \eqref{eq:LB-P} is valid
on all scales (sizes of $\sfx$) and as such goes beyond usual asymptotic
form of the local CLT. For instance,
if $\epsilon$ is small (horizontal or almost horizontal wall) and
if ${\rm e}^{-b_\epsilon}\abs{\sfx}_1 \leq 1$, then
the statistics of steps $\sfy\neq \sfe_1$
of the effective random walk from $0$ to $\sfx$
follow Poissonian asymptotics (as
$\beta\to\infty$). Gaussian asymptotics start to carry over only when
${\rm e}^{-b_\epsilon}\abs{\sfx}_1 \gg 1$, and there is an intermediate range of values
of $\abs{\sfx}_1$ when one should interpolate between these two regimes. In
Subsection~\ref{sub:Rk} we introduce a representation \eqref{eq:Rk-repr} of the
effective random walk $\sfR_k$ which makes this heuristics mathematically tractable:
The first term in \eqref{eq:Rk-repr} is a random staircase, whereas the second term
is a diluted
random sum of (uniformly - see Lemma~\ref{lem:V-walk} where this is quantified)
non-degenerate $\calY$-valued random variables.

In Subsection~\ref{sub:AD} we derive crucial
bounds on $\bbP_{\beta , +}^\sfh \lb \sfw , \sfz \rb$
for effective random walks which
are constrained stay above the wall. In view of decomposition
\eqref{eq:Ruv-decomp} one needs to study quantities
$\bbP_{\beta , +}^\sfh \lb 0 , \sfv  \rb$ and $\hat\bbP_{\beta , +}^\sfh \lb \sfv , 0  \rb$,
see the definition \eqref{eq:Heights-vrsRplus} in terms of ladder variables.
At this stage we rely on the adjustment \cite{CIL} of the Alili-Doney \cite{AD}
representation
formulas \eqref{eq:AD-expBound}.

We use \eqref{eq:AD-expBound} for deriving lower bounds in Subsection~\ref{sub:lower}.
The rest of Subsection~\ref{sub:AD} is devoted to upper bounds which are based on
H\"{o}lder inequalities  \eqref{eq:Holder-AD-OZ} and
\eqref{eq:Holder-AD-OZ-d}  (in Section~\ref{sec:provaprop10} it will be enough to
use Cauchy-Schwarz). The required bounds on expected number of ladder heights are derived
in Lemmas \ref{lem:Np-k}-\ref{lem:Nminus}.

\subsection{Low temperature structure of $\pKb$, $\tau_\beta$ and $\bbP_\beta^{\sfh }$.}
\label{sub:steps}
Recall the notation:
$\sfh = \sfh_{\sfx} = \nabla\tau_\beta (\sfx )$. Probability measures
$\bbP_\beta^{\sfh}$ are
defined on the very same set of irreducible animals $\frA$, regardless of our
running choice
of $\sfn$ and $\sfx$ and, accordingly, of $\sfh$.

Consider two
elementary irreducible animals $\Gamma_i = [\sfe_i , \emptyset]$; $i=1,2$. For
$\sfh = (h_1 , h_2 )\in \calQ_+\cap\pKb$ their $\bbP_\beta^{\sfh }$-probabilities
are given by:
\[
 \bbP_\beta^{\sfh} (\Gamma_i ) = {\rm e}^{-\beta +h_i },
\]
which means that $0\leq h_i \leq \beta$. For $\sfx \in \calQ_+$ and $\sfh = \sfh_\sfx$
it would be therefore convenient to define $a_\sfx (\beta ) = \beta - h_1$ and
$b_\sfx (\beta ) = \beta - h_2$. By the above,
\be
\label{eq:ax-bx}
0\leq a_\sfx , b_\sfx
\leq \beta. \ee
In this notation the probabilities of $\Gamma_i$ are recorded as
\be
\label{eq:ProbG1G2}
\bbP_\beta^{\sfh_\sfx} (\Gamma_1 ) = {\rm e}^{- a_\sfx (\beta ) }\quad\text{and}
\quad
\bbP_\beta^{\sfh_\sfx} (\Gamma_2) = {\rm e}^{- b_\sfx (\beta ) }.
\ee
We are going to derive asymptotic description
of $\bbP_\beta^{\sfh_\sfx} \lb 0, \sfx \rb =
\sum_\ell
\bbP_\beta^{\sfh}\lb
\frA_\ell (0, \sfx )
\rb$.
In order to formulate it  we shall employ the following %convenient
asymptotic notation:
\begin{definition}
 \label{def:eqvs}
 Let us say that two sequences $\lbr\phi (n)\rbr$ and $\lbr\psi (n)\rbr$ satisfy
 $\phi (n )\eqvs \psi (n)$ {\em uniformly} in $n\in A$ if there exists a positive
 constant $c\geq 1$ such that
 \[
  \frac{1}{c}\phi (n ) \leq \psi (n) \leq c\phi (n )
 \]
for all $n\in A\subseteq \bbN$. The same convention applies for notation $\phi (\sfx )\eqvs
\psi (\sfx )$ uniformly in $\sfx \in A\subseteq \bbZ^2$.
\end{definition}
The principal result of the forthcoming  Subsection~\ref{sub:Rk} is:
\begin{proposition}
\label{prop:LB-P}
The following asymptotic relation holds uniformly
in $\beta$ large and ${\rm arg}\lb \sfx\rb \in [0, \frac{2\pi}{5}]$:
\be
\label{eq:LB-P}
\bbP_\beta^{\sfh_\sfx} \lb 0, \sfx \rb
{
\df
\bbP_\beta^{\sfh_\sfx}\lb
\frA  (0, \sfx )
\rb} =
\sum_\ell
\bbP_\beta^{\sfh_\sfx}\lb
\frA_\ell (0, \sfx )
\rb
 \eqvs \frac{1}{\sqrt{{\rm e}^{-b_\sfx (\beta ) }\abs{\sfx}_1}\vee 1} ,
\ee
\end{proposition}
The proof of \eqref{eq:LB-P} is based on a careful analysis of asymptotics of $a_\sfx$ and
$b_\sfx$, particularly for $\sfx$-s close to the horizontal axis. Since
$\sfh_\sfx = \nabla\tau_\beta (\sfx )$ depends only on the direction of $\sfx$, it would be convenient
to consider $|\cdot|_1$-normalized versions  of various  $\sfx$ with
${\rm arg}\lb \sfx\rb \in [0, \frac{2\pi}{5}]$, or more generally of $\sfx\in\calQ_+$.
%which could be recorded as $\lb 1-\epsilon , \epsilon \rb$ for $\epsilon \in [0,1]$. Accordingly,
Below, if $\frac{\sfx}{|\sfx|_1} = (1-\epsilon , \epsilon )$,
we
shall use
notation $a_\epsilon \df a_\sfx $ and $b_\epsilon \df b_\sfx$.

\noindent
Here is the main result of the current subsection:
\begin{proposition}
 \label{prop:aep-bep}
 The following asymptotic relations hold uniformly in ${\rm
   arg}\lb \sfx\rb\in [0,
 \frac{2\pi}{5}]$
and
 $\beta$ sufficiently large:
\be
\label{eq:aeps}
c_1 \max\lbr \epsilon ,  {\rm e}^{-\beta}\rbr \leq a_\epsilon \leq
c_2\max\lbr \epsilon ,  {\rm e}^{-\beta} \rbr ,
\ee
As far as asymptotics of $b_\epsilon$ are considered:
%Set $u_\epsilon \df  \beta - b_\epsilon$. \blue{  This was $h_2$ before.} \newline
 If
$\epsilon \geq{\rm e}^{-2\beta}$, then
\be
\label{eq:uep-cases}
\begin{cases}
 \beta - b_\epsilon  \in \left[ \beta +\log\epsilon +
 %\bigof{
 c_3{\rm e}^{-\beta}
 %}
 ,
 \beta +\log\epsilon + \frac{c_4}{{\epsilon {\rm e}^{2\beta}}}\right] ,
 \quad &\text{if $
{
 \epsilon\geq 2{\rm e}^{-\beta}
}
 $}\\
\beta -b_\epsilon \eqvs \epsilon{\rm e}^{\beta}, \quad &\text{if ${\rm e}^{-2\beta}\leq
 \epsilon< 2{\rm e}^{-\beta}$
} \end{cases}
.
\ee
{If $\epsilon <  {\rm e}^{-2\beta}$ }, then
\be
\label{eq:uep-upper}
 0\leq
 %u_\epsilon =
 \beta - b_\epsilon
\leq c_5
%\sqrt{
{ {\rm e}^{-\beta}}\qquad   .
\ee
 \end{proposition}
\begin{proof}[Proof of Proposition~\ref{prop:aep-bep}] Let us start with
considerations which apply for all $\sfx\in\calQ_+$, or, equivalently, for any $\epsilon\in [0,1]$.
As it was already noticed in  \eqref{eq:ax-bx}, $0\leq a_\epsilon, b_\epsilon \leq \beta$. By convexity
and axis symmetries of the Wulff shape, $a_\epsilon$ is non-increasing in $\epsilon$, whereas $b_\epsilon$
is non-decreasing.
\newline
Next, by \eqref{eq:exp-tails} there exists $R>0$ such that uniformly in
$\beta$ large,
\[
\sum_{\Gamma\in\frA} \bbP_\beta^{\sfh_\sfx} (\Gamma ) \abs{\sfX (\Gamma )}^2
\1_{\lbr \abs{\sfX (\Gamma )} >R\rbr}
= \bigof{{\rm e}^{-2\beta }} .
\]
The sum $\sum_{\Gamma_i\in\frA} \bbP_\beta^{\sfh_\sfx} (\Gamma _i) =1.$
The contribution to it from all irreducible animals $\Gamma = [\gamma , \ucalC ]$ with
$\abs{\sfX (\gamma )}\leq R$ and non-empty decoration $\ucalC $ is
$\bigof{{\rm e}^{-{2}\beta }}$. It remains to consider the contributions of
irreducible paths $\gamma$
with {empty decorations and} $\abs{\sfX (\gamma )}\leq R$. The $\bbP_\beta^{\sfh_\sfx }$
probabilities
of the latter are given by
\[
 \bbP_\beta^{\sfh_\sfx }\lb [\gamma , \emptyset ]\rb = {\rm e}^{-\beta\abs{\gamma}
 +\sfh_\sfx \cdot \sfX (\gamma )} .
\]
By \eqref{eq:ax-bx} any  path $\gamma$  which contains a backtrack, that
is either both $\pm\sfe_1$ steps or both $\pm\sfe_2$ steps contributes
at most $\bigof{{\rm e}^{-2\beta }}$. Paths which contain only forward $\sfe_1$ and
$\sfe_2$ steps  and have at least two bonds   are reducible.
 Paths
which contain at least two backward steps from $\lbr -\sfe_1 ,-\sfe_2\rbr$ also
contribute at most $\bigof{{\rm e}^{-2\beta }}$. There are only two
staircase paths left (see Figure~\ref{fig:Animals} in Section~\ref{sec:A-D-RW}) :
\[
 \gamma_3 =
 \lb\sfe_1 , 2\sfe_1, 2\sfe_1 - \sfe_2 , 3\sfe_1 - \sfe_2, 4\sfe_1 - \sfe_2\rb\
 {\rm and}\
 \gamma_4 =
 \lb\sfe_2 , 2\sfe_2, 2\sfe_2 - \sfe_1 , 3\sfe_2 - \sfe_1, 4\sfe_2 - \sfe_1\rb .
\]
%%%%%%%%%%%%%%%%%%%%%%%%%%%%%%%%%
\begin{figure}[t]
\begin{center}
\includegraphics[width=0.9\textwidth]{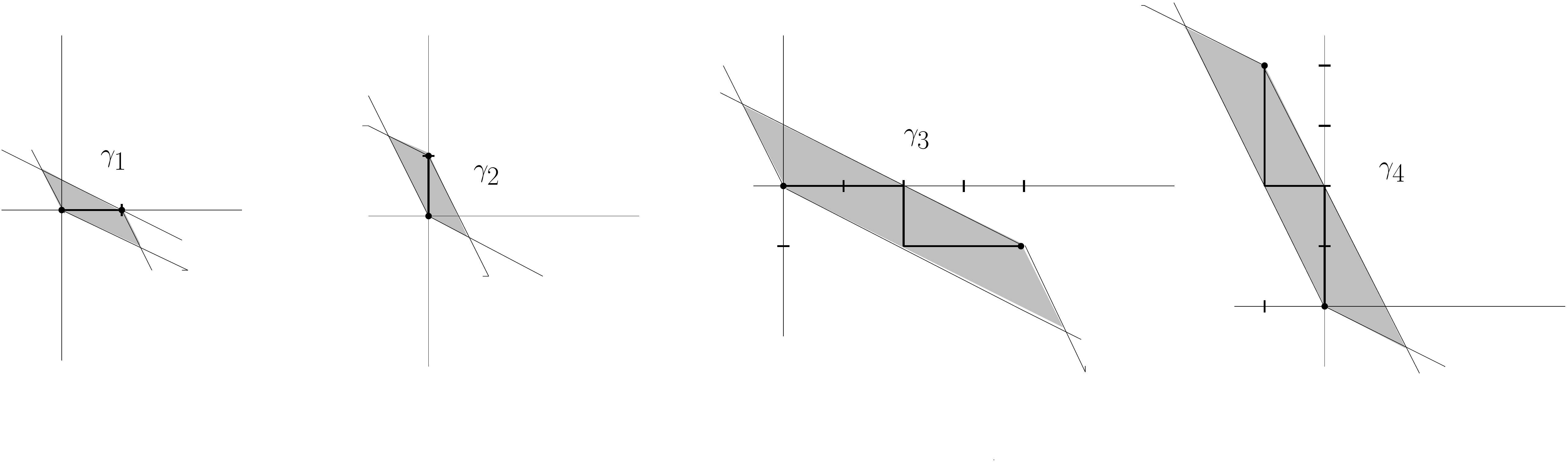}
\end{center}
\caption{Irreducible animals
$\Gamma_1, \Gamma_2, \Gamma_3, \Gamma_4$; $\Gamma_i = [\gamma_i , \emptyset ]$.}
\label{fig:Animals}
\end{figure}
%%%%%%%%%%%%%%%%%%%%%%%%%%%%%%%%
Define
\be
\label{eq:Delta-x-ab}
\Delta^a_\epsilon (\beta ) = 4a_\epsilon + (\beta - b_\epsilon )\ {\rm and}\
\Delta^b_\epsilon (\beta ) = 4b_\epsilon + (\beta - a_\epsilon )
\ee
By \eqref{eq:ax-bx}, both $\Delta^a_\epsilon, \Delta^b_\epsilon \geq 0$.
For $i=3,4$, the probabilities of $\Gamma_i = [\gamma_i , \emptyset]$
are given by
\be
\label{eq:ProbG3G4}
\bbP_\beta^{\sfh_\sfx} (\Gamma_3 ) = {\rm e}^{- \beta - \Delta^a_\epsilon (\beta )} \ {\rm and}\
\bbP_\beta^{\sfh_\sfx} (\Gamma_4 ) = {\rm e}^{- \beta - \Delta^b_\epsilon (\beta )} .
\ee
We conclude:
\be
\label{eq:G1-4Pcobtr}
{\rm e}^{-a_\epsilon} + {\rm e}^{-b_\epsilon} + {\rm e}^{- \beta - \Delta^a_ \epsilon(\beta )} + {\rm e}^{- \beta - \Delta^b_\epsilon (\beta )} =
1 -\bigof{{\rm e}^{-2\beta}}
\ee
and
\be
\label{eq:G1-4Ecobtr}
\begin{split}
&\lb {\rm e}^{-a_\epsilon} + 4{\rm e}^{- \beta - \Delta^a_\epsilon (\beta )} -
{\rm e}^{- \beta - \Delta^b_\epsilon(\beta )} ,
{\rm e}^{-b_\epsilon} + 4{\rm e}^{- \beta - \Delta^b_\epsilon (\beta )} -
{\rm e}^{- \beta - \Delta^a_\epsilon (\beta ) }\rb \\
&\quad = \bbE_\beta^{\sfh_\sfx} \sfX (\Gamma ) +\bigof{{\rm e}^{-2\beta}} \df
\sfv^*_\epsilon  (\beta ) +\bigof{{\rm e}^{-2\beta}}.
\end{split}
\ee
Recall that $\sfv^*_\epsilon = \abs{\sfv^*_\epsilon }_1 (1-\epsilon , \epsilon )$. By
\eqref{eq:G1-4Pcobtr} and \eqref{eq:G1-4Ecobtr}, $\abs{\sfv^*_\epsilon }_1 \geq 1$.
\smallskip

From now on let us consider
{
$\arg \lb (1-\epsilon , \epsilon ) \rb \in [0, \frac{2\pi}{5}]$.
}
In this case {the inspection of the first coordinate of the vector} \eqref{eq:G1-4Ecobtr} (for the horizontal component) readily implies that $a_\epsilon\leq c_5$
uniformly in $\beta$ large.
and, setting $\Delta_\epsilon = \Delta_\epsilon^a (\beta)= 
4a_\epsilon + (\beta - b_\epsilon )
$
(see \eqref{eq:Delta-x-ab}), we conclude:
Uniformly in
{
$\arg \lb (1-\epsilon , \epsilon ) \rb \in [0, \frac{2\pi}{5}]$.
}
and $\beta$ large,
\be
\label{eq:G1-New}
\begin{split}
&{\rm e}^{-a_\epsilon} + {\rm e}^{-b_\epsilon} + {\rm e}^{- \beta - \Delta_\epsilon} = 1 - \bigof{{\rm e}^{-2\beta}}\\ &\qquad \text{and}\\
&\lb {\rm e}^{-a_\epsilon} + 4{\rm e}^{- \beta - \Delta_\epsilon  },
{\rm e}^{-b_\epsilon}  -
{\rm e}^{- \beta - \Delta_\epsilon  }\rb = \abs{\sfv^*_\epsilon  }_1 \lb 1-\epsilon , \epsilon\rb
+\bigof{{\rm e}^{-2\beta}}.
\end{split}
\ee
\paragraph{\bf {Proof of \eqref{eq:aeps}}}
By \eqref{eq:ax-bx}, $b_\epsilon \leq \beta$, the first of
 \eqref{eq:G1-New} implies that
 $a_\epsilon \geq c_6{\rm e }^{-\beta}$ for any
{
$\epsilon$ in question.
}
 %$\epsilon\in [0,\sin\frac{2\pi}{5}]$.
 Next,
 since by both of \eqref{eq:G1-New},
 \be
 \label{eq:sv-star}
 \abs{\sfv^*_\epsilon }_1 =1 + \bigof{{\rm e}^{-\beta -\Delta_\epsilon}}
 ,
 \ee
 the second of \eqref{eq:G1-New} (for the horizontal coordinate)  implies that $a_\epsilon \eqvs \epsilon$,
uniformly in
$\epsilon\geq c_7 {\rm e}^{-\beta}$. Since $a_\epsilon$ is monotone non-decreasing in
$\epsilon$,
{ this implies that
$a_\epsilon \leq c_8 {\rm e}^{-\beta}$ for all $\epsilon\leq  c_7 {\rm e}^{-\beta}$, and
}
the first claim \eqref{eq:aeps} of Proposition~\ref{prop:aep-bep}  follows.
\smallskip

\noindent
\paragraph{\bf {Proof of \eqref{eq:uep-cases} and \eqref{eq:uep-upper}}}
Consider now the second of \eqref{eq:G1-New} (for the vertical coordinate). In view of
\eqref{eq:sv-star}, and after multiplying both sides by ${\rm e}^{\beta}$,
it reads (recall that $\Delta_\epsilon = 4a_\epsilon +
(\beta -b_\epsilon ) > \beta -b_\epsilon
$):
\be
\label{eq:u-eps}
{\rm e}^{(\beta -b_\epsilon )} - {\rm e}^{- (\beta -b_\epsilon )} \leq
{\rm e}^{(\beta -b_\epsilon )} - {\rm e}^{-\Delta_\epsilon} =
\epsilon{\rm e}^\beta + \bigof{\epsilon
{\rm e}^{-\Delta_\epsilon} +{\rm e}^{-\beta}} \leq
{\rm e}^{(\beta -b_\epsilon )}.
\ee
If $\epsilon {\rm e}^{\beta}\geq 2$, then $\bigof{\epsilon
{\rm e}^{-\Delta_\epsilon} +{\rm e}^{-\beta}}/(\epsilon {\rm e}^\beta) =
\bigof{ {\frac{1}{\epsilon {\rm e}^{2\beta }}}}$.
Hence, the first of \eqref{eq:uep-cases}.
\newline
Furthermore, since $
\beta -b_\epsilon$ is non-{increasing} and
non-negative, and since $a_\epsilon$ is non-negative and uniformly bounded,
\[
 {\rm e}^{
 \beta -b_\epsilon
} - {\rm e}^{-\Delta_\epsilon} = {\rm e}^{
 %u_\epsilon
 \beta -b_\epsilon}
 - {\rm e}^{-
 %u_\epsilon
 (\beta -b_\epsilon ) - 4a_\epsilon}
 \eqvs
 %u_\epsilon
 (\beta -b_\epsilon ) + a_\epsilon ,
\]
uniformly in $\epsilon\in [0, 2{\rm e}^{-\beta}]$ and $\beta$ large. Hence,
by \eqref{eq:u-eps},
\be
\label{eq:u-eps-small}
%u_\epsilon
(\beta -b_\epsilon)  + a_\epsilon \eqvs \epsilon{\rm e}^\beta + \bigof{\epsilon
{\rm e}^{-\Delta_\epsilon} +{\rm e}^{-\beta}} ,
\ee
also uniformly in $\epsilon\in [0, 2{\rm e}^{-\beta}]$ and $\beta$ large. The
asymptotic behavior of $a_\epsilon$ is already verified \eqref{eq:aeps}. Both the second
of \eqref{eq:uep-cases} and the upper bound \eqref{eq:uep-upper} follow.
\end{proof}
 \begin{remark}
 Consider $\sfx = \abs{\sfx}_1 \lb 1-\epsilon , \epsilon\rb$,
Since $\tau_\beta (\sfx ) = \sfh_\sfx\cdot\sfx = \beta\abs{\sfx }_1 - \lb a_\sfx , b_\sfx\rb\cdot
\sfx$, the asymptotics of surface tension $\tau_\beta$ are given by:
\be
\label{eq:tau-approx}
0\leq   \abs{\sfx}_1 -
\frac{\tau_\beta (\sfx )}{\beta}
 =
{\abs{\sfx}_1}\frac{(1-\epsilon)a_\epsilon + \epsilon b_\epsilon}{\beta} ,
\ee
where $a_\epsilon$ and $b_\epsilon
$ comply with asymptotic relations
 \eqref{eq:aeps}, \eqref{eq:uep-cases} and \eqref{eq:uep-upper}
uniformly in $\beta$ large and
$\sfx\in \calQ_+\cap\lbr {{\rm arg} (\sfx ) \in [0,2\pi /5]}\rbr$.
{
In particular,
the rescaled Wulff shape  $(1/\beta) \Kb$ tends to the square $Q=[-1,+1]^2$ in
Hausdorff distance, as $\beta\to\infty$, and the boundary of
$(1/\beta) \Kb$ is at
Hausdorff distance $\bigof{\frac{1}{\beta}}$ from $\partial Q$.
Sharper asymptotics could be read from Proposition~\ref{prop:aep-bep}, in particular
the boundary of $(1/\beta) \Kb$ is within distance
$\bigof{\frac{{\rm e}^{-\beta}}{\beta}}$ from $\partial Q$ along  axis directions.
}
\end{remark}

\subsection{Decomposition of $\sfR_k$ and proof of Proposition~\ref{prop:LB-P}.}
\label{sub:Rk}
The effective random walk $\sfR_k$ was defined in \eqref{eq:Rwalk}.
{
We  summarize  computations of Subsection~\ref{sub:steps} as follows:
\begin{definition}
Define the  set of basic steps as $\frS_0 = \lbr \sfe_1,\sfe_2,
4\sfe_1 - \sfe_2\rbr$.
\end{definition}
The probabilities of three basic steps
are given by
\be
\label{eq:BasicProb}
%\begin{split}
\bbP_\beta^{\sfh_{\sfx}}\lb
\sfX = \sfe_1
%\Gamma_1
\rb = {\rm e}^{-a_\sfx (\beta )},\
\bbP_\beta^{\sfh_{\sfx}}\lb
\sfX = \sfe_2
%\Gamma_2
\rb = {\rm e}^{-b_\sfx  (\beta ) }\
{\rm and}\
\bbP_\beta^{\sfh_{\sfx}}\lb
\sfX = 4\sfe_1 -\sfe_2
%\Gamma_3
\rb
= {\rm e}^{-\beta - \Delta_\sfx (\beta )}.
\ee
The coefficients $a_\sfx, b_\sfx$ and $\Delta_\sfx = 4a_\sfx +(\beta -b_\sfx)$ satisfy asymptotic
relations \eqref{eq:aeps}, \eqref{eq:uep-cases} and \eqref{eq:uep-upper}.
\newline
Non-basic steps
do not contribute in the following sense:
\be
\label{eq:non-basic}
\sup_{{\rm arg} (\sfx )\in [0,2\pi /5]}
 \sum_{\sfy\not\in\frS_0}\abs{\sfy}^2 \bbP_\beta^{\sfh_{\sfx}}\lb \sfX = \sfy \rb =
\bigof{ {\rm e}^{-2\beta}}  .
\ee
\begin{remark}
\label{eq:remGamma3}
Note that for fixed $\epsilon > 0$
{ and  for any $\sfx$ with $\frac{\sfx}{\abs{\sfx}_1} = (1-\epsilon, \epsilon )$}
the probability $\bbP_\beta^{\sfh_\sfx} (\sfX = 4\sfe_1 -\sfe_2)$ is
(asymptotically in $\beta$) of order ${\rm e}^{-2\beta}$.
However, for $\epsilon$-s of
order ${\rm e}^{-\beta}$ the probability of $\bbP_\beta^{\sfh_\sfx} ( \sfX = 4\sfe_1 -\sfe_2 )$ is comparable to
$\bbP_\beta^{\sfh_\sfx} (\sfX = \sfe_2 )$. For the sake of a
unified exposition we always include
 $4\sfe_1 -\sfe_2 $ to the set of basic steps   $\frS_0$.
\end{remark}
}
Recall our notation $ \sfv^* (\beta , \sfx ) =
\sfv^* (\beta , \sfh_\sfx ) =
\bbE_\beta^{\sfh_{\sfx}} \lb \sfX (\Gamma ) \rb$
{
for
the mean displacement over an
irreducible animal
sampled from $\lb \sfA , \bbP_\beta^{\sfh_\sfx} \rb$.
}
%By \eqref{eq:BasicProb}
By \eqref{eq:G1-4Ecobtr}   and  \eqref{eq:non-basic},
\be
\label{eq:vstar}
%\begin{split}
 \sfv^* (\beta , \sfx )
= \lb {\rm e}^{-a_\sfx (\beta )} , {\rm e}^{- b_\sfx (\beta )}\rb +
{\rm e}^{-\beta -\Delta_\sfx (\beta )}\lb 4, -1 \rb +
\bigof{{\rm e}^{-2\beta }} .
%\end{split}
\ee
By Theorem~\ref{thm:OZ},   $\sfv^* (\beta , \sfx )$ points in the direction of  $\sfx$,
in other words  there exists $\ell_\sfx \in \bbR_+$ such that
\be
\label{eq:ell-x}
\sfx = \ell_\sfx
\sfv^* (\beta , \sfx ) ,
\ee
writing $\sfx = \abs{\sfx }_1\lb 1-\epsilon , \epsilon \rb$,
we, in view of the first of \eqref{eq:G1-New},  conclude that
\be
\label{eq:ell-x-vrsx}
%\abs{
\ell_\sfx
=
%-
\abs{\sfx}_1
\lb 1 +
\bigof{
 %\leqs \ell_\sfx
{\rm e}^{- \lb \beta +\Delta_\sfx (\beta )\rb} }\rb ,
\ee
and, as a consequence, that
\be
\label{eq:expectation}
{
\left|  \lb {\rm e}^{-a_{\sfx }(\beta ) }+ 4{\rm e}^{-\beta - \Delta_\sfx (\beta )} ,
{\rm e}^{- b_{\sfx}(\beta )} -{\rm e}^{-\beta - \Delta_\sfx (\beta )}\rb  -
%\frac{\sfx}{\ \|\sfx\|_1}
(1- \epsilon , \epsilon  ) \right|_{{1}} \leq c {\rm e}^{-  \beta -  \Delta_\sfx (\beta)} ,
}
\ee
uniformly in ${\rm arg} (\sfx ) \in [0,\frac{2\pi}{ 5}]$ and $\beta$ large.

\paragraph{\bf Decomposition of $\sfR_k$.}
We shall always represent random walk $\sfR_k$ as
\be
\label{eq:Rk-repr}
\sfR_k = \sum_{1}^{k}\lb\xi_i \sfU_i + (1-\xi_i )\sfV_i\rb ,
\ee
where
\begin{enumerate}
 \item $\lbr \xi_i\rbr$ is  a  sequence
of i.i.d. Bernoulli random variables with probability of success  $\bbP (\xi_i  =1 )= q$;
\be
\label{eq:q-choice}
q = \alpha_\sfx^1{\rm e}^{-a_\sfx (\beta )} + \alpha_\sfx^2 {\rm e}^{-b_\sfx (\beta )}  .
\ee
There are two different choices of $\alpha_\sfx^i \in [0,1]$, according to
the direction $\sfx$, as described in \opt{1} and \opt{2} below.
In both cases, however, $q$ in \eqref{eq:q-choice} will satisfy:
\be
\label{eq:q-bound}
{
1-q \eqvs  {\rm e}^{- (\beta +\Delta_\sfx )} .
}
\ee
\item $\lbr \sfU_i\rbr$ is an
independent (from $\lbr \xi_i\rbr$)  sequence of
 i.i.d random vectors which take values $(1,0)$ and $(0,1 )$ with
probabilities
\be
\label{eq:p-choice}
 1- p = \bbP \lb \sfU_i = (1,0 )\rb =
 \frac{{\alpha_\sfx^1}{\rm e}^{- a_\sfx (\beta )}}
{ \alpha_\sfx^1 {\rm e}^{-a_\sfx (\beta )} +\alpha_\sfx^2 {\rm e}^{-b_\sfx (\beta )}} ,
\ee
and $\bbP \lb \sfU_i = (0,1 )\rb = p$, respectively.
%, $1-p$. Let  Note that
\item
$\lbr \sfV_i\rbr$ is  yet another independent (from $\lbr \xi_i\rbr$ and
$\lbr \sfU_i\rbr$)
 sequence of i.i.d.
random vectors with
{
\be
\label{eq:Vi}
\begin{split}
&\bbP \lb \sfV_i = \sfe_1\rb = \frac{(1-\alpha_\sfx^1 ){\rm e}^{-a_\sfx}}{1-q}, \quad
\bbP \lb \sfV_i = \sfe_2 \rb = \frac{(1-\alpha_\sfx^2 ){\rm e}^{-b_\sfx}}{1-q}\\
&\qquad \text{and, for $\sfy\neq \sfe_1, \sfe_2 $,}\
\bbP\lb \sfV_i = \sfy \rb = \frac{1}{1-q}
  \bbP_\beta^\sfh (\sfX (\Gamma ) = \sfy) .
\end{split}
\ee
}
\end{enumerate}
{
By \eqref{eq:exp-tails}
for any choice of
$\alpha_\sfx^i\in [0,1]$ as above
}
the distribution of $\sfV_i$  has exponential tails:
%two properties:
$\exists r_0, \nu_0 >0\ \text{such that}$
\be
\label{eq:p-U-1}
 \bbP\lb \abs{\sfV_i} >r \rb
 {\leq
}  {\rm e}^{-\beta\nu_0 r} \
\text{uniformly in $r\geq r_0$, $\sfx$ and $\beta$ large}.
\ee
In addition,
{
we shall choose $\alpha_\sfx^i\in [0,1]$ in such a way that}
the distribution of $\sfV_i$ will be  uniformly non-degenerate in
the following sense:
{
There exist $\delta_1 , \delta_2 \in (0,1 )$  such that
\be
\label{eq:Vi-vert}
\delta_1 \leq \min
{
\lbr
\bbP \lb \sfV_i = \sfe_2\rb, \bbP \lb \sfV_i = 4\sfe_1 - \sfe_2\rb
\rbr
}
%\cdot \sfe_2 = \pm 1\rb  \leq \max
%\bbP \lb \sfV_i\cdot \sfe_2 = \pm 1\rb  \leq 1-\delta
\ {\rm and}\
\bbE \sfV\cdot \sfe_1 \in [\delta_2 ,\delta_2^{-1}] .
\ee
uniformly in
{
${\rm arg} (\sfx ) \in [0,\frac{2\pi}{ 5}]$
}
and $\beta$ large.
}

Let $N_\ell$ be the number of failures (zeros)  of $\lbr \xi_i\rbr$ until the $\ell$-th success,
and let $\sfR^\sfU$ and $\sfR^\sfV$ be the random walks with steps
$\lbr \sfU_i\rbr$ and $\lbr \sfV_i\rbr$.
Then,
\be
\label{eq:decompUV-sum}
\begin{split}
& \sum_{\ell ,m}\bbP (N_\ell =m )\sum_{\sfy} \bbP (\sfR_\ell^U = \sfy ) \bbP (\sfR^V_m = \sfx -\sfy )\\
&\quad  \leq
\bbP_\beta^\sfh \lb 0,  \sfx \rb
{
 \leq c_2
\sum_{\ell ,m}\bbP (N_\ell =m )\sum_{\sfy} \bbP (\sfR_\ell^U = \sfy ) \bbP (\sfR^V_m = \sfx -\sfy ) .
}
\end{split}
\ee
{
Indeed, the difference between $\bbP_\beta^\sfh \lb 0,  \sfx \rb $ and the l.h.s. sum in
\eqref{eq:decompUV-sum} is that the former takes into account all possible superpositions
of steps of $\sfU$ and $\sfV$ walks, whereas the latter ignores the situation
when $\sfx$ is hit by a $\sfV$-step. The upper bound in \eqref{eq:decompUV-sum} 
 follows from
\eqref{eq:q-bound}.
}
\medskip

\begin{proof}[Proof of Proposition~\ref{prop:LB-P}]
We now turn to an analysis of \eqref{eq:decompUV-sum}.
It would be helpful to  remember that by \eqref{eq:ell-x} the running scale $\ell_\sfx$ satisfies:
\be
\label{eq:lx-scale}
q \ell_\sfx \lb  1-p ,  p\rb + (1-q )\ell_\sfx \bbE V_i
\df \lxu (1-p , p ) +\lxv\bbE \sfV_i
 = \sfx ,
\ee
{
where we have defined $\lxu = q\ell_\sfx$ and $\lxv = (1-q)\ell_\sfx$.
}
We shall rely on the elementary Lemma~\ref{lem:V-walk} below,
which is claimed to hold uniformly in i.i.d sequences $\lbr \sfV_i\rbr$ satisfying
\eqref{eq:p-U-1} and \eqref{eq:Vi-vert}:
%random walks $\sfR^\sfV_k$:
 Let us fix $ C\in (0,\infty )$ and, given $\sfu\in\bbR^2$ and
$r \in\bbN$ define
\[
 \Lambda_r (\sfu ) = \sfu + \{[-Cr, \dots ,Cr]\times [-r , r]\} .
\]
For a horizontal lattice line $\calB$ let us say that
% $\calL\sim \Lambda_r (\sfu )$ if
$\calB\cap\Lambda_r (\sfu ) \neq \emptyset$
% (in other words,
if $\calB$ passes through
$\Lambda_r (\sfu )$.
\begin{lemma}
 \label{lem:V-walk}
There exist $C = C (\delta , r_0 , \nu_0 )$ and $c = c (\delta , r_0 , \nu_0 )$
such that
\be
\label{eq:V-walk-hits}
\sum_{\abs{k -
{
n
}
}\leq cr} \bbP \lb
\sfR^\sfV_k \in\calB\cap\Lambda_r (
{
n
}
\bbE\sfV   )\rb \eqvs \frac{r^2}{
{
n
}
 \vee 1}
\ee
uniformly in $
{
n \geq 0
}
$, integers $r\leq \sqrt{
{
n
}
}\vee 1$,
horizontal lines $
{
\calB\cap\Lambda_r (n\bbE\sfV   )\neq \emptyset
}
$ and
$\beta$ large.
\end{lemma}
\paragraph{\bf  Sketch of the proof:}
Eq. \eqref{eq:V-walk-hits} is a coarse estimate, and the logic behind it should be transparent:
By \eqref{eq:p-U-1} $\sfV_i$-s have exponential tails. On the other hand,
\eqref{eq:Vi-vert} yields a lower bound on the non-degeneracy of covariance
structure of $\sfV_i$.
 A usual local limit analysis implies that one can choose $c$ and  $C$ in
such a way that
$\bbP \lb
\sfR^\sfV_k \in\calB\cap\Lambda_r (
n
\bbE\sfV   )\rb \eqvs \frac{r}{n\vee 1} $
 uniformly in all lines
$\calB\cap\Lambda_r (n\bbE\sfV   )\neq \emptyset$ and in all times $k$ with
$\abs{k - n}\leq c r $.

\medskip

\noindent
{
Let us fix a sufficiently large constant $c_0$. How exactly it is fixed is
explained below when we consider \opt{2}.
}
\smallskip

\noindent
\opt{1}. {$\epsilon \leq c_0{\rm e}^{-\beta }$.
}
By Proposition~\ref{prop:aep-bep}
in this regime the probabilities $\bbP_\beta^\sfx (\Gamma_2 ) = {\rm e}^{-b_\epsilon}$ and
$\bbP_\beta^\sfx (\Gamma_3 ) = {\rm e}^{-\beta -\Delta_\epsilon}$ are of the same order ${\rm e}^{-\beta}$.
Consequently, if we  take
$\alpha^2_\sfx = p = 0$ and $\alpha_\sfx^1 = 1$,
%By observation (and by \eqref{eq:non-basic})
both \eqref{eq:q-bound}
%\eqref{eq:p-U-1}
and, also in view of \eqref{eq:non-basic}, \eqref{eq:Vi-vert}
 are  satisfied.

If $p=0$, then the $\sfR^U$-walk is trivial: $\bbP\lb \sfR^U_k = (\ell , 0)\rb = \delta_{k\ell}$.
Consequently, \eqref{eq:decompUV-sum} takes a particular simple form:
\be
\label{eq:decompUV-sum-1}
\bbP_\beta^\sfh \lb 0,  \sfx \rb \eqvs
\sum_{\ell ,m}\bbP (N_\ell =m )
\bbP \lb \sfR^\sfV_m = (x_1 -\ell , x_2) \rb  .
\ee
Recall how $\lxu$ and $\lxv$ were defined in \eqref{eq:lx-scale}. Note that
$\lxu = \frac{q}{1-q}\lxv$.
By construction  $p=0$. Consequently,
$(x_1 -\ell , x_2) = \sfx - (\ell , 0 ) = (\lxu -\ell , 0) +\lxv\bbE\sfV$, and
\[
\bbP \lb \sfR^\sfV_m = (x_1 -\ell , x_2) \rb = \bbP \lb
 \sfR^\sfV_m = \bbE\sfR^\sfV_m - (m-\lxv)\bbE\sfV - (\ell -\lxu , 0)\rb  .
\]
In view of the last of \eqref{eq:Vi-vert}, the main  contribution to \eqref{eq:decompUV-sum-1}
should come from the values of $m$ and $\ell$ satisfying
\[
 \abs{m - \lxv} \leq c_3 \sqrt{\lxv}\ {\rm and}\ \abs{\ell - \lxu} \leq
 c_3 \sqrt{\lxv} .
\]
{ Let us estimate \eqref{eq:decompUV-sum-1} for the values of $m$ and $\ell$
restricted to the latter region.} By a direct application of Stirling formula,
\be
\label{eq:Stirl-bound-N}
\bbP (N_\ell =m ) \eqvs \frac{1}{\sqrt{\lxv}\vee 1}
\ee
uniformly in $\abs{m - \lxv} \leq c_3 \sqrt{\lxv}$  and
$\abs{\ell - \lxu} \leq c_3 \sqrt{\lxv}$.
Furthermore, there exists $c <\infty$, such that
{
  \begin{eqnarray}
    \label{eq:4}
\bbP (N_\ell =m ) \le  \frac{c}{\sqrt{\lxv}\vee 1}
  \end{eqnarray}
for every $m$.}

An application of Lemma~\ref{lem:V-walk} with $r=\sqrt{\lxv}\vee 1$ implies, therefore:
\be
\label{eq:Lem-appl-1}
 \sum_{\abs{\ell^\prime}\leq c_3 \sqrt{\lxv}}\
\sum_{\abs{m - \lxv}\leq c_3 \sqrt{\lxv}}
\bbP \lb \sfR^\sfV_m = \bbE \sfR^\sfV_{\lxv} + \lb  \ell^\prime  , 0\rb\rb \eqvs 1 .
\ee
Together with \eqref{eq:Stirl-bound-N} and {\eqref{eq:4}} this implies that
\be
\label{eq:case1}
\bbP_\beta^\sfh \lb 0, \sfx \rb \eqvs  \frac{1}{\sqrt{\lxv}\vee 1}
{
\eqvs \frac{1}{\sqrt{{\rm e}^{-b_\sfx (\beta )} \abs{\sfx}_1}\vee 1},
}
\ee
uniformly in $\beta$ large and $\sfx$-s complying with \opt{1}.
{
The last asymptotic equivalence;  $\lxv \eqvs {\rm e}^{-b_\sfx (\beta )} \abs{\sfx}_1$,
holds
since by
\eqref{eq:lx-scale}, \eqref{eq:ell-x-vrsx} and by \eqref{eq:q-bound},
\[
\lxv = (1-q)\ell_\sfx \eqvs  (1-q )\abs{\sfx}_1
\eqvs {\rm e}^{-\beta -\Delta_\sfx }\abs{\sfx}_1 .
\]
However, by Proposition~\ref{prop:aep-bep},
 ${\rm e}^{-\beta -\Delta_\sfx }\eqvs {\rm e}^{-\beta} \eqvs {\rm e}^{-b_\sfx }$ uniformly
in $\epsilon \leq c_0 {\rm e}^{-\beta}$ and $\beta$ large.
}
\medskip

\noindent
\opt{2}. {$ \epsilon > c_0{\rm e}^{-\beta}$.}
By \eqref{eq:non-basic} there exists $\eta >0$ such that
\be
\label{eq:G3-control}
{\rm e}^{-\beta - \Delta_\epsilon} \geq \eta\sum_{\sfy \neq\sfe_1 , \sfe_2}\abs{\sfy}_1\bbP_\beta^{\sfh_\epsilon}
\lb \sfX = \sfy\rb ,
\ee
uniformly in $\beta$ large and $\arg\lb (1-\epsilon , \epsilon )\rb \in [0, \frac{2\pi}{5} ]$.
We shall choose
$\alpha_\epsilon^i$ in the decomposition \eqref{eq:Rk-repr} of $R_k$ as follows
\be
\label{eq:Case2-alpha}
\lb 1- \alpha_\epsilon^1\rb {\rm e}^{-a_\epsilon} = \frac{1}{\eta}{\rm e}^{-\beta - \Delta_\epsilon}
\ {\rm and}\
\lb 1- \alpha_\epsilon^2\rb {\rm e}^{-b_\epsilon} = {\rm e}^{-\beta - \Delta_\epsilon}.
\ee
Since $a_\epsilon$ is uniformly  bounded and  $b_\epsilon < \beta < \beta +\Delta_\epsilon$,
\eqref{eq:Case2-alpha} is a feasible choice. In view of \eqref{eq:G3-control} we
readily verify \eqref{eq:Vi-vert} {and also
  \eqref{eq:q-bound}}.
  {In fact recalling how  $q$  was defined in \eqref{eq:q-choice},
  we, in view of \eqref{eq:G3-control}, infer that under \eqref{eq:Case2-alpha} $1-q$ satisfies
  the following bound
 \be
  \label{eq:q-eta-bound}
  1- q\in {\rm e}^{-\beta  - \Delta_\epsilon}\lb 2 +\frac{1}{\eta} , 2 +\frac{2}{\eta}\rb ,
  \ee
  uniformly in $\beta$ large and $\epsilon > c_0 {\rm e}^{-\beta}$.
} %
%This can hold only if $x_2 >0$.
Furthermore, by \eqref{eq:uep-cases} there exists $c_4 <\infty$  such that
{
\be
\label{eq:Case2-bounds}
b_\epsilon \leq \beta - \log c_0  + c_4 {\rm e}^{-\beta }\ {\rm and}\
 \Delta_\epsilon \geq  \log c_0
-c_4   {\rm e}^{-\beta} ,
%+\bigof{{\rm e}^{-\beta}},
\ee
uniformly in $c_0 \geq 2$, $\beta$ large and $\epsilon > c_0{\rm e}^{-\beta}$.
}
This means (see
\eqref{eq:lx-scale} for the definition of $\lxu$ and $\lxv$)  that
\be
\label{eq:Case2-scales}
\lxv = \frac{1-q}{q} \lxu \leq p\lxu \eqvs x_2 \eqvs \abs{\sfx }_1 {\rm e}^{-b_\sfx} ,
\ee
also uniformly  in \opt{2}.
{ Indeed,
the last equivalence follows from \eqref{eq:expectation} and
Proposition~\ref{prop:aep-bep} {choosing $c_0$ large}. On the other hand,
going back to the definition of $p$ in \eqref{eq:p-choice}, the choice of $\alpha_\epsilon^2$ in
in \eqref{eq:Case2-alpha} implies that $p \geq {\rm e}^{-b_\epsilon} - {\rm e}^{-\beta -\Delta_\epsilon}$.
 Comparing with \eqref{eq:q-eta-bound}, and in view of \eqref{eq:Case2-bounds}, we conclude that
  $p\geq (1- q)/q$ as soon as $c_0$ is large enough. Hence the first inequality in
  \eqref{eq:Case2-scales}.
}

Let us go back to \eqref{eq:decompUV-sum} and write:
\be
\label{eq:sum-lmr}
 \bbP_\beta^\sfh \lb 0,  \sfx \rb \eqvs
\sum_{\ell ,m}\bbP (N_\ell =m )\sum_{
%\ell_1 +\ell_2 = \ell
r}
\bbP (\sfR_\ell^\sfU = (
%\ell_1 ,\ell_2 )
\ell-r , r) \bbP \lb \sfR^\sfV_m = \sfx -(
%\ell_1 , \ell_2 )
\ell-r , r)\rb
\ee
By Stirling's formula the main contribution to
\be
\label{eq:1-term}
 \bbP (N_\ell =m )\eqvs \frac{1}{\sqrt{m}}\ \text{
comes from $\abs{\ell - \frac{q}{1-q} m} \leq c_4 \sqrt{m}$}.
\ee
\newline
Next, since $\bbE \sfR_\ell^\sfU = \ell (1-p , p )$,
\[
 \bbP \lb \sfR_\ell^\sfU = (
%\ell_1 ,\ell_2 )
\ell-n , n) \rb = \bbP \lb \sfR_\ell^\sfU = \bbE \sfR_\ell^\sfU + (\ell p -n , n-\ell p )\rb ,
\]
and consequently, again by Stirling's formula,  the main contribution to
\be
\label{eq:2-term}
%\sum_{r}
\bbP \lb \sfR_\ell^\sfU = (
%\ell_1 ,\ell_2 )
\ell-n , n) \rb
%\bbP (\sfR^\sfV_m = \sfx -(
%\ell_1 , \ell_2 )
%\ell-r )
\eqvs
\frac{1}{\sqrt{p\ell}} \ \text{comes from $\abs{n - \ell p }\leq c_5 \sqrt{\ell p}$}.
\ee
\newline
Recall \eqref{eq:lx-scale} that $\sfx = \lxu (1-p , p) +\lxv\bbE\sfV$. Therefore,
setting $\bar n = p\lxu $,
\be
\label{eq:x-decomp}
\begin{split}
\sfx -(
%\ell_1 , \ell_2 )
\ell-n , n) &= \bbE\sfR_{\lxv}^\sfV  + \lb (\lxu - \ell ) - (\bar n - n),  (\bar n -  n )\rb
\\
& =
\bbE\sfR_{ m}^\sfV + (\lxv - m) \bbE\sfV
+ \lb (\lxu - \ell ) - (\bar n - n),  (\bar n -  n )\rb .
%\lb (\lxu - \ell ) + (r -\bar r), - (r - \bar r )\rb .
\end{split}
\ee
Since we restrict attention to $\ell$ and $m$ satisfying the second of \eqref{eq:1-term},
and since $\lxu = \frac{q}{1-q}\lxv$, the main contribution to
\be
\label{eq:3-term}
\bbP \lb \sfR^\sfV_m = \sfx -(
%\ell_1 , \ell_2 )
\ell-n , n)\rb   \
\text{comes from $\abs{m-\lxv}, \abs{\ell - \lxu}, \abs{n - \bar n}\leq c_6\sqrt{\lxv}$}.
\ee
Let us go back to \eqref{eq:sum-lmr}. In view of \eqref{eq:1-term}-\eqref{eq:3-term},
\be
\label{eq:case2-subst}
%\begin{split}
\bbP_\beta^\sfh\lb 0, \sfx \rb
\eqvs
\frac{1}{\sqrt{p\lxu}} \frac{1}{\sqrt{\lxv}\vee 1}
{\sum}^*
\bbP \lb \sfR^\sfV_m = \bbE \sfR_{\lxv}^\sfV +
\lb (\lxu - \ell ) - (\bar n - n),  (\bar n -  n )\rb
\rb
%\end{split}
\ee
where
\[
 {\sum}^* \df
\sum_{\abs{\ell-\lxu }\leq c_7 \sqrt{\lxv}\vee 1}\
\sum_{\abs{n - \bar n}\leq c_7 \sqrt{\lxv}\vee 1} \
\sum_{\abs{m -\lxv}\leq c_7 \sqrt{\lxv}\vee 1} .
\]
By an application of Lemma~\ref{lem:V-walk} (again with $r = \sqrt{\lxv}\vee 1$),
{
and in view of \eqref{eq:Case2-scales},
}
\be
\label{eq:case12}
\bbP_\beta^\sfh \lb 0, \sfx \rb \eqvs \frac{1}{\sqrt{x_2}}
{
\eqvs \frac{1}{\sqrt{{\rm e}^{-b_\sfx (\beta )}\abs{\sfx}_1}}
}
,
\ee
uniformly in \opt{2}.

Putting \eqref{eq:case1} and \eqref{eq:case12} together we deduce the claim \eqref{eq:LB-P}
of Proposition~\ref{prop:LB-P}.
\end{proof}

\subsection{Alili-Doney representation.}
\label{sub:AD}
The (strict) event $\hat \frR_\ell^+$ is defined similarly to \eqref{eq:Rlplus-n},
\be
\label{eq:Rlplus}
 {\hat\frR_\ell^{+ ,\sfn} (\sfw,\sfz ) = \lbr \uGamma:\, S_1, \dots , S_{\ell -1} > 0
 ; \sfR_0=\sfw,\sfR_\ell =\sfz\rbr .}
\ee
In order to explore
$\bbP_{\beta , +}^\sfh (\sfw , \sfz ) = \sum_\ell
\bbP_\beta^\sfh \lb\frR_\ell^{+ ,\sfn} (\sfw , \sfz )\rb$-terms in \eqref{eq:Gbound-1}
we need both strict and non-strict events .
Indeed, define $\hat \bbP_{\beta , +}^\sfh (\sfw , \sfz ) = \sum_\ell
\bbP_\beta^\sfh \lb\hat \frR_\ell^{+ ,\sfn} (\sfw , \sfz )\rb$
Then, the  decomposition of
{
effective random walk trajectory  $\lb \sfw = \sfR_0, \dots , \sfR_\ell = \sfz\rb$
with respect to the first {absolute} minimum $\sfy = \sfR_k$;
\[
 \sfy\cdot\sfn = \min_{m}\sfR_m\cdot\sfn\quad{\rm and}\quad
 \sfR_m\cdot\sfn >\sfy\cdot \sfn\ \text{for any $m<k$},
\]
yields:
}
{
\be
\label{eq:Ruv-decomp}
\bbP_{\beta ,+}^\sfh \lb
%\frR^{+ ,\sfn} (
\sfw , \sfz
%)
\rb =
%\sum_\ell
%\bbP_\beta^\sfh \lb \frR_\ell^{+ ,\sfn} (\sfw , \sfz )\rb =
\sum_\sfy \hat \bbP_{\beta , +}^{
\sfh
%}
}
\lb \sfw - \sfy , 0
\rb
\bbP_{\beta ,+}^\sfh \lb
%\frR^{+ ,\sfn} (
0, \sfz -\sfy
% )
\rb ,
\ee
}
\smallskip

\noindent
{\em Non-strict  {ascending} ladder height} $H_1$ of $S_\ell$  is defined via
\be
\label{eq:LV-ns}
 H_1 = S_{\tau_1}\quad {\rm where}\quad \tau_1 = \min\{\ell >0 \, :\, S_\ell \geq S_0\}
\ee
$H_2, H_3, \dots$ are defined recursively.

\noindent
{\em Strict {descending} ladder height} $\hat H_1$ of $S_\ell$  is defined via
\be
\label{eq:LV-s}
 \hat H_1 = S_{\hat \tau_1}\quad {\rm where}\quad \hat\tau_1 = \min\{\ell >0 \, :\, S_\ell  {<} S_0\}
\ee
$\hat H_2, \hat H_3, \dots$ are defined recursively.

Let
 $N^+_m ( z )$
 {be the total number of
  {\em non-negative}
 non-strict ladder heights $H_i \leq z$
 reached during first $m$ steps by  the
effective random walk $S_\ell$ defined in \eqref{eq:Swalk}.
Similarly. let
 $ N^-_m (z)$ be the total number of {\em non-negative}
   strict descending ladder heights  $ \hat H_j\leq z$
   reached during the first $m$ steps of  $S_\ell$.
  We drop sub-index
$m$ whenever talking about $m=\infty$ ({that is whenever talking about the}
total number of ladder heights).
}

{
For $\sft\in \calH_{+ , \sfn}$ define events
\be
\label{eq:L-heights}
\calL^+_\sft = \lbr \exists\, i : \sfR_{\tau_i} =\sft\rbr\quad
\text{and}\quad
\calL^-_\sft =
\lbr \exists\, i : \sfR_{\hat\tau_i} =\sft\rbr .
\ee
Then,  recalling that $\frA_m(\sfx,\sfy)$ was defined just after
\eqref{eq:D-shape},
\be
\label{eq:Heights-vrsRplus}
\hat\bbP_{\beta ,+}^\sfh \lb \sfv , 0\rb = \sum_m \bbP_\beta^\sfh
\lb \frA_m ( \sfv , 0); \calL^-_0\rb\ {\rm and}\
\bbP_{\beta ,+}^\sfh \lb 0, \sfv \rb = \sum_m \bbP_\beta^\sfh
\lb \frA_m (0, \sfv ); \calL^+_\sfv \rb .
\ee
The first of \eqref{eq:Heights-vrsRplus} is straightforward. The second
follows by a well-known rearrangement argument: If $s = \sum_1^m s_j$
 and $\sum_1^ks_j \geq 0$ for any $k$, then $\hat S_k = \sum_{m-k+1}^m s_j$
for $k=1... m$
 satisfies $\hat S_m = s = \max_k\hat S_k$.
}

An
adaptation of the
 combinatorial lemma by Alili and Doney \cite{AD}
for the effective random walk setup was formulated \cite{CIL}.
In particular, for any $\sfv\in \calQ_+\setminus 0$ and $\sfn$,
{
\be
\label{eq:AD-expBound}
\begin{split}
\bbP_{\beta, +}^{\sfh}\lb 0  ,\sfv \rb
&= \sum_m  \frac{1}{m} \, \bbE_\beta^{\sfh}\lb {\frA_{m}}   (0 ,\sfv )\, ;\,
N^+_m (\sfv\cdot \sfn )\rb
{
\eqvs \frac{1}{\abs{\sfv}_1} \sum_{\frac{\abs{\sfv}_1}{c} \leq m \leq c\abs{\sfv}_1}
\!\!\!\! \!\!\!\bbE_\beta^{\sfh}\lb {\frA_{m}}   (0 ,\sfv )\, ;\,
N^+_m (\sfv\cdot \sfn )\rb
}
\\
&\quad {\rm and}\\
\hat \bbP_{\beta , +}^{\sfh}\lb \sfv , 0 \rb
&= \sum_m \frac{1}{m} \, \bbE_\beta^{\sfh}\lb {\frA_{m}}   ( \sfv , 0 )\, ;\,
N^-_m  (\sfv\cdot \sfn )\rb
\eqvs \frac{1}{\abs{\sfv}_1} \sum_{\frac{\abs{\sfv}_1}{c} \leq m \leq c\abs{\sfv}_1}
\!\!\!\! \!\!\!
\bbE_\beta^{\sfh}\lb {\frA_{m}}   ( \sfv , 0 )\, ;\,
N^-_m  (\sfv\cdot \sfn )\rb .
\end{split}
\ee
{ Here for the event ${\frA}$ and the random variable $N$ we use the notation $\bbE\lb {\frA}\, ;\,
N\rb$ for the expectation of the random variable $\bbI_{{\frA}}N$.}
The $\eqvs$ relation in both of \eqref{eq:AD-expBound} follows from
\eqref{eq:breakpoints}, and it is uniform in $\sfn$ and $\beta$ large.
}
However, since in
our context the dependence on $\beta$ of coefficients in inequalities is important,
and since for general directions $\sfn$ the range of the effective random walk $S_\ell$ is
quite different from $\bbZ$,
we need to rerun
the \eqref{eq:AD-expBound}-based  computations of \cite{CIL} more carefully.

{
Relations \eqref{eq:AD-expBound} imply that
\be
\label{eq:AD-expBound-ub}
\begin{split}
\bbP_{\beta, +}^{\sfh}\lb 0  ,\sfv \rb &\leq \frac{c_1}{\abs{\sfv}_1}
\bbE_\beta^{\sfh}\lb {\frA}   (0 ,\sfv )\, ;\,
N^+_{c\abs{\sfv}_1} (\sfv\cdot \sfn )\rb\ \\
&\quad {\rm and}\\
\hat \bbP_{\beta, +}^{\sfh}\lb \sfv , 0 \rb &\leq \frac{c_1}{\abs{\sfv}_1}
\bbE_\beta^{\sfh}\lb {\frA}   (\sfv , 0 )\, ;\,
N^-_{c\abs{\sfv}_1} (\sfv\cdot \sfn )\rb .
\end{split}
\ee
}
Since $\sum_m\1_{\lbr \sfR_m = \sfv\rbr}$ is an indicator,  we have for all $k\ge 1$
\be
\label{eq:Holder-AD-OZ}
\begin{split}
&\bbE_\beta^{\sfh}\lb
{
\frA
}
(0 ,\sfv )\, ;\, N^+_{c\abs{\sfv}_1} (\sfv\cdot \sfn )\rb =
\sum_m \bbE_\beta^{\sfh}\lb
\frA_{m}
(0 ,\sfv )\, ;\, N^+_{c\abs{\sfv}_1}
(\sfv\cdot \sfn )\rb \\
&\quad = \bbE_\beta^{\sfh}\lb \sum_m\1_{\lbr \sfR_m = \sfv\rbr}\, ;\,
N^+_{c\abs{\sfv}_1} (\sfv\cdot \sfn )\rb
\leq \lb
\bbP_\beta^{\sfh} \lb 0, \sfv\rb\rb^{\frac12}
 \lb \bbE_\beta^{\sfh}\lb N^+_{c\abs{\sfv}_1} (\sfv\cdot \sfn )\rb^2 \rb^{\frac{1}{2}}.
\end{split}
\ee
{
Similarly,
\be
\label{eq:Holder-AD-OZ-d}
\bbE_\beta^{\sfh}\lb\frA   ( \sfv , 0 )\, ;\, N^-_{c\abs{\sfv}_1} (\sfv\cdot \sfn )\rb
\leq \lb
\bbP_\beta^{\sfh} \lb \sfv , 0\rb\rb^{\frac12}
 \lb \bbE_\beta^{\sfh}\lb N^-_{c\abs{\sfv}_1} (\sfv\cdot \sfn )\rb^2 \rb^{\frac{1}{2}}
\ee
}
Let us make a general statement for one-dimensional random walks
$S_\ell$, {with $S_0 {\geq} 0$}:
\begin{lemma}
\label{lem:Np-k}
Assume that for some $
{
\eta
}
, p >0$,
\be
\label{eq:ep-p}
\bbP (H_1
{
-S_0
}
\geq \eta ; {\tau_1 <\infty}  ) \geq  p .
\ee
Then, for any $z \geq 0$ and for any power $k\geq 1$,
%\blue{ REPLACE z by $z-S_0$}
\be
\label{eq:Np-k}
\bbE\lb N^+_m (z)^k \rb  \leq \bbE  \lb
%\left\lceil\frac{z}{\epsilon}\right\rceil +
\sum_1^{\left\lceil\frac{z}{\eta }\right\rceil \wedge m} M_i\rb^k \leq \lb c_k
\frac{\left\lceil\frac{z}{\eta}\right\rceil\wedge m}{p}\rb^k
%\lb \left\lceil\frac{z}{\eta}\right\rceil\wedge m\rb\rb^k \frac{1}{p^k },
\ee
where $M_1, M_2, \dots , $ are independent ${\rm Geo} (p )$. Here $c_k$ is a combinatorial
constant which does not depend on $p$, $z$ or $\eta$.
The same holds for
strict
{descending}
ladder heights (i.e. if \eqref{eq:ep-p} holds for $ S_0 - \hat H_1$
then \eqref{eq:Np-k} holds for $ N^-$).
\end{lemma}
\begin{proof} Let us say that $H_i$ (respectively $\hat H_i$) is a
  substantial ascending ladder (descending strict ladder) height if $H_i - H_{i-1}\ge \eta$
  (respectively, if $\hat H_{i-1} - \hat H_{i}\ge \eta$).

  We proceed with
talking only about ascending ladder heights, the argument for strict
descending
ladder heights would be a literal
repetition.
{Since we are counting non-negative ladder heights}
there are at most $\left\lceil\frac{z}{\eta}\right\rceil$ substantial  ladder heights
$H_i$ with $0 \leq H_i\leq z$.
The number $M$ of ladder heights between two successive substantial ladder heights is
stochastically
dominated by ${\rm Geo} (p )$.
The inequality \eqref{eq:Np-k} just states that $N^+_m (z)$ is bounded above
by $M_1 +\dots + M_{\left\lceil\frac{z}{\eta}\right\rceil\wedge m}$, which is the
total  number
of ladder heights
reached during
the first $m$ steps of the walk
 or until the $\left\lceil\frac{z}{\eta}\right\rceil$-th
substantial ladder height is produced.
\end{proof}
\smallskip

\paragraph{\bf Upper bounds on $\bbE_\beta^{\sfh_\epsilon}\lb
N^+_{m} (\sfs\cdot \sfn )\rb^k$ and
$\bbE_\beta^{\sfh_\epsilon}\lb
N^-_{m} (\sfs\cdot \sfn )\rb^k$.}
Recall that we are assuming that $\sfu\cdot \sfn \leq \sfv\cdot\sfn$, and
that $(\sfu , \sfv )$ belongs to
the set $\calA$ defined in \eqref{eq:crude-rho}. In particular, $\sfu$ and $\sfv$ satisfy
\eqref{eq:nu-uv-cond}. We continue to denote $ \sfv - \sfu
= \abs{\sfv - \sfu}_1\lb 1-\epsilon , \epsilon\rb$, $\sfh_\epsilon =
\nabla\tau_\beta (\sfv -\sfu )$   and
$\sfn = \abs{\sfn}_1\lb -\epsilon_\sfn , 1-\epsilon_\sfn \rb$. The second of
\eqref{eq:nu-uv-cond} implies that
\be
\label{eq:ep-epn}
0\leq \epsilon_\sfn \leq \epsilon \leq \epsilon_\sfn + {\rm e}^{-\nu\beta} ,
\ee
for some $\nu >1$.

{
Let us start with expectations of $N^+$:
\begin{lemma}
 \label{lem:Nplus}
  (a) If $\arg (\sfn ) = \frac{\pi}{2}$, then
%the bound \eqref{eq:Holder-AD-OZ1} is satisfied for ladder heights with
\be
\label{eq:Nplus-1}
\lb\bbE_\beta^{\sfh_\epsilon}
N^+_{m} (\sfs\cdot\sfn )^k\rb^{\frac{1}{k}} \leq
c_k \lb {\rm d}_\sfn (\sfs ) {\rm e}^{b_\epsilon}\rb\wedge m,
\ee
uniformly
%$\arg (\sfn  )\in [\frac{\pi}{2} , \frac{3\pi}{4}]$,
 in $\sfs$ and in
all $\beta$ sufficiently large.
\newline
(b) On the other hand, the bound
\be
\label{eq:Nplus-2}
\lb\bbE_\beta^{\sfh_\epsilon}
N^+_{m} (\sfs\cdot\sfn )^k\rb^{\frac{1}{k}} \leq
c_k {\rm d}_\sfn (\sfs ) \wedge m
\ee
is satisfied
uniformly in $\arg (\sfn ) \in (\frac{\pi}{2} , \frac{3\pi}{4}]$,
in $\sfs$
 and
all $\beta$ sufficiently large.
\end{lemma}
In order to formulate consequences of Lemma~\ref{lem:Np-k}
for expectations of $N^-$
let us introduce
the following notation: For $\sfs\cdot\sfn >0$ let
\be
\label{eq:ls}
\ell_\sfn (\sfs ) = \inf\lbr \ell > 0: (\sfs +\ell\sfe_1 )\cdot \sfn < 0\rbr .
\ee
\begin{lemma}
 \label{lem:Nminus}
  (a) If $\arg (\sfn ) = \frac{\pi}{2}$, then
%the bound \eqref{eq:Holder-AD-OZ1} is satisfied for ladder heights with
\be
\label{eq:Nminus-1}
\lb\bbE_\beta^{\sfh_\epsilon}
N^-_{m} (\sfs\cdot\sfn )^k\rb^{\frac{1}{k}} \leq
c_k  {\rm d}_\sfn (\sfs ) \wedge m,
\ee
uniformly
%$\arg (\sfn  )\in [\frac{\pi}{2} , \frac{3\pi}{4}]$,
 in $\sfs$ in question and in
all $\beta$ sufficiently large.
\newline
(b) On the other hand, the bound
\be
\label{eq:Nminus-2}
\lb\bbE_\beta^{\sfh_\epsilon}
N^-_{m} (\sfs\cdot\sfn )^k\rb^{\frac{1}{k}} \leq
c_k \ell_\sfn (\sfs)\wedge \lb {\rm d}_\sfn (\sfs ){\rm e}^{b_\epsilon} \rb \wedge m ,
\ee
is satisfied
uniformly in $\arg (\sfn ) \in (\frac{\pi}{2} , \frac{3\pi}{4}]$,
in $\sfs$ in question
 and
all $\beta$ sufficiently large.
\end{lemma}
}

{
\begin{proof}
 Consider $\arg (\sfn ) = \frac{\pi}{2}$.
 The claim (a) of Lemma~\ref{lem:Nplus}
 is a direct consequence of \eqref{eq:BasicProb}.
 {Indeed, we will see about ${\rm e}^{b_\epsilon}$    steps of
 the type $\Gamma_1$ before seeing anything else, since the
  $\Gamma_2$-step
 has probability ${\rm e}^{-b_\epsilon}$. Once happened, $\Gamma_2$-step gives rise
 to a
 substantial ladder height with $\eta_+ =1$.}
  \newline
   Similarly,
 a $\Gamma_3$ step which follows a sequence of
 $\Gamma_1$ steps, which has probability of order one,
 gives rise to  {the first}  strict descending ladder height of size
 $\eta_- = 1$, and claim (a) of Lemma~\ref{lem:Nminus}  follows as well.
 \newline
 We need to explain claims (b) in both Lemmas.
 Assume that
 $\sfn  = \abs{\sfn}_1(-\epsilon_\sfn , 1-\epsilon_\sfn )$
 satisfies $\arg (\sfn ) \in (\frac{\pi}{2} , \frac{3\pi}{4}]$.
 This means that $\epsilon_\sfn >0$.    Then a $\Gamma_2$-step, which follows
at most $\frac{1-\epsilon_\sfn}{2\epsilon_\sfn}$ successive $\Gamma_1$-steps, gives rise to
a substantial increment with ${\eta}_+  =\frac{1-\epsilon_\sfn}{{2}}\abs{\sfn}_1$.
The latter is bounded below
by
$c_2 >0$
%$1/4\sqrt{2}$
uniformly in $\arg (\sfn )\in \lb \frac{\pi}{2} , \frac{3\pi}{4}\right]$.
Set $N = \lceil\frac{1-\epsilon_\sfn}{2\epsilon_\sfn}\rceil$.
Then,
 \be
\label{eq:p-good}
 p_+ = p_+  (\beta ,\sfh  ) \geq {\rm e}^{-b_\epsilon (\beta )}\sum_{n=0}^{N-1}
{\rm e}^{-n a_\epsilon (\beta )}
=
\frac{{\rm e}^{-b_\epsilon (\beta )}}{1- {\rm e}^{- a_\epsilon (\beta )}}
\lb 1- {\rm e}^{-N a_\epsilon (\beta )}\rb ,
\ee
and $\eta_+ = \eta_+  (\beta ,\epsilon  ) ={(1-\epsilon_\sfn )/2\ge 1/(4\sqrt 2)}$.
Let us develop a more explicit bound for the quantity on the right-hand side of
\eqref{eq:p-good}.
 Since $\epsilon \geq \epsilon_\sfn$,
$N\geq \lceil \frac{1-\epsilon}{2\epsilon} \rceil$.  Therefore, by \eqref{eq:aeps},
\[
Na_\epsilon  (\beta ) \geq c_3
\lceil \frac{1-\epsilon}{2\epsilon}\rceil
{
\lb  \epsilon\vee {\rm e}^{-\beta}\rb} \geq c_4 >0 ,
\]
 uniformly in all the situations in question.
 \newline
 On the other hand, by \eqref{eq:G1-4Pcobtr} and \eqref{eq:uep-cases},
${\rm e}^{-b_\epsilon (\beta )}\geq \lb 1- {\rm e}^{-a_\epsilon (\beta )}\rb/2$ also
uniformly in all the situations in question.
\newline
We conclude: There exists $c_4 >0$ such that
the right hand side of \eqref{eq:p-good} is bounded below by $c_4$
uniformly in $\arg (\sfn ) \in [\frac{\pi}{2} , \frac{3\pi}{4}]$,
$\epsilon >\epsilon_\sfn$ and $\beta$ large.
\newline
Let us turn to the claim (b) of Lemma~\ref{lem:Nminus}.
 A $\Gamma_3$ step of the effective random walk,
 which has probability ${\rm e}^{-\beta -\Delta_\epsilon}$ gives rise to a
 strict descending ladder height of size at least $1$. On the other hand
 since $\epsilon_\sfn >0$,  a horizontal
 $\Gamma_1$ step, which has probability ${\rm e}^{-a_\epsilon}\eqvs 1$ also
 gives rise to a strict descending ladder height. The quantity
 $\ell_\sfn (\sfs )$ in \eqref{eq:ls}  describes maximal possible number of
 such $\Gamma_1$-ladder epochs.  Therefore, \eqref{eq:Np-k} implies:
 \be
 \label{eq:Nminus-2-D}
  \lb\bbE_\beta^{\sfh_\epsilon}
N^-_{m} (\sfs\cdot\sfn )^k\rb^{\frac{1}{k}} \leq  c_k \ell_{\sfn} (\sfs )\wedge \lb {\rm e}^{\beta +\Delta_\epsilon} {\rm d}_\sfn (\sfs )
\rb \wedge m
 \ee
 Note that $\ell_\sfn(\sfs)\leq \lceil \frac{\sfs\cdot\sfn }{\epsilon_\sfn }\rceil$.
 Clearly $\sfs\cdot \sfn\leq {\rm d}_\sfn (\sfs )$.
 If, in addition,  $\epsilon_\sfn > {\rm e}^{-\beta}$, then
 $\epsilon_\sfn \geq \epsilon/2 \geq c_5 {\rm e}^{-b_\epsilon}$,
 as it follows from \eqref{eq:uep-cases}, \eqref{eq:G1-New} and
 \eqref{eq:ep-epn}.
 Therefore, $\ell_\sfn (\sfs ) \leq  c_6 {\rm d}_\sfn (\sfs ) {\rm e}^{b_\epsilon }$ whenever
 $\epsilon_\sfn > {\rm e}^{-\beta}$.
 \newline
 On the other hand, if
 $\epsilon_\sfn \leq {\rm e}^{-\beta}$, then $\epsilon \leq 2{\rm
   e}^{-\beta}$ (cf. \eqref{eq:ep-epn})
 and, consequently, ${\rm e}^{-\beta - \Delta_\epsilon }\geq c_8 {\rm
   e}^{-b_\epsilon}$
(recall the definition of $\Delta_\epsilon$ in Section \ref{sub:steps} and use \eqref{eq:aeps}-\eqref{eq:uep-cases}).
 \newline
 In both cases \eqref{eq:Nminus-2-D} implies \eqref{eq:Nminus-2}.
\end{proof}
}

\section{Proof of Proposition~\ref{lem:Plus-bounds} and Proposition~\ref{prop:P-plus-bounds}}
\label{sec:provaprop10}
Recall the definition of $\ell_\sfn$ in \eqref{eq:ls}. For the rest of this
Section set
\be
\label{eq:lu}
\ell_\sfw = \ell_\sfn (\sfw )  \wedge {\rm e}^{b_\epsilon}
= \inf\lbr \ell : \sfw +\ell\sfe_1 \not\in \calH_{+ , \sfn }\rbr
\wedge {\rm e}^{b_\epsilon}.
\ee
Note that in the case of the   {horizontal wall} $\arg (\sfn )  = \frac{\pi}{2}$,  $\ell_\sfw \equiv {\rm e}^{b_\epsilon}$
(for all $\sfw\in\calH_{+ ,\sfn }$).

Proofs of Proposition~\ref{prop:P-plus-bounds} (in Subsection~\ref{sub:proofPx})
 and of the
target bounds on $a_\delta$ and $b_\delta$
 of Proposition~\ref{lem:Plus-bounds} (in Subsection~\ref{sub:proofPlus-bounds})
hinge on careful lower, respectively upper,  bounds on quantities
(see \eqref{eq:RecQuant-P})
$\overline{\bbP}_{\beta , +}^{\sfh_\epsilon , \delta} (\sfu , \sfv )$ and
 $\underline{\bbP}_{\beta , +}^{\sfh_\epsilon , \delta} (\sfw , \sfz )$,
which we proceed to derive in Subsections~\ref{sub:lower} and \ref{sub:upper}.
%\paragraph{\bf Lower bound on  $\bbP_{\beta, + }^{\sfh , \delta} (\sfu, \sfv )$.}
\subsection{Lower bounds.}
\label{sub:lower}
Let $\ell\in\bbN$ and set $\sfu_\ell = \sfu +\ell\sfe_1$.
Then,
\be
\label{eq:step1-lb}
\bbP_{\beta , +}^{\sfh_\epsilon} \lb \sfu , \sfv\rb \geq {\sum_   {\ell=0} ^{\ell_\sfu -1}}
{\rm e}^{-\ell a_\epsilon} \hat\bbP_{\beta , +} \lb 0, \sfv - \sfu_\ell\rb
\geq c_1
%\frac{1 - {\rm e}^{-\ell_\sfu a_\epsilon}}{1 - {\rm e}^{-a_\epsilon}}
\frac{\ell_\sfu {\rm e}^{-\ell_\sfu a_\epsilon}}{\abs{\sfv - \sfu}_1}
\min_{\ell\leq \ell_\sfu}
\bbP_\beta^{\sfh_\epsilon} \lb 0, \sfv - \sfu_\ell\rb .
\ee
Above we considered random walks {which,   first, make} $\ell$ horizontal steps and
then start climbing to $\sfv$, and
relied on $\ell_\sfu \ll\abs{\sfv - \sfu}_1$ (see
\eqref{eq:nu-uv-cond} and recall $b_\epsilon\le \beta, \nu>1$) and \eqref{eq:AD-expBound}.

\paragraph{\bf Curvature of $\pKb$ and lower bound on
$\bbP_\beta^{h_\epsilon} \lb 0, \sfv - \sfu_\ell\rb$.}
We already have a good estimate  \eqref{eq:LB-P}  on $\bbP_\beta^{\sfh_{\epsilon_\ell}}
\lb 0, \sfv - \sfu_\ell\rb$ for  $\sfh_{\epsilon_\ell}\df \nabla\tau_\beta (\sfv -\sfu_\ell )$.
 Here we make changes which are needed to take into account  the
 discrepancy between $\sfh_{\epsilon_\ell}$ and such $\sfh_\epsilon$.

Let us start with some general considerations: Assume that ${\mathbf K} \subset \bbR^2$ is a
convex compact
set with a smooth strictly convex boundary $\partial {\mathbf K}$ which is parametrized
by the direction of the exterior normal $\sfm (\theta ) = \lb \cos \theta , \sin \theta\rb$;
\[
 \partial {\mathbf K} = \lbr \sfh (\theta )\, ;\, \theta \in [0, 2\pi )\rbr .
\]
Then, expanding for $\theta$ in a small neighbourhood of $\theta_0$,
\be
\label{eq:h-exp}
\begin{split}
&\lb \sfh (\theta ) - \sfh (\theta_0 )\rb\cdot \sfm (\theta )  =
\lb \theta -\theta_0\rb \int_0^1 \sfh^\prime (\theta_t )\cdot \sfm (\theta )\dd t = \\
&\qquad
\lb \theta -\theta_0\rb \int_0^1 \sfh^\prime (\theta_t ) \cdot\lb \sfm (\theta )-  \sfm (\theta_t )\rb \dd t
\leq \frac{ (\theta -\theta_0 )^2}{2} \max_{t\in [0,1]} \abs{ \sfh^\prime (\theta_t )}
\end{split}
\ee
where $\theta_t = \theta_0 +t (\theta - \theta_0 )$.  Above we used $\sfh^\prime (\theta_t )
\cdot \sfm (\theta_t )\equiv 0$.

Since the support function
$\tau$ of ${\mathbf K}$ is given by $\tau (\theta ) = \sfh (\theta ) \cdot \sfm (\theta )$,
the radius of curvature  $r (\theta )$ is given by
\be
\label{eq:radius}
r (\theta ) =  \tau^{\prime\prime} (\theta ) +\tau (\theta ) = \sfh^\prime (\theta )
\cdot \sfm^{\perp} (\theta ) =
\abs{
\sfh^{\prime} (\theta )}
%(\theta )\cdot \sfn (\thea ).
\ee
where $\sfm^\perp (\theta ) = (-\sin\theta  , \cos \theta  )$. The curvature of
$\partial {\mathbf K}$ at $\theta$ is $\chi (\theta ) = r (\theta )^{-1}$.
In this notation \eqref{eq:h-exp} reads:
\be
\label{eq:h-exp-curv}
\lb \sfh (\theta ) - \sfh (\theta_0 )\rb\cdot \sfm (\theta )
\leq \frac{(\theta - \theta_0 )^2}{2\min_t \chi (\theta_t )} .
\ee
Next, assume that in a neighbourhood of $\theta_t$ the boundary
$\partial {\mathbf K}$ is given by an implicit equation $ F (\sfh ) =
0 $: then,
\be
\label{eq:implicit}
\chi ( \sfh_t ) = \chi (\theta_t ) =
 \frac{[
 {\rm Hess}F(\sfh_t)\sfm^{\perp} (\theta_t )]
 \cdot\sfm^{\perp} (\theta_t )}{|\nabla F (\sfh_t ) |} .
\ee
\smallskip

Going back to $\bbP_\beta^{h_\epsilon} \lb 0, \sfv - \sfu_\ell\rb$
define $\epsilon_\ell$ via $\sfv - \sfu_\ell = \abs{\sfv - \sfu_\ell}_1\lb
 1-\epsilon_\ell , \epsilon_\ell \rb $ and set $\sfh_{\epsilon_\ell} =
 \nabla\tau_\beta ( 1-\epsilon_\ell , \epsilon_\ell ) = \nabla\tau_\beta (\sfv -\sfu_\ell )$.
 Since $\abs{\sfv - \sfu_\ell}_1\eqvs \abs{\sfv - \sfu }_1$
 and $b_\epsilon \eqvs b_{\epsilon_\ell}$, the local limit result
 \eqref{eq:LB-P} implies:
 \be
 \label{eq:change-h}
 \bbP_\beta^{\sfh_\epsilon} \lb 0, \sfv - \sfu_\ell\rb \eqvs \frac{1}{\sqrt{
 \abs{\sfv - \sfu}_1 {\rm e}^{-b_{\epsilon}}}}
 {\rm e}^{ (\sfh_\epsilon - \sfh_{\epsilon_\ell} )\cdot
 (\sfv - \sfu_\ell )} .
 \ee
{ The extra factor here comes from the difference between the distributions $\bbP_\beta^{\sfh_\epsilon}$ and $\bbP_\beta^{\sfh_{\epsilon_\ell}}$.} Define $
\sfm (\theta_{\epsilon_\ell} )
= (1-\epsilon_\ell , \epsilon_\ell )/
\|{(1-\epsilon_\ell , \epsilon_\ell )}\|_2$. By construction $\sfm (\epsilon_\ell )$
is the unit exterior normal to $\pKb$ at $\sfh_{\epsilon_\ell}$. Since
\[
 (\sfh_{\epsilon} - \sfh_{\epsilon_\ell} )\cdot
 (\sfv - \sfu_\ell ) \eqvs \abs{\sfv - \sfu}_1
 (\sfh_{\epsilon} - \sfh_{\epsilon_\ell} )\cdot\sfm (\theta_{\epsilon_\ell} ) ,
\]
and since $\abs{\epsilon - \epsilon_\ell}\eqvs \abs{\theta_\epsilon -\theta_{\epsilon_\ell}}$
we may rely on \eqref{eq:h-exp-curv}.

In order to derive  lower bounds on $\chi (h_{\epsilon_\ell})$ we shall rely on
\eqref{eq:implicit}: The  boundary $\pKb$ in a neighbourhood of $\sfh_{\epsilon_\ell}$ is parametrized as
\[
 \sfh\in \pKb\, \Leftrightarrow
 \log \bbE_\beta^{\sfh_{\epsilon_\ell}}{\rm e}^{ (\sfh - \sfh_{\epsilon_\ell} )\cdot\sfX (\Gamma )}
 \df F (\sfh )
 =0 .
\]
Note first of all that
\[
 |\nabla F (\sfh_{\epsilon_\ell} ) | =
 \abs{
 \bbE_\beta^{\sfh_{\epsilon_\ell}} \sfX (\Gamma )
} \eqvs 1.
\]
On the other hand, for any $\sfv$,
\be
\label{eq:Hess-bound}
\begin{split}
{\rm Hess} F (\sfh_{\epsilon_\ell} )\sfv\cdot\sfv &=
{\mathbb V}{\rm ar}_\beta^{\sfh_{\epsilon_\ell}} \lb \sfX (\Gamma )\cdot \sfv \rb
= \min_x \bbE_\beta^{\sfh_{\epsilon_\ell}} \lb \sfX (\Gamma )\cdot \sfv -x \rb^2 \\
&\geq \min_x \lbr {\rm e}^{-a_{\epsilon_\ell}}\lb \sfe_1\cdot\sfv - x\rb^2
+ {\rm e}^{-b_{\epsilon_\ell}}\lb \sfe_2\cdot\sfv - x\rb^2 \rbr.
\end{split}
\ee
Substituting $\sfv = \sfm^{\perp} ( \theta_{\epsilon_\ell} ) = \lb -\epsilon_\ell , 1-\epsilon_\ell \rb/
\|{(1-\epsilon_\ell , \epsilon_\ell )}\|_2$, we conclude:
\[
 \chi_\beta (\sfh_{\epsilon_\ell} ) \ge c {\rm e}^{-b_{\epsilon_\ell}}  \eqvs {\rm e}^{-b_{\epsilon}} .
\]
Hence,
\[
 (\sfh_{\epsilon} - \sfh_{\epsilon_\ell} )\cdot
 \sfm (\theta_{\epsilon_\ell} )
 %(1 -\epsilon_\ell , \epsilon_\ell )
 \geq -
 c_2
 {\rm e}^{b_\epsilon}  \lb \epsilon - \epsilon_\ell \rb^2 .
\]
However, $0 \leq \epsilon_\ell - \epsilon \leq c_3 {\rm d}_\sfn (\sfu )/\abs{\sfv - \sfu }_1$.
Recalling that ${\rm d}_\sfn (\sfu ) \leq {\rm e}^{-\nu \beta}\abs{\sfv - \sfu }_1$,
we infer:
\be
\label{eq:curv-bound}
  \bbP_\beta^{\sfh_\epsilon} \lb 0, \sfv - \sfu_\ell\rb \geq c_4
  \frac{{\rm e}^{-c_5{\rm e}^{b_\epsilon - \nu\beta} {\rm d}_\sfn (\sfu  )}}{
  \sqrt{\abs{\sfv - \sfu}_1 {\rm e}^{-b_\epsilon}}} .
\ee
\paragraph{\bf Lower bound on  $\bar\bbP_{\beta, + }^{\sfh_\epsilon , \delta} (\sfu, \sfv )$.}
Since $\nu >1$, ${\rm e}^{b_\epsilon - \nu\beta} \ll \delta\beta $ for all $\beta$
sufficiently large.
Putting things together we derive from \eqref{eq:step1-lb} and \eqref{eq:curv-bound}:
\be
\label{eq:target-lb-uv}
\bar\bbP_{\beta , +}^{\sfh_\epsilon , \delta} \lb \sfu , \sfv\rb
\geq
c_6\frac{{\rm e}^{\beta\delta\lb \frac{ {\rm d}_\sfn (\sfu )}{2} +{\rm d}_\sfn (\sfv )\rb} \ell_\sfu {\rm e}^{-\ell_\sfu a_\epsilon}
}{\abs{\sfv - \sfu}_1 \sqrt{\abs{\sfv - \sfu}_1 {\rm e}^{-b_\epsilon}}}
\geq
c_7 \frac{{\rm e}^{\beta\delta\lb \frac{ {\rm d}_\sfn (\sfu )}{2} +
\frac{
{\rm d}_\sfn (\sfv )}
{2}
\rb
}
\ell_\sfu
}{\abs{\sfv - \sfu}_1 \sqrt{\abs{\sfv - \sfu}_1 {\rm e}^{-b_\epsilon}}} .
\ee
In the last inequality we used  {$a_\epsilon \ell_\sfu\leq a_\epsilon {\rm e}^{b_\epsilon}\leq c_8$,} as
it follows from \eqref{eq:aeps} and \eqref{eq:uep-cases}.
\begin{remark}
 \label{rem-lb-pi}
 Note that if $\arg (\sfn ) = \frac{\pi}{2}$, then
 $\ell_\sfu = {\rm e}^{b_\epsilon}$. Consequently in the latter case:
 \be
\label{eq:target-lb-uv-pi}
\bar\bbP_{\beta , +}^{\sfh_\epsilon , \delta} \lb \sfu , \sfv\rb
\geq
c_7 \frac{{\rm e}^{\beta\delta\lb \frac{ {\rm d}_\sfn (\sfu )}{2} +
\frac{
{\rm d}_\sfn (\sfv )}
{2}
\rb
}
}{ \lb \sqrt{\abs{\sfv - \sfu}_1 {\rm e}^{-b_\epsilon}} \, \rb^3} .
\ee
\end{remark}

\subsection{Upper bounds.}
\label{sub:upper}
Upper bounds are, naturally, more involved. We must explore all the terms in
\eqref{eq:Ruv-decomp}. In doing so we shall rely on \eqref{eq:AD-expBound-ub},
\eqref{eq:Holder-AD-OZ}
and the claims of Lemma~\ref{lem:Nplus} and Lemma~\ref{lem:Nminus}.

By \eqref{eq:Ruv-decomp}
one can fix  $c_1$, such that:
\be
\label{eq:two-terms-ub}
\begin{split}
&\bbP_{\beta ,+}^{\sfh_\epsilon} \lb
%\frR^{+ ,\sfn} (
\sfw , \sfz
%)
\rb =
%\sum_\ell
%\bbP_\beta^\sfh \lb \frR_\ell^{+ ,\sfn} (\sfw , \sfz )\rb =
\sum_{{\sfy\cdot \sfn \geq 0}} \hat \bbP_{\beta , +}^{
\sfh_\epsilon
%}
}
\lb \sfw - \sfy , 0
\rb
\bbP_{\beta ,+}^{\sfh_\epsilon} \lb
%\frR^{+ ,\sfn} (
0, \sfz -\sfy
% )
\rb \\
&\leq
\sum_{\abs{\sfy - \sfw}_1\leq c_1\abs{\sfz - \sfw}_1}
\hat \bbP_{\beta , +}^{
\sfh_\epsilon}
\lb \sfw - \sfy , 0
\rb\max_{2\abs{\sfz - \sfy}_1\geq \abs{\sfz - \sfw}_1}
\bbP_{\beta ,+}^{\sfh_\epsilon} \lb 0, \sfz -\sfy
% )
\rb\\
&+
\max_{2\abs{\sfy- \sfw}_1\geq \abs{\sfz - \sfw}_1}
\hat \bbP_{\beta ,+}^{\sfh_\epsilon} \lb \sfw -\sfy , 0
% )
\rb
\sum_{\abs{\sfy - \sfz}_1\leq c_1\abs{\sfz - \sfw}_1}
\bbP_{\beta , +}^{
\sfh_\epsilon}
\lb 0, \sfz - \sfy
\rb .
\end{split}
\ee
{
Note that {the following two functions of the effective random walk $S_{\ell}$ coincide:}
\[
 \sum_{\sfy\cdot \sfn\geq 0} \1_{\lbr \frA (\sfw , \sfy )\cap \calL^-_\sfy\rbr} = N^- ( \sfw \cdot\sfn ) ,
\]
 {see \eqref{eq:L-heights}.} As in \eqref{eq:AD-expBound-ub} we may ignore effective trajectories from  $\frA (\sfw , \sfy )$
with more than $c \abs{\sfy -\sfw}$ steps. Hence,
}
 {for some $c_3$}
\be
\label{eq:ENm-bound}
\sum_{\abs{\sfy - \sfw}_1\leq c_1\abs{\sfz - \sfw}_1} \!\!\!\!\!
\hat \bbP_{\beta , +}^{
\sfh_\epsilon}
\lb \sfw - \sfy , 0
\rb
= \!\!\!\!\! \sum_{\abs{\sfy - \sfw}_1\leq c_1\abs{\sfz - \sfw}_1}
\!\!\!\!\!
\bbP_{\beta }^{
\sfh_\epsilon}
\lb \frA (\sfw , \sfy ); \calL^-_\sfy \rb
\leq c_2 \bbE_\beta^{\sfh_\epsilon} \lb N^-_{c_3\abs{\sfz - \sfw}_1}(
\sfw \cdot\sfn )\rb .
\ee
Similarly,
\be
\label{eq:ENp-bound}
\sum_{\abs{\sfz - \sfy}_1\leq c_1\abs{\sfz - \sfw}_1} \!\!\!\!\!
\bbP_{\beta , +}^{
\sfh_\epsilon}
\lb 0, \sfz- \sfy
\rb \leq c_2 \bbE_\beta^{\sfh_\epsilon} \lb N^+_{c_3\abs{\sfz - \sfw}_1}
(\sfz \cdot\sfn )\rb .
\ee
The right hand sides of \eqref{eq:ENm-bound} and \eqref{eq:ENp-bound} are
controlled by Lemma~\ref{lem:Nminus} and, respectively, by Lemma~\ref{lem:Nplus}.
{So what remains is the upper bounds on $\max$-terms in \eqref{eq:two-terms-ub}.}
\paragraph{\bf Upper bounds on $\bbP_{\beta , +}^{\sfh_\epsilon}\lb 0, \sfs \rb$ and
$\hat \bbP_{\beta , +}^{\sfh_\epsilon}\lb \sfs ,0 \rb$.}
Consider the first of \eqref{eq:AD-expBound-ub}.
Set $m = m (\sfs ) = c\abs{\sfs}_1$
and define:
\[
 \calL^+_{\sft , m} = \lbr \exists i~:~\tau_i\leq m\ {\rm and}\  \sfR_{\tau_i} = \sft\rbr .
\]
In this way $\calL_\sft^+$ defined in \eqref{eq:L-heights} is recorded
as $\calL_\sft^+ = \calL^+_{\sft , \infty}$.
%On the other hand,
{  Also,} the number of ladder heights is recorded as $N^+_m (\sfs \cdot \sfn ) =
\sum_{0\leq \sft\cdot\sfn \leq \sfs\cdot\sfn}\1_{\calL^+_{\sft , m}}$.
Then, \eqref{eq:AD-expBound-ub} could be recorded as:
{
\be
\label{eq:decomp-t}
\bbP_{\beta , +}^{\sfh_\epsilon}\lb 0, \sfs \rb \leq \frac{c_1}{\abs{\sfs}_1}
\sum_{0\leq \sft\cdot\sfn \leq \sfs\cdot\sfn}
\bbP_{\beta }^{\sfh_\epsilon}\lb \calL^+_{\sft , m}\rb
\bbP_{\beta }^{\sfh_\epsilon} (\sft , \sfs ).
\ee
}
{
In view of Proposition~\ref{prop:LB-P} we can rely on the following 
large deviation upper bound: There exists $c_4 $ such that
\be
\label{eq:LB-P-ub}
\bbP_\beta^{\sfh_\epsilon} \lb \sfs, \sft \rb \leq
\frac{c_4}{\sqrt{{\rm e}^{-b_\epsilon  }|\sft-\sfs |_1}\vee 1} ,
\ee
uniformly in $\arg (\sfn )\in \left[ \frac{\pi}{2}, \frac{3\pi}{4}\right]$,
$\beta$ large, $\epsilon$ (in $\sfv - \sfu \df  \abs{\sfv - \sfu}_1 (1-\epsilon , \epsilon )$)
satisfying  \eqref{eq:nu-uv-cond} and $\sft -\sfs\in\calY$.
}
{
The bound is sharp for $\sft -\sfs$ pointing
in the (average under $\bbP_\beta^{\sfh_\epsilon}$) direction $\sfv - \sfu$ or close to it.
For other directions it is a crude large deviation bound
}  {(compare with the discussion around the relation \eqref{eq:change-h}).}
We shall split the sum in \eqref{eq:decomp-t} according to the values of $|\sft|_1$:
\smallskip

\noindent
(i) $|\sft|_1 \leq \frac{|\sfs|_1}{2}$. In this case,
\[
 \bbP_{\beta }^{\sfh_\epsilon}  (0, \sfs - \sft ) \leq
\frac{c_6 }{\sqrt{{\rm e}^{-b_\epsilon   }|\sfs |_1\vee 1}} ,
\]
as it follows from \eqref{eq:LB-P-ub}.  On the other hand,
{
\[
 \sum_{\sft:0\leq \sft\cdot \sfn \leq \sfs\cdot \sfn}  \bbP_{\beta }^{\sfh_\epsilon}\lb \calL^+_{\sft , m}\rb =
\bbE_\beta^{\sfh_\epsilon} N^+_m (\sfs\cdot \sfn ).
%{\rm e}^{b_\sfh (\beta )}(1+\dd_\sfn (\sfs )) .
\]
}
Therefore, the total contribution of $|\sft|_1 \leq \frac{|\sfs|_1}{2}$ to
the right hand side of \eqref{eq:decomp-t} is bounded above by
\be
\label{eq:t-bound1}
 c_7 \frac{ \bbE_\beta^{\sfh_\epsilon} N^+_{c\abs{\sfs}_1} (\sfs\cdot \sfn )} {|\sfs |_1 \sqrt{{\rm e}^{-b_\epsilon  }|\sfs |_1\vee 1}} .
\ee
\smallskip

\noindent
(ii) $|\sft|_1 >  \frac{|\sfs|_1}{2}$. Since
$\bbP_{\beta }^{\sfh_\epsilon}\lb \calL^+_{\sft , m}\rb \leq
\bbP_{\beta }^{\sfh_\epsilon}\lb \calL^+_{\sft} \rb =
\bbP_{\beta , +}^{\sfh_\epsilon}\lb 0, \sft \rb$, in this case the first
of \eqref{eq:AD-expBound-ub} implies:
{
\[
 \bbP_{\beta }^{\sfh_\epsilon}\lb \calL^+_{\sft , m}\rb \leq
 \frac{c_8}{\abs{\sfs}_1} \bbE_\beta^{\sfh_\epsilon}\lb \frA  (0, \sft );
 N^+_m (\sft\cdot \sfn )\rb.
\]
}
By \eqref{eq:Holder-AD-OZ} and \eqref{eq:LB-P-ub} the latter is bounded above
by
\[
 \frac{c_9}{\abs{\sfs}_1}\lb
 \frac{1}{\sqrt{{\rm e}^{-b_\epsilon  }|\sfs |_1}\vee 1}\rb^{\frac12}
 \lb
 \bbE_\beta^{\sfh_\epsilon} N^+_{m} (\sfs\cdot \sfn )^2\rb^{\frac{1}{2}} .
\]
{Since the point $\sft$  satisfies both $|\sft|_1 > \frac{|\sfs|_1}{2}$ and
$0\leq \sft\cdot \sfn \leq \sfs\cdot \sfn$, we have  that
the distance $|\sft - \sfs|_1 \leq 2 |\sfs|_1$.}   Therefore,
\eqref{eq:LB-P-ub} implies: {\[
\sumtwo{0\leq \sft\cdot \sfn \leq \sfs\cdot \sfn }{2|\sft|_1 >  |\sfs|_1}
\bbP_{\beta }^{\sfh_\epsilon} (\sft, \sfs  )
\leq
c_{10}
\sum_{n=1}^{2|\sfs|_1} \frac{\dd_\sfn (\sfs )}{\sqrt{n{\rm e}^{-b_\epsilon }}\vee 1}
\leq c_{11} \dd_\sfn (\sfs )\sqrt{\abs{\sfs}_1}\sqrt{\abs{\sfs}_1\wedge {\rm e}^{b_\epsilon}}.
\]
}
Altogether,
the total contribution of $|\sft|_1 >  \frac{|\sfs|_1}{2}$ to
the right hand side of \eqref{eq:decomp-t} is bounded above by
\be
\label{eq:t-bound2}
 c_{12}
 \frac{
 \dd_\sfn (\sfs )
 \sqrt{\abs{\sfs}_1\wedge {\rm e}^{b_\epsilon}}}
 {\abs{\sfs}_1^{3/2}}
 \lb
 \frac{1}{\sqrt{{\rm e}^{-b_\epsilon  }|\sfs |_1}\vee 1}\rb^{\frac12}
 \lb
 \bbE_\beta^\sfh N^+_{m} (\sfs\cdot \sfn )^2\rb^{\frac{1}{2}} .
 \ee
In view of \eqref{eq:t-bound1} and \eqref{eq:t-bound2} we have proved:
\begin{lemma}
\label{lem:UB-Pplus-s}
 The following upper bound holds uniformly in
%$\sfh\in \pKb^+$,
${\rm arg} (\sfn )\in [\frac{\pi}{2}, \frac{3\pi}{4}]$,
$\epsilon$
satisfying  \eqref{eq:nu-uv-cond},
$\sfs\in \calY$
 and $\beta$ large:
% ${\rm arg} (\sfs ) \in [0, \frac{2\pi}{5}]$
%and  $|\sfs|_1 {\rm e}^{-b_\sfh }\geq \frac{1}{2}$:
\be
\label{eq:UB-Pplus-s}
\bbP_{\beta , +}^{\sfh_\epsilon} \lb  0,\sfs \rb
\leq
c_{13}
\frac{  \dd_\sfn (\sfs ) \lb
 \bbE_\beta^{\sfh_\epsilon} N^+_{c\abs{\sfs}_1} (\sfs\cdot \sfn )^2\rb^{\frac{1}{2}}}
{|\sfs|_1
 \sqrt{{\rm e}^{-b_\epsilon  }|\sfs |_1}\vee 1} ,
\ee
\end{lemma}
A completely similar analysis reveals:
\begin{lemma}
\label{lem:UB-hatPplus-s}
 The following upper bound holds uniformly in
%$\sfh\in \pKb^+$,
${\rm arg} (\sfn )\in [\frac{\pi}{2}, \frac{3\pi}{4}]$,
$\epsilon$
satisfying  \eqref{eq:nu-uv-cond},
$\sfs\in \calY$
 and $\beta$ large:
% ${\rm arg} (\sfs ) \in [0, \frac{2\pi}{5}]$
%and  $|\sfs|_1 {\rm e}^{-b_\sfh }\geq \frac{1}{2}$:
\be
\label{eq:UB-hatPplus-s}
\hat \bbP_{\beta , +}^{\sfh_\epsilon} \lb \sfs , 0 \rb
\leq
c_{13}
\frac{  \dd_\sfn (\sfs ) \lb
 \bbE_\beta^{\sfh_\epsilon} N^-_{c\abs{\sfs}_1} (\sfs\cdot \sfn )^2\rb^{\frac{1}{2}}}
{|\sfs|_1
 \sqrt{{\rm e}^{-b_\epsilon  }|\sfs |_1}\vee 1} ,
\ee
\end{lemma}
\paragraph{\bf Upper bound on $\underline{\bbP}_{\beta , +}^{ \sfh_\epsilon , \delta}
(\sfw , \sfz )
$}. Decomposition \eqref{eq:two-terms-ub}, bounds \eqref{eq:ENp-bound} and
\eqref{eq:ENm-bound} together with  Lemma~\ref{lem:UB-Pplus-s} and
Lemma~\ref{lem:UB-hatPplus-s} imply %{that for any $k > 1$}
\be
\label{eq:upper-bound-all}
 \bbP_{\beta , +}^{\sfh_\epsilon} \lb \sfw , \sfz \rb
 \leq c_{14}
\frac{
\dd_\sfn (\sfw ) \lb
 \bbE_\beta^{\sfh_\epsilon} N^-_{c\abs{\sfz- \sfw}_1} (\sfw \cdot\sfn  )^2
 \rb^{\frac{1}{2}}
 \dd_\sfn (\sfz ) \lb
 \bbE_\beta^{\sfh_\epsilon} N^+_{c\abs{\sfz- \sfw}_1} ( \sfz\cdot \sfn  )^2
 \rb^{\frac{1}{2}}
}
{
|\sfz - \sfw|_1
 \sqrt{{\rm e}^{-b_\epsilon  }|\sfz -\sfw  |_1}\vee 1
},
\ee
{since the  expectation $\lb\bbE N^k
 \rb^{\frac{1}{k}}$ increases in $k$.}

There are two cases to consider:

\smallskip

\noindent
{
\opt{1}.
If $\arg{\sfn} = \frac{\pi}{2}$, then by Lemma~\ref{lem:Nplus},
\be
\label{eq:ub-plus-1}
 \lb \bbE_\beta^{\sfh_\epsilon} N^+_{c\abs{\sfz- \sfw}_1} ( \sfz\cdot \sfn  )^2
 \rb^{\frac{1}{2}} \leq c_2 c\, \dd_\sfn (\sfz ) \lb
 {\rm e}^{b_\epsilon}\wedge \abs{\sfz- \sfw}_1\rb
 = c_2 c\,\dd_\sfn (\sfz ) \lb \ell_\sfw \wedge \abs{\sfz- \sfw}_1\rb .
\ee
In the last equality we used that in the case of
the horizontal wall $\ell_\sfw \equiv {\rm e}^{b_\epsilon}$, see the remark right after
\eqref{eq:lu}. On the other hand, by Lemma~\ref{lem:Nminus},
\be
\label{eq:ub-minus-1}
\lb \bbE_\beta^{\sfh_\epsilon} N^-_{c\abs{\sfz- \sfw}_1} (\sfw \cdot\sfn  )^2
 \rb^{\frac{1}{2}} \leq c_2 \dd_\sfn (\sfw ) .
\ee
}
\smallskip

\noindent
{
\opt{2}.
If $\arg{\sfn} > \frac{\pi}{2}$, then by Lemma~\ref{lem:Nplus},
\be
\label{eq:ub-plus-2}
\lb \bbE_\beta^{\sfh_\epsilon} N^+_{c\abs{\sfz- \sfw}_1} ( \sfz\cdot \sfn  )^2
 \rb^{\frac{1}{2}} \leq c_2 \dd_\sfn (\sfz ) .
\ee
On the other hand, by Lemma~\ref{lem:Nminus},
\be
\label{eq:ub-minus-2}
\begin{split}
 \lb \bbE_\beta^{\sfh_\epsilon} N^-_{c\abs{\sfz- \sfw}_1} (\sfw \cdot\sfn  )^2
 \rb^{\frac{1}{2}} &\leq c_2c\,  \dd_\sfn (\sfw ) \lb \ell_\sfn (\sfw )
 \wedge {\rm e}^{b_\epsilon}\wedge \abs{\sfz -\sfw}_1\rb \\
 &\stackrel{\eqref{eq:lu}}{=} c_2
 c\, \dd_\sfn (\sfw ) \lb \ell_\sfw\wedge \abs{\sfz -\sfw}_1\rb  .
 \end{split}
\ee
}
\smallskip

\noindent
{
Since
$
 {\rm e}^{-\frac{\delta \beta\dd_\sfn (\cdot )}{2} }\dd_\sfn (\cdot )^2
$
is uniformly bounded, a substitution of \eqref{eq:ub-plus-1} and
\eqref{eq:ub-minus-1} in the case of  $\arg{\sfn} =\frac{\pi}{2}$
(respectively of
\eqref{eq:ub-plus-2} and \eqref{eq:ub-minus-2} in the case of $\arg{\sfn} >\frac{\pi}{2}$
)
into \eqref{eq:upper-bound-all} implies:
}
\be
\label{eq:target-ub-wz}
\underline{\bbP}_{\beta , +}^{ \sfh_\epsilon , \delta}
(\sfw , \sfz ) \leq
c_{15} \frac{
{\rm e}^{-\frac{\delta \beta {\rm d}_\sfn (\sfw )}{2}}
 \ell_\sfw \wedge \abs{\sfz - \sfw}_1
 {\rm e}^{-\frac{\delta\beta {\rm d}_\sfn (\sfz )}{2}}
}
{
|\sfz - \sfw|_1
 \sqrt{{\rm e}^{-b_\epsilon  }|\sfz -\sfw  |_1}\vee 1
},
\ee
uniformly in
%$\sfh\in \pKb^+$,
${\rm arg} (\sfn )\in [\frac{\pi}{2}, \frac{3\pi}{4}]$,
$\epsilon$
satisfying  \eqref{eq:nu-uv-cond},
$\sfz , \sfw \in \calH_{ + , \sfn}$
 and $\beta$ large.

\begin{remark}
\label{rem:lw-n-pi}
Note that $\ell_\sfw \equiv {\rm e}^{b_\epsilon}$
if $\arg (\sfn ) = \frac{\pi}{2}$. Hence in the latter case:
\be
\label{eq:target-ub-wz-pi}
\underline{\bbP}_{\beta , +}^{ \sfh_\epsilon , \delta}
(\sfw , \sfz ) \leq
c_{15} \frac{
{\rm e}^{-\frac{\delta \beta {\rm d}_\sfn (\sfw )}{2}}
 {\rm e}^{-\frac{\delta\beta {\rm d}_\sfn (\sfz )}{2}}
}
{
 \lb \sqrt{ {\rm e}^{-b_\epsilon  }|\sfz -\sfw  |_1} \vee 1\rb^3
}.
\ee
\end{remark}
\subsection{Proof of Proposition~\ref{prop:P-plus-bounds}.}
\label{sub:proofPx}
In view of \eqref{eq:target-lb-uv} and \eqref{eq:target-ub-wz}, and
since we permit dependence $c_1 = c_1 (\beta )$ in \eqref{eq:ImplProp10},
the inequality \eqref{eq:ImplProp10} is, as it is stated, a rather crude bound.
First of all
we
can assume that ${\rm e}^{-\sfh}\abs{\sfx}_1 \geq 2$. Then,
by \eqref{eq:target-lb-uv} (taking $\sfu = 0$ and $\sfv = \sfx$),
\[
 \bbP_{\beta , +}^{\sfh} (0, \sfx )\geq \frac{c_7 {\rm e}^{-\beta\delta}}{\abs{\sfx}_1
 \sqrt{\abs{\sfx}_1 {\rm e}^{-\sfh}}} .
\]
This already rules of the exponential term on the left hand side of \eqref{eq:ImplProp10}.
Also we may restrict attention to  $\abs{\sfu - \sfv}\geq \frac{1}{2}\abs{\sfx}_1$. In this
case
\eqref{eq:target-ub-wz},
\[
 \bbP_{\beta , +}^{\sfh} (\sfu , \sfv ) \leq
 \frac{\sqrt{8} c_{15} {\rm e}^{b_\sfh} {\rm e}^{\beta\delta (\dd_\sfn (\sfu ) +\dd_\sfn (\sfv ))}}
 {\abs{\sfx}_1
 \sqrt{\abs{\sfx}_1 {\rm e}^{-\sfh}}} .
\]
as it follows fro \eqref{eq:target-ub-wz} and \eqref{eq:lu} (which implies $\ell_\sfu\leq
{\rm e}^{b_\sfh}$).
It remains to recall
\eqref{eq:abs-vs-d}, and \eqref{eq:ImplProp10} follows. Indeed, by the above
\[
{\rm e}^{-\beta\delta (\abs{\sfu}_1 +\abs{\sfx-\sfv}_1)} \bbP_{\beta , +}^{\sfh} (\sfu , \sfv )
\leq \frac{\sqrt{8} c_{15} {\rm e}^{3\beta\delta +\sfh}}{c_7} \bbP_{\beta , +}^{\sfh} (0, \sfx )
\df c_1 (\beta )\bbP_{\beta , +}^{\sfh} (0, \sfx ).
\]
\subsection{\bf Upper bounds on $a_\delta$ and $b_\delta$:
Proof of Proposition~\ref{lem:Plus-bounds}.}
\label{sub:proofPlus-bounds}
Recall that we assume that $(\sfu , \sfv )\in \calA$. In particular
 \eqref{eq:nu-uv-cond} holds and $\abs{\sfv -\sfu }_1\geq {\rm e}^{\nu\beta}$.
 In the sequel we shall rely on the  decay estimate \eqref{eq:K-kernel}
  on the kernel $\calK_\beta$.
 Let us elaborate on Remark~\ref{rem:Rstrip}: If $\phi$ is a positive
 function on $\bbN$, then
 \be
 \label{eq:strip}
 \sum_{\sfw, \sfz} \phi\lb \abs{\sfw - \sfu}_1\rb \calK_\beta (\sfw , \sfz ), \
 \sum_{\sfw, \sfz} \calK_\beta (\sfw , \sfz ) \phi\lb \abs{\sfv - \sfz}_1\rb
 \leq c_1 R \sum_n \phi (n ) .
 \ee
\noindent
\paragraph{\bf Upper bound on $a_\delta$.}
Consider the sum on the right hand side of {\eqref{eq:a-delta}.} By  \eqref{eq:Psi-ub}
 it is bounded above by
\be
\label{eq:a-delta-1}
c_2 {\rm e}^{-2\chip\beta}
%\sup_{\lb\sfu , \sfv \rb\in\calA}
\sum_{\sfw ,\sfz} \frac{
\underline{\bbP}_{\beta , +}^{ \sfh_\epsilon , \delta}
(\sfu , \sfw )
{\mathcal K}_\beta (\sfw , \sfz )
\underline{\bbP}_{\beta , +}^{ \sfh_\epsilon , \delta}
(\sfz , \sfv )
}{
\overline{\bbP}_{\beta , +}^{ \sfh_\epsilon , \delta}
(\sfu, \sfv )
} .
\ee
By \eqref{eq:Psi-ub}  {and \eqref{eq:K-kernel}} we may restrict
attention to $\max\lbr \abs{\sfw-\sfu}_1 ,
\abs{\sfv - \sfz}_1\rbr  \geq \frac{\abs{\sfv - \sfu}_1}{3}$.
\smallskip

\noindent
(i) If $3\abs{\sfv - \sfz}_1 \geq \abs{\sfv - \sfu}_1$, then (recall $\ell_\sfv\le  e^{b_\epsilon}$)
\[
 \frac{
 \underline{\bbP}_{\beta , +}^{ \sfh_\epsilon , \delta}
(\sfz , \sfv )
}
{
\overline{\bbP}_{\beta , +}^{ \sfh_\epsilon , \delta}
(\sfu, \sfv )
}
\leq \frac{{\rm e}^{b_\epsilon}}{\ell_\sfu},
\]
as it follows from \eqref{eq:target-lb-uv} and \eqref{eq:target-ub-wz}
(recall also \eqref{eq:target-lb-uv-pi} and \eqref{eq:target-ub-wz-pi} in
the special  case of $\arg{\sfn} = \frac{\pi}{2}$) and $\abs{\sfv - \sfu}_1 \geq
{\rm e}^{\nu\beta}$. On the other hand, \eqref{eq:target-ub-wz} and
\eqref{eq:strip} imply:
\[
 \sum_{\sfw , \sfz}
 \underline{\bbP}_{\beta , +}^{ \sfh_\epsilon , \delta}
(\sfu , \sfw ) \calK_\beta (\sfw , \sfz ) \leq
c_3 R\lbr
\sum_1^{{\rm e}^{b_\epsilon}} \frac{\ell_\sfu}{n} +
\sum_{{\rm e}^{b_\epsilon}}^\infty \frac{\ell_\sfu
\sqrt{{\rm e}^{b_\epsilon}}}{n^{3/2}}\rbr
\leq c_4 R\ell_\sfu b_\epsilon,
\]
which means that the total contribution of (i) to
  \eqref{eq:a-delta-1}  and, as a result, to  \eqref{eq:a-delta} is bounded above
by
\be
\label{eq:bound-a1}
c_5 R {\rm e}^{- (2\chip\beta  - b_\epsilon )} b_\epsilon .
\ee
(ii) If $3\abs{\sfw -\sfu}_1 \geq \abs{\sfv - \sfu}_1$, then
\eqref{eq:target-lb-uv} and \eqref{eq:target-ub-wz} imply:
\[
 \frac{
 \underline{\bbP}_{\beta , +}^{ \sfh_\epsilon , \delta}
(\sfu , \sfw )
}
{
\overline{\bbP}_{\beta , +}^{ \sfh_\epsilon , \delta}
(\sfu, \sfv )
}
\leq c_6.
\]
On the other hand,
\eqref{eq:target-ub-wz} and
\eqref{eq:strip} imply:
\[
 \sum_{\sfw , \sfz}
 \calK_\beta (\sfw , \sfz )
\underline{\bbP}_{\beta , +}^{ \sfh_\epsilon , \delta}
(\sfz , \sfv )
\leq
c_7 R\lbr
\sum_1^{{\rm e}^{b_\epsilon}} 1  +
\sum_{{\rm e}^{b_\epsilon}}^\infty \frac{\ell_\sfu {\rm e}^{3 b_\epsilon/2}}{n^{3/2}}\rbr
\leq c_8 R {\rm e}^{b_\epsilon} ,
\]
which means that the total contribution of (ii) to \eqref{eq:a-delta} is bounded above
by
$
c_9 R {\rm e}^{- (2\chip\beta  - b_\epsilon )}  .
$

Altogether we conclude that
\be
\label{eq:3}
a_\delta \leq c_{10} R b_\epsilon \rm e^{- (2\chip\beta  - b_\epsilon )}
\ee
{uniformly}  in $\beta$ large.
Since (by \eqref{eq:ax-bx})
$b_\epsilon \leq \beta$ and since  $\chip > 1/2$, the expression
in \eqref{eq:3} is actually $\smof{1}$ uniformly in $\beta$ large.

\paragraph{\bf Upper bound on $b_\delta$.} Exactly in the same fashion 
{we
derive the following upper bound on $b_\delta$:
 There exists a constant ${c_{11}}$, such that}
\be
\label{eq:b-d-bound-f}
b_\delta \leq c_{11} R^2   b_\epsilon {\rm e}^{- (4\chip \beta - 2b_\epsilon )}  .
\ee
also {uniformly}  in $\beta$ large. Again,
since $b_\epsilon\leq \beta$ and
 since $\chip >1/2$,
 $b_\delta = \smof{1}$ uniformly in $\beta$ large, as it was claimed. \qed

\appendix

\label{sec:Appa}

\section{ A correction to \cite{DKS}.}

The first motivation for one of the authors of the present paper (S.S.) was to
correct the mistake in the Wulff construction book \cite{DKS}. Namely, one
statement in that book -- the Theorem 4.16, dealing with spatial sensitivity of the surface tension
-- is not correct; more precisely, the upper bound statement 4.19 is
erroneous. This mistake was uncovered by the authors of the paper \cite{CLMST1}. But the reader of the
present paper should not think that some forty pages have to be added to
\cite{DKS} in order to correct it, because a weaker version of the Theorem
4.16 is quite sufficient to get all other results of \cite{DKS}. We will give
here the formulation of this weaker statement, in the notations of the book
\cite{DKS}:

\begin{theorem}
\label{SSS} Theorem 4.16 of \cite{DKS} holds for $\bar{V}_{N}=U_{N,d,\varkappa
},$ with $d<\bar{d}/2,$ i.e. when the change from the interaction $\Phi$ to
$\tilde{\Phi}$ happens far away from the range $\bar{V}$ of the random contour.
\end{theorem}

In terms of the present paper, the meaning of the above statement is that the
surface tension does not change if the interaction is perturbed far from the
range of the contour. For example, if we compute the surface tension over
the polymers $\Gamma_{N}$ fitting a strip $\left\{  x,y:\left\vert y-\varkappa
x\right\vert <\frac{1}{2}N^{\alpha}\right\}  ,$ but perturb the interactions
$\Phi_{\beta}\left(  \mathcal{C}\right)  $ only if $\mathcal{C}$ does not fit
the wider strip $\left\{  x,y:\left\vert y-\varkappa x\right\vert <N^{\alpha
}\right\}  ,$ then the claim that the surface tension is unaffected by the
perturbation holds true, and is easy to prove. For a motivated reader of
\cite{DKS}, who reached  Theorem 4.16 of it, the proof of the above statement
 and the check that it is sufficient for all the needs of the book,
will be an easy exercise.

But the problem of spatial sensitivity of the surface tension in its stronger
form of Theorem \ref{NoPin} is important in various applications and is of independent interest.

\section*{Acknowledgments}

F. L. T. is very grateful to P. Caputo and to F. Martinelli for
countless discussions on these issues.

\end{document}